\documentclass[11pt,a4paper,utf8x]{article}
\usepackage{color}
\usepackage{fancyhdr}
\usepackage{tikz}
\usepackage[vmargin=90pt, hmargin=45pt]{geometry}
\usetikzlibrary{decorations.pathmorphing}
\usepackage[T1]{fontenc}
\usepackage{ucs}
\usepackage[utf8x]{inputenc}
\usepackage{amsfonts}
\usepackage{listings}
\usepackage{amsmath}
\usepackage[mathscr]{euscript}
\usepackage{amssymb}
\usepackage{amsfonts}
\usepackage{dsfont}
\usepackage{amsthm}
\usepackage{fancybox}
\usepackage{enumitem}
\usepackage{dsfont}
\usepackage{algpseudocode}
\usepackage{algorithm}
\usepackage{graphicx}
\usepackage{caption}
\usepackage{subcaption}
\usepackage{float}
\usepackage{authblk}
\usepackage{relsize}
\usepackage{setspace}
\usepackage{comment}

\DeclareMathOperator{\LL}{L}
\DeclareMathOperator{\Tr}{Tr}

\newcommand*{\DomA}{\mathcal{D}(A)}
\makeatletter
\newcommand*{\inlineequation}[2][]{%
  \begingroup
    \refstepcounter{equation}%
    \ifx\\#1\\%
    \else
      \label{#1}%
    \fi

 \relpenalty=100000000%
 \binoppenalty=10000000 %

    \ensuremath{%
      #2%
    }%
    ~\@eqnnum
  \endgroup
}
\makeatother

\newenvironment{alphafootnotes}
  {\par\edef\savedfootnotenumber{\number\value{footnote}}
   
   \setcounter{footnote}{0}}
  {\par\setcounter{footnote}{\savedfootnotenumber}}

\newtheorem{mytheo}{Theorem}[section]
\newtheorem{myprop}{Proposition}[section]
\newtheorem{lemme}{Lemma}[section]

\newtheorem{remark}{Remark}[section]

\date{}

\begin{document}

\title{Discretization of the Ergodic Functional Central Limit Theorem}
\author[1]{Gilles Pag\`es}
\author[2]{Cl\'ement Rey$^{\ast,}$}
\affil[1]{LPSM, Sorbonne Universit\'e, 4 Place Jussieu, 75005 Paris, France} 
\affil[2]{CMAP, \'Ecole Polytechnique, Institut Polytechnique de Paris, Route de Saclay, 91120 Palaiseau, France}
\maketitle

\begin{alphafootnotes}
\footnote{e-mails : gilles.pages@sorbonne-universite.fr, clement.rey@polytechnique.edu $^{\ast}$. }
\end{alphafootnotes}

\abstract{ 

In this paper, we study the discretization of the ergodic Functional Central Limit Theorem (CLT) established by Bhattacharya (see \cite{Bhattacharya_1982}) which states the following: Given a stationary and ergodic Markov process $(X_t)_{t \geqslant 0}$ with unique invariant measure $\nu$ and infinitesimal generator $A$, then, for every smooth enough function $f$, $(n^{1/2} \frac{1}{n}\int_0^{nt} Af(X_s)ds)_{t \geqslant 0}$ converges in distribution towards the distribution of the process $(\sqrt{-2 \langle f, Af \rangle_{\nu}} W_{t})_{t \geqslant 0}$ with $(W_{t})_{t \geqslant 0}$ a Wiener process. In particular, we consider the marginal distribution at fixed $t=1$, and we show that when $\int_0^{n} Af(X_s)ds$ is replaced by a well chosen discretization of the time integral with order $q$ ($e.g.$ Riemann discretization in the case $q=1$),  then the CLT still holds but with rate $n^{q/(2q+1)}$ instead of $n^{1/2}$.  Moreover, our results remain valid when $(X_t)_{t \geqslant 0}$ is replaced by a $q$-weak order approximation (not necessarily stationary). This paper presents both the discretization method of order $q$ for the time integral and the $q$-order ergodic CLT we derive from them. We finally propose applications concerning the first order CLT for the approximation of Markov Brownian diffusion stationary regimes with Euler scheme (where we recover existing results from the literature) and the second order CLT for the approximation of Brownian diffusion stationary regimes using Talay's scheme \cite{Talay_1990} of weak order two.\\


Data Availability Statement: Data sharing not applicable to this article as no datasets were generated or analyzed during the current study.
\section{Introduction}

In this paper, we design a recursive algorithm which aims to approximate the invariant distribution (denoted $\nu$) of a Feller process $(X_t)_{t \geqslant 0}$.  Moreover, we establish rates of convergence for our approximation. In particular, we prove a discretized version of the Functional Central Limit Theorem (CLT) presented in \cite{Bhattacharya_1982} and which states the following: \\
Let $(X_{t})_{t \geqslant 0}$ be a progressively measurable Markov stationary ergodic process with initial and invariant distribution $\nu$ and infinitesimal generator $A$ with domain $\DomA$ (see Section \ref{sec:Construction of the random measures} for definition). Then, for every $f \in \DomA$,
$(n^{1/2}\frac{1}{n} \int_0^{t n} Af(X_s)ds)_{t \geqslant 0}$ converges in distribution, as $n$ tends to infinity, towards the distribution of $(\sqrt{-2 \langle f, Af \rangle_{\nu}} W_{t})_{t \geqslant 0}$ with $(W_{t})_{t \geqslant 0}$ a Wiener process. \\

In this  work, we are interested in proving a version of this result when considering the marginal asymptotic distribution of $n^{1/2}\frac{1}{n} \int_0^{t n} Af(X_s)ds$ for fixed $t >0$.  The main difference of our approach compared to \cite{Bhattacharya_1982}, is that we consider discrete time approximations of the time integral $\int_0^{t n} Af(X_s)ds$ and make possible to replace $(X_t)_{t \geqslant 0}$ by a weak order approximation (in a sense made precise in Section \ref{Rate of convergence - A general approach}) with arbitrary initial condition. We then establish a CLT (see Theorem \ref{th:conv_gnl_ordre_q}) for our approximation of the time integral. However, the rate of convergence is altered by the discretization. In particular, if we use a $q$-weak order approximation for $(X_t)_{t \geqslant 0}$, $q \in \mathbb{N}^{\ast}$, we propose an adapted discretization of the time integral such that the CLT is satisfied with order $n^{1/2}$ replaced by $n^{q/(2q+1)}$. \\

In order to build our approximation of $\nu$, we consider random weighted empirical measures built using a recursive algorithm introduced in \cite{Pages_Rey_2020} and inspired by \cite{Lamberton_Pages_2002}. Let us be more specific about the motivation of this algorithm. \\

Invariant distributions are crucial in the study of the long term behavior of stochastic differential systems (see \cite{Hasminskii_1980} and \cite{Ethier_Kurtz_1986} for an overview of the subject) and their computation has already been widely explored in the literature. In \cite{Soize_1994}, explicit exact expressions of the invariant density distribution for some solutions of Stochastic Differential Equations are given. \\

However, in many cases, there is no explicit formula for $\nu$ and other approaches must be developed. A first method consists in studying the convergence towards $\nu$ of the semigroup $(P_t)_{t \geqslant 0}$ ($i.e.$ $\mathbb{E}[f(X_{t}]$) of the Markov process $(X_t)_ {t\geqslant 0}$ as $t$ tends to infinity. This is done $e.g.$ in \cite{Ganidis_Roynette_Simonot_1999} for the total variation topology. As soon as $X_{T}$ can be simulated, for $T$ large enough, we can design a Monte Carlo method to estimate $P_{T}$. Remark that, in addition to the convergence error of $P_{T} \underset{T \to \infty }{\to}\nu$, it gives rise to a second term in the error analysis due to the Monte Carlo error for the computation of $P_{T}$.\\

Unfortunately, most of the time, $(X_t)_{t \geqslant 0}$ cannot be simulated at a reasonable cost. A solution is then to replace $(X_t)_{t \geqslant 0}$ by a simulable approximation $(\overline{X}^{\gamma}_{\Gamma_n})_{n\in \mathbb{N}}$, built with transition functions $( \mathscr{Q}_{\gamma_n})_{n \in\mathbb{N}^{\ast}}$ (given a step sequence $(\gamma_n)_{n \in \mathbb{N}}$, $\Gamma_0=0$ and $\Gamma_n=\gamma_1+\ldots+\gamma_n$).  We usually refer to $(\overline{X}^{\gamma}_{\Gamma_n})_{n\in \mathbb{N}}$ as a numerical scheme of $(X_t)_{t \geqslant 0}$.  It is then possible to build approximations of $\nu$ using Monte Carlo methods resulting from the weak approximation properties satisfied by $(\overline{X}^{\gamma}_{\Gamma_n})_{n\in \mathbb{N}}$.
For instance, when $\gamma_{n}=\gamma_{1}$, $n \in \mathbb{N}^{\ast}$,  $T \in \{\Gamma_{n},n \in \mathbb{N} \}$, if $\overline{X}^{\gamma_1}_{T}$ weakly converges to $P_{T}$ as $\gamma_{1}$ tends to zero, we can approximate $\nu(f)$ using the Monte Carlo approximation of $\mathbb{E}[f(\overline{X}^{\gamma_1}_{T}))]$ taking $\gamma_{1}$ small enough and $T$ large enough (at least for continuous and bounded $f$).\\

%
%
%

\noindent The Monte Carlo methods mentioned above do not fully benefit from the ergodic feature of $(X_t)_{t \geqslant 0}$. In fact, as investigated in \cite{Talay_1990} for strongly Brownian diffusions, the ergodic (or positive recurrence) property of $(X_t)_{t \geqslant 0}$ is also satisfied by its approximation $(\overline{X}^{\gamma_{1}}_{\Gamma_n})_{n\in \mathbb{N}}$ at least for small enough $\gamma_{1}$. In particular, $(\overline{X}^{\gamma_1}_{\Gamma_n})_{n\in \mathbb{N}}$ has an invariant distribution $\nu^{\gamma_1}$ (supposed to be unique for simplicity) and the sequence of empirical measures 
 \begin{align}
 \label{eq:def_weight_const_emp_meas_intro}
 \nu^{\gamma_1}_n(dx)=\frac{1}{ \Gamma_n} \sum_{k=1}^n \gamma_1 \delta_{\overline{X}^{\gamma_1}_{\Gamma_{k-1}}}(dx), \qquad \Gamma_n = n \gamma_1
 \end{align}
(which can be seen as a discrete version of the time integral $\Gamma_n^{-1} \int_0^{\Gamma_{n}} \delta_{X_s}ds$ studied in \cite{Bhattacharya_1982} with $X$ replaced by $\overline{X}^{\gamma_1}$) almost surely weakly converges to $\nu^{\gamma_1}$. In other words, for every continuous and bounded function $f$,  $\nu^{\gamma_1}_n(f)$ almost surely converges to $\nu^{\gamma_1}(f)$. This last result makes possible to compute by simulation, arbitrarily accurate approximations of $\nu^{\gamma_1}(f)$ using only one simulated path of $(\overline{X}^{\gamma_1}_{\Gamma_n})_{n\in \mathbb{N}}$. It is an ergodic - or Langevin - simulation of $\nu^{\gamma_1}(f)$. At this point, it remains to establish at least that $\nu^{\gamma_1}(f)$ converges to $\nu(f)$ when $\gamma_1$ converges to zero and, if possible, at which rate. In \cite{Talay_1990} this rate was shown to depend closely on the weak order of the numerical scheme. Notice that the rate of convergence of $(\nu^{\gamma_1}_n)_{n \in \mathbb{N}^{\ast}}$ to $\nu^{\gamma_1}$ is not established in \cite{Talay_1990}.\\

%
To take a step further, the intuition of our algorithm is to build a version (\ref{eq:def_weight_const_emp_meas_intro}) such that we avoid the asymptotic analysis between $\nu^{\gamma_1}$ and $\nu$.  Concerning Monte Carlo approaches for Brownian diffusions, it is proved in \cite{Basak_Hu_Wei_1997}, that the discrete time weak approximation Markov process $(\overline{X}^{\gamma}_{\Gamma_n})_{n\in \mathbb{N}}$, with step sequence $\gamma=(\gamma_n)_{n \in \mathbb{N}}$ vanishing to 0, weakly converges towards $\nu$. It is then possible to approximate $\nu(f)$ using the Monte Carlo approximation of $\mathbb{E}[f(\overline{X}^{\gamma}_{\Gamma_{n}})]$ for $n$ large enough.\\
\noindent In \cite{Lamberton_Pages_2002}, the ideas from \cite{Talay_1990}  and \cite{Basak_Hu_Wei_1997} are combined to design a Langevin Euler Monte Carlo recursive algorithm with decreasing step  which $a.s.$ weakly converges to an invariant distribution. This paper treats the case where $ (\overline{X}^{\gamma}_{\Gamma_n})_{n\in \mathbb{N}}$ is an (inhomogeneous) Euler scheme with decreasing step  associated to a strongly mean reverting Brownian diffusion process taking values in $\mathbb{R}^{d}$. The sequence $(\nu^{\gamma}_n)_{n \in \mathbb{N}^{\ast}}$ is defined as the weighted empirical measures of the path of $ (\overline{X}^{\gamma}_{\Gamma_n})_{n\in \mathbb{N}}$ (which is the procedure that is used in every work we mention from now on and which is also the one we use in this paper). In particular, the $a.s.$ weak convergence of
 \begin{align}
 \label{eq:def_weight_emp_meas_intro}
 \nu^{\gamma}_n(dx)=\frac{1}{\Gamma_n} \sum_{k=1}^n \gamma_k \delta_{\overline{X}^{\gamma}_{\Gamma_{k-1}}}(dx), \qquad \Gamma_n=\sum\limits_{k=1}^n \gamma_k,
 \end{align}
towards the (non-empty) set $\mathcal{V}$ of the invariant distributions of the underlying Brownian diffusion is established. Notice also that, this approach does not require that the invariant measure $\nu$ is unique by contrast with the results obtained in \cite{Talay_1990} and \cite{Basak_Hu_Wei_1997} or in \cite{Durmus_Moulines_2015} where the authors study of the total variation convergence for the Euler scheme with decreasing step  of the over-damped Langevin diffusion. Moreover, when the invariant measure $\nu$ is unique, it is proved in \cite{Lamberton_Pages_2002} that $\lim\limits_{n \to + \infty} \nu^{\gamma}_n f=\nu f \; a.s.$ for a class of test functions $f$ that is not simply restricted to continuous and bounded functions but for a larger class, made of continuous functions with polynomial growth.  More specifically, it is shown that, given $p >0$,  $\lim\limits_{n \to + \infty} \nu^{\gamma}_n f=\nu f \; a.s.$ for every function $f$ satisfying $\vert f (x) \vert \leqslant C(1+\vert x \vert^{p})$ for every $x \in \mathbb{R}^{d}$. This last result implies the $a.s.$ convergence for the $p$-Wasserstein distance (this is a consequence of Theorem 6.9 in \cite{Villani_OTOldandNew_2008}).\\

  In the spirit of \cite{Bhattacharya_1982}, a CLT is also established in \cite{Lamberton_Pages_2002} for the empirical measures (\ref{eq:def_weight_emp_meas_intro}) of the Euler scheme with rate $n^{1/3}$. More specificaly, it is shown that, for a well chosen step sequence $(\gamma_n)_{n \in \mathbb{N}^{\ast}}$, when $n$ tends to infinity, $n^{1/3} \nu^{\gamma}_n(A f)$ converges in distribution towards the centered Gaussian distribution with variance $-2 \langle f, Af \rangle_{\nu}$. This whole study is made in a strongly mean reverting setting, and the extension to the weakly mean reverting setting has been realized first in \cite{Panloup_2008_rate}. \\

 Concerning the study of the almost sure convergence, the results established in \cite{Lamberton_Pages_2002} gave rise to many generalizations and extensions. In \cite{Lamberton_Pages_2003}, the initial result is extended to the case of Euler scheme of Brownian diffusions with weakly mean reverting properties. Thereafter, in \cite{Lemaire_thesis_2005}, the class of test functions for which we have $\lim\limits_{n \to + \infty} \nu^{\gamma}_n f =\nu f \; a.s.$ (when the invariant distribution is unique) is extended to include functions with exponential growth. Finally, in \cite{Panloup_2008}, the results concerning the polynomial case are shown to hold for the computation of invariant measures for weakly mean reverting Levy driven diffusion processes. For a more complete overview of the studies concerning (\ref{eq:def_weight_emp_meas_intro}) for the Euler scheme, the reader can also refer to \cite{Pages_2001_ergo}, \cite{Lemaire_2007}, \cite{Panloup_2008_rate}, \cite{Pages_Panloup_2009}, \cite{Pages_Panloup_2012} or \cite{Mei_Yin_2015}.\\
 
 Those results are extended in \cite{Pages_Rey_2020} and generalized to the abstract case where both the Markov transition sequence $( \mathscr{Q}_{\gamma_n})_{n \in\mathbb{N}^{\ast}}$ (and then $ (\overline{X}^{\gamma}_{\Gamma_n})_{n\in \mathbb{N}}$) and the Feller process $(X_t)_{t \geqslant 0}$ are not specified explicitly.  In \cite{Pages_Rey_2020}, abstract properties are developed to prove $a.s.$ weak convergence of (\ref{eq:def_weight_emp_meas_intro}) in this abstract framework.  In particular, it suggests various applications beyond the Euler scheme of Levy processes. See for instance \cite{Pages_Rey_2019_milstein}.  An interest of such an abstract framework is that it can be applied to schemes with higher $q$-weak order than the Euler scheme (which has weak order of convergence $q=1$).  In this paper, we aim to show that this procedure may improve the rate of convergence in the CLT from $n^{1/3}$ to $n^{q/(2q+1)}$.

\medskip
\noindent In particular, we extend the abstract framework introduced in \cite{Pages_Rey_2020} to prove the CLT in the weakly mean reverting setting.  We establish an abstract ergodic $q$-order CLT (see Theorem \ref{th:conv_gnl_ordre_q}) which enables to obtain a discretized version of \cite{Bhattacharya_1982} and recover results from \cite{Lamberton_Pages_2002}, \cite{Lemaire_thesis_2005}, \cite{Panloup_2008_rate} or \cite{Mei_Yin_2015} which are all restricted to the case $q=1$.  The proof of Theorem \ref{th:conv_gnl_ordre_q} relies both on the fact that we deal with a $q$-weak order stochastic approximation $(\overline{X}^{\gamma}_{\Gamma_n})_{n\in \mathbb{N}}$ for $(X_t)_{t \geqslant 0}$ and that we consider a generalization of (\ref{eq:def_weight_emp_meas_intro}), defined by
 \begin{align}
 \label{eq:def_weight_emp_meas_intro_weight}
 \nu^{\eta_q}_n(dx)=\frac{1}{H_n} \sum_{k=1}^n \eta_{q,k} \delta_{\overline{X}^{\gamma}_{\Gamma_{k-1}}}(dx), \qquad H_n=\sum\limits_{k=1}^n \eta_k,
 \end{align}
with $(\eta_{q,n})_{n \in \mathbb{N}^{\ast}}$ a well-chosen weight sequence given in (\ref{def:weight_q_order}). Notice that the weights for $q=1,2$ or $3$ appears as extension of the standards Riemann, Trapezoidal or Simpson's homogeneous approximations of integrals. Up to our knowledge, no second or higher order CLT had been derived in any situation so far in the literature. However, acceleration techniques inherited from multilevel Monte Carlo (see \cite{Giles_2008} for seminal paper) and inspired from Richardson-Romberg extrapolation already exist. For instance, we can refer to \cite{Lemaire_Pages_2017} or \cite{Pages_Panloup_2018} which allows to reach similar rates as with our approach that is $n^{R/(2R+1)}$, for the Richardson-Romberg method of order $R \geqslant 2$. \\

The paper is organized in the following way. Section \ref{section:convergence_inv_distrib_gnl} presents the results from \cite{Pages_Rey_2020} to obtain $a.s.$ weak convergence of (\ref{eq:def_weight_emp_meas_intro_weight}) in an abstract setting. The extension of this abstract framework to be adapted to derive $q$-order ergodic CLT is developed in Section \ref{Rate of convergence - A general approach} where our main abstract result is established (see Theorem \ref{th:conv_gnl_ordre_q}). Almost sure weak convergence and first order CLT for the Euler scheme are given as example at the end of Section \ref{Rate of convergence - A general approach}. 
 Then, in Section \ref{Application - The Talay second weak order scheme}, we apply Theorem \ref{th:conv_gnl_ordre_q} to the second weak order scheme of Talay for Brownian diffusion processes introduced in \cite{Talay_1990}. In particular, in Theorem \ref{th:cv_was_Talay}, we establish the $a.s.$ weak convergence of the empirical measures. We also establish a first order CLT for $(\nu^{\gamma}_n)_{n \in \mathbb{N}^{\ast}}$. In this case, the convergence has the same rate as for the Euler ($i.e.$ $n^{1/3}$) scheme. Finally, we establish the second order CLT for $(\nu^{\eta_2}_n)_{n \in \mathbb{N}^{\ast}}$ with rate $n^{2/5}$. This last result can not be obtained for the Euler scheme as it is simply a first weak order scheme.

\section{Convergence to invariant distributions - A general approach}
\label{section:convergence_inv_distrib_gnl}
%
In this section, we present the abstract framework from \cite{Pages_Rey_2020} to show the convergence of weighted empirical measures defined in a similar way as in (\ref{eq:def_weight_emp_meas_intro_weight}) and built from an approximation $(\overline{X}^{\gamma}_{\Gamma_n})_{n\in \mathbb{N}}$ of a Feller process $(X_t)_{t \geqslant 0}$ (which are not specified explicitly). Given that the step sequence $(\gamma_n)_{n \in \mathbb{N}^{\ast}} \underset{n \to + \infty}{\to}0$, it $a.s.$ weakly converges to the set $\mathcal{V}$, of the invariant distributions of $(X_t)_{t \geqslant 0}$. This framework is based on as weak as possible mean reverting assumptions on the pseudo-generator of $(\overline{X}^{\gamma}_{\Gamma_n})_{n\in \mathbb{N}}$ on the one hand and appropriate rate conditions on the step sequence $(\gamma_n)_{n \in \mathbb{N}^{\ast}}$ on the other hand.
\subsection{Presentation of the abstract framework}
\subsubsection{Notations}
Let $(E,\vert . \vert)$ be a locally compact separable metric space,  let$\mathcal{C}(E)$ the set of continuous functions on $E$ and $\mathcal{C}_0(E)$ the set of continuous functions that vanish at infinity. We equip this space with the sup norm $\Vert f \Vert_{\infty}=\sup_{x \in E} \vert f(x) \vert$ so that $(\mathcal{C}_0(E),\Vert . \Vert_{\infty})$ is a Banach space. We will denote $\mathcal{B}(E)$ the $\sigma$-algebra of Borel subsets of $E$ and $\mathcal{P}(E)$ the family of Borel probability measures on $E$. We will denote by $\mathcal{K}_E$ the set of compact subsets of $E$.\\
Finally, for every Borel function $f:E \to \mathbb{R}$, and every $l_{\infty} \in \mathbb{R} \cup \{-\infty,+\infty\}$, $\lim\limits_{x\to \infty}f(x)= l_{\infty}$ if and only if for every $\epsilon >0$, there exists a compact $K_{\epsilon} \subset \mathcal{K}_E$ such that $\sup_{x \in K_{\epsilon}^c} \vert f(x)- l_{\infty} \vert < \epsilon$ if $l_{\infty} \in \mathbb{R} $, $\inf_{x \in K_{\epsilon}^c}  f(x)  > 1/\epsilon$ if $l_{\infty} =+\infty$, and $\sup\limits_{x \in K_{\epsilon}^c}  f(x)  < -1/\epsilon$ if $l_{\infty} =-\infty$ with $K_{\epsilon}^c=E \setminus K_{\epsilon}.$  \\
%
%
%
%
\subsubsection{Construction of the random measures}
\label{sec:Construction of the random measures}
Let $(\Omega,\mathcal{G}, \mathbb{P})$ be a probability space. We consider a Feller process $(X_t)_{t \geqslant 0}$ (see \cite{Feller_1952} for details) on $(\Omega,\mathcal{G}, \mathbb{P})$ taking values in a locally compact and separable metric space $E$. We denote by $(P_t)_{t \geqslant 0}$ the Feller semigroup (see \cite{Pazy_1992}) of this process. We recall that $(P_t)_{t \geqslant 0}$  is a family of linear operators from $\mathcal{C}_0(E)$ to itself such that $P_0 f=f$, $P_{t+s}f=P_tP_sf$, $t,s \geqslant 0$ (semigroup property) and $\lim\limits_{t \to 0} \Vert P_tf-f \Vert_{\infty}=0$ (Feller property). Using this semigroup, we can introduce the infinitesimal generator of $(X_t)_{t \geqslant 0}$ as a linear operator $A$ defined on a subspace $\DomA$ of $\mathcal{C}_0(E)$, satisfying: For every $f \in \DomA$,
\begin{align*}
Af= \lim\limits_{t \to 0} \frac{P_tf-f}{t}
\end{align*}
exists for the $\Vert . \Vert_{\infty}$-norm. The operator $A: \DomA \to \mathcal{C}_0(E)$ is thus well defined and $\DomA$ is called the domain of $A$. As a consequence of the Echeverria Weiss theorem (see $e.g.$~\cite{Ethier_Kurtz_1986}, Theorem 9.17), the set of invariant distributions for $(X_t)_{t \geqslant 0}$ can be characterized in the following way: 
\begin{align*}
\mathcal{V}=\{ \nu \in \mathcal{P}(E), \forall t \geqslant 0, P_t \nu=  \nu \}=\{ \nu \in \mathcal{P}(E), \forall f \in \DomA, \nu(Af)=0 \}.
\end{align*}
The starting point of our reasoning is thus to consider an approximation of $A$. First, we introduce the family of transition kernels $(\mathscr{Q}_{\gamma})_{\gamma >0}$ from $\mathcal{C}_0(E)$ to itself. Now, let us define the family of linear operators $\widetilde{A} : = (\widetilde{A}_{\gamma} )_{ \gamma >0}$ from $\mathcal{C}_0(E)$ into itself, as follows
 \begin{equation*}
 \forall f \in \mathcal{C}_0(E), \quad \gamma>0,  \qquad \widetilde{A}_{\gamma}f=\frac{\mathscr{Q}_{\gamma}f -f}{\gamma}.
 \end{equation*}
The family $\widetilde{A}$ is usually called the pseudo-generator of the transition kernels $(\mathscr{Q}_{\gamma})_{\gamma>0}$ and is an approximation of $A$ as $\gamma$ tends to zero. From a practical viewpoint, the main interest of our approach is that it is reasonable to assume that there exists $\overline{\gamma}>0$ such that for every $x \in E$ and every $\gamma \in [0, \overline{\gamma}]$, $\mathscr{Q}_{\gamma} (x,dy)$ is simulable at a reasonable computational cost. The family $(\mathscr{Q}_{\gamma})_{\gamma>0}$ is used to build $(\overline{X}_{\Gamma_n})_{n \in \mathbb{N}}$ (this notation replaces $(\overline{X}^{\gamma}_{\Gamma_n})_{n \in \mathbb{N}}$ from now for clarity in the writing) as the non-homogeneous Markov approximation of the Feller process $(X_t)_{t \geqslant 0}$. It is defined on the time grid $\{ \Gamma_n=\sum\limits_{k=1}^n \gamma_k, n \in \mathbb{N} \} $ with the time-step sequence $\gamma:=(\gamma_n)_{n\in \mathbb{N}^{\ast} }$ satisfying
\begin{align*}
\forall n \in \mathbb{N}^{\ast}, \quad 0 < \gamma_n  \leqslant \overline{\gamma}:= \sup_{n \in \mathbb{N}^{\ast}} \gamma_n< + \infty, \quad \lim\limits_{n \to + \infty} \gamma_n = 0 \quad \mbox{ and } \quad \lim\limits_{n \to + \infty}\Gamma_n=+ \infty.
\end{align*}
Notice that we will sometimes use the notation $\gamma_{-m}$ for $m \in \mathbb{N}$. In this case we will always use the convention $\gamma_{-m}=0$. The transition probability distributions of $(\overline{X}_{\Gamma_n})_{n \in \mathbb{N}}$ are given by $\mathscr{Q}_{\gamma_n} (x,dy),n\in \mathbb{N}^{\ast}$, $x\in E$, $i.e. :$
\begin{align*}
 \mathbb{P}(\overline{X}_{\Gamma_{n+1}} \in dy \vert \overline{X}_{\Gamma_n})= \mathscr{Q}_{\gamma_{n+1}}(\overline{X}_{\Gamma_n},dy), \quad n \in \mathbb{N}.
\end{align*}
%
%
We can canonically extend $(\overline{X}_{\Gamma_n})_{n \in \mathbb{N}}$ into a \textit{c\`adl\`ag} process by setting $\overline{X}(t,\omega) =\overline{X}_{\Gamma_{n(t)}}(\omega)$ with $n(t)= \inf \{n \in \mathbb{N}, \Gamma_{n+1}>t \}$. Then $(\overline{X}_{\Gamma_n})_{n \in \mathbb{N}}$ is a simulable (as soon as $\overline{X}_0$ is) non-homogeneous Markov chain with transitions 
\begin{align*}
\forall m  \leqslant n, \qquad \overline{P}_{\Gamma_m,\Gamma_n}(x,dy)= \mathscr{Q}_{\gamma_{m+1}} \circ \cdots \circ \mathscr{Q}_{\gamma_n}(x,dy),
\end{align*}
and law
\begin{align*}
\mathcal{L}(\overline{X}_{\Gamma_n} \vert \overline{X}_{0}=x)=\overline{P}_{\Gamma_n}(x,dy)= \mathscr{Q}_{\gamma_1} \circ \cdots \circ \mathscr{Q}_{\gamma_n}(x,dy).
\end{align*}
We use $(\overline{X}_{\Gamma_n})_{n \in \mathbb{N}}$ to design a Langevin Monte Carlo algorithm. Notice that this approach is generic since the approximation transition kernels $(\mathscr{Q}_{\gamma})_{\gamma>0}$ are not explicitly specified and then, it can be used in many different configurations including among others, weak numerical schemes or exact simulation $i.e.$ $(\overline{X}_{\Gamma_n})_{n \in \mathbb{N}}=(X_{\Gamma_n})_{n \in \mathbb{N}}$. This is of main interest in this paper as we show later that using high weak order schemes of $(X_t)_{t \geqslant 0}$ leads to higher rates of convergence in the CLT satisfied by the weighted empirical measures. Notice that weighted empirical measures are built in a quite more general way than in (\ref{eq:def_weight_emp_meas_intro}) as we consider some general weights which are not necessarily equal to the time steps. We define this weight sequence. Let $\eta:=(\eta_n)_{n \in \mathbb{N}^{\ast}}$ be such that
\begin{equation}
\label{eq:weight_def}
\forall n \in \mathbb{N}^{\ast}, \quad \eta_n \geqslant 0, \quad \lim\limits_{n \to + \infty} H_n=+ \infty, \qquad \mbox{with} \qquad H_n:= H_{\eta,n}= \sum\limits_{k=1}^n \eta_k.
\end{equation}
Now we present our algorithm introduced in \cite{Pages_Rey_2020} and adapted from the one introduced in \cite{Lamberton_Pages_2002} designed with a Euler scheme with decreasing step  $(\overline{X}_{\Gamma_n})_{n \in \mathbb{N}}$ of a Brownian diffusion process $(X_t)_{t \geqslant 0}$. For $x \in E$, let $\delta_x$ denote the Dirac mass at point $x$. For every $n \in \mathbb{N}^{\ast}$, we define the random weighted empirical random measures as follows
 \begin{equation}
 \label{eq:def_weight_emp_meas}
 \nu^{\eta}_n(dx)=\frac{1}{H_n} \sum_{k=1}^n \eta_k \delta_{\overline{X}_{\Gamma_{k-1}}}(dx).
 \end{equation}
%
%
%
%
%
This section of the paper is dedicated to present how to prove that $a.s.$ every weak limiting distribution of $(\nu^{\eta}_n)_{n \in \mathbb{N}^{\ast}}$ belongs to $\mathcal{V}$. In particular when the invariant measure of $(X_t)_{t \geqslant 0}$ is unique, $i.e. \; \mathcal{V}=\{\nu\}$, then $\mathbb{P}-a.s. \;\lim\limits_{n \to + \infty} \nu^{\eta}_n f =\nu f  $, for a generic class of continuous test functions $f$. The approach consists in two steps. First, we establish a tightness property to obtain existence of at least one weak limiting distribution for $(\nu^{\eta}_n )_{n \in \mathbb{N}^{\ast}}$. Then, in a second step, we identify everyone of these limiting distributions with an invariant distributions of the Feller process $(X_t)_{t \geqslant 0}$.
\subsubsection{Assumptions on the random measures}
%
In this part, we present the necessary assumptions on the pseudo-generator $\widetilde{A}  = (\widetilde{A}_{\gamma} )_{ \gamma >0}$ in order to prove the convergence of the empirical measures $(\nu^{\eta}_n)_{n \in \mathbb{N}^{\ast}}$.
\paragraph{Mean reverting recursive control \\}
 In this framework, we introduce a well suited assumption, referred to as the mean reverting recursive control of the pseudo-generator $\widetilde{A}$. This assumption leads to a tightness property on $(\nu^{\eta}_n)_{n \in \mathbb{N}^{\ast}}$ from which follows the existence (in weak sense) of a limiting distribution for $(\nu^{\eta}_n)_{n \in \mathbb{N}^{\ast}}$. A supplementary interest of this approach is that it is designed to obtain the $a.s.$ convergence of $(\nu^{\eta}_n(f))_{n \in \mathbb{N}^{\ast}}$ for a generic class of continuous test functions $f$ which is larger then $\mathcal{C}_b(E)$. To do so, we introduce a Lyapunov function $V$ related to $(\overline{X}_{\Gamma_n})_{n \in \mathbb{N}}$. Assume that $V$ a Borel function such that
\begin{equation}
\label{hyp:Lyapunov}
\mbox{L}_{V}  \quad \equiv \qquad   V :E \to [v_{\ast},+\infty), v_{\ast}> 0 \quad  \mbox{ and } \quad \lim\limits_{ x  \to \infty} V(x)=+ \infty. \\
\end{equation}
We now relate $V$ to $(\overline{X}_{\Gamma_n})_{n \in \mathbb{N}}$ introducing its mean reversion Lyapunov property. Let $\psi, \phi : [v_{\ast},\infty) \to (0,+\infty) $ some Borel functions such that $\widetilde{A}_{\gamma}\psi \circ V$ exists for every $\gamma \in (0, \overline{\gamma}]$. Let $\alpha>0$ and $\beta \in \mathbb{R}$. We assume  
%
%
 \begin{eqnarray}
\label{hyp:incr_sg_Lyapunov}
&\mathcal{RC}_{Q,V} (\psi,\phi,\alpha,\beta) \quad \equiv  \nonumber\\
& \quad  \left\{
    \begin{array}{l}
  (i) \; \quad \exists n_0 \in \mathbb{N}^{\ast},   \forall n   \geqslant n_0, x \in E,  \quad\widetilde{A}_{\gamma_n}\psi \circ V(x)\leqslant  \frac{\psi \circ V(x)}{V(x)}(\beta - \alpha \phi \circ V(x)). \\
  (ii) \quad     \liminf\limits_{y \to + \infty} \phi(y)> \beta / \alpha .
    \end{array}
\right.
\end{eqnarray}

$\mathcal{RC}_{Q,V} (\psi,\phi,\alpha,\beta)$ is called the weakly mean reverting recursive control assumption of the pseudo generator for Lyapunov function $V$. \\

Lyapunov functions are usually used to show the existence and sometimes the uniqueness (see $e.g.$ \cite{DaPrato_Frankowska_2004} or \cite{Bianca_Dogbe_2017}) of the invariant measure of Feller processes. In particular, when $p=1$, the condition $\mathcal{RC}_{Q,V}(I_d,I_d,\alpha,\beta) (i)$ appears as the discrete version of $AV \leqslant \beta-\alpha V$, which is used in that interest for instance in \cite{Hasminskii_1980}, \cite{Ethier_Kurtz_1986}, \cite{Basak_Hu_Wei_1997} or \cite{Pages_2001_ergo}. \\

 The condition $\mathcal{RC}_{Q,V}(V^p,I_d,\alpha,\beta) (i)$, $p \geqslant 1$, is studied in the seminal paper \cite{Lamberton_Pages_2002} (and then in \cite{Lamberton_Pages_2003} with $\phi(y)=y^a,a\in (0,1]$,$y \in [v_{\ast},\infty)$) concerning the Wasserstein convergence of the weighted empirical measures of the Euler scheme with decreasing step  of a Brownian diffusion. When $\phi=I_d$, the Euler scheme is also studied for Markov switching Brownian diffusions in \cite{Mei_Yin_2015}. Notice also that $\mathcal{RC}_{Q,V}(I_d,\phi,\alpha,\beta) (i)$ with $\phi$ concave appears in \cite{DFMS_2004} to prove sub-geometrical ergodicity of Markov chains. In \cite{Lemaire_thesis_2005}, a similar hypothesis to $\mathcal{RC}_{Q,V}(I_d,\phi,\alpha,\beta) (i)$ (with $\phi$ not necessarily concave and $\widetilde{A}_{\gamma_n}$ replaced by $A$), is also used  to study the Wasserstein but also exponential convergence of the weighted empirical measures (\ref{eq:def_weight_emp_meas}) for the Euler scheme of a Brownian diffusion. Finally in \cite{Panloup_2008} similar properties as $\mathcal{RC}_{Q,V}(V^p,V^a,\alpha,\beta) (i)$, $a\in (0,1]$, $p >0$, are developed in the study of the Euler scheme for Levy processes.\\

On the one hand, the function $\phi$ controls the mean reverting property. In particular, we call strongly mean reverting property when $\phi=I_d$ and weakly mean reverting property when $\lim\limits_{y \to +\infty} \phi(y)/y=0$, for instance $\phi(y)=y^a$, $a \in (0,1)$ for every $y \in [v_{\ast},\infty)$. On the other hand, the function $\psi$ is closely related to the identification of the set of test functions $f$ for which we have $\lim\limits_{n \to +\infty} \nu^{\eta}_n(f)=\nu(f) \; a.s.$, when $\nu$ is the unique invariant distribution of the underlying Feller process.\\
 
  To this end, for $s \geqslant 1$, which is related to step weight assumption, we introduce the sets of test functions for which we will show the $a.s.$ convergence of the weighted empirical measures (\ref{eq:def_weight_emp_meas}):
\begin{align}
\label{def:espace_test_function_cv}
\mathcal{C}_{\tilde{V}_{\psi,\phi,s}}(E)=& \big\{ f \in \mathcal{C}(E), \vert f(x) \vert=\underset{  x \to \infty}{o}( \tilde{V}_{\psi,\phi,s} (x) ) \big\}, \\
&\mbox{with} \quad \tilde{V}_{\psi,\phi,s}:E \to \mathbb{R}_+, x \mapsto\tilde{V}_{\psi,\phi,s}(x): =\frac{\phi\circ V(x)\psi \circ V(x)^{1/s}}{V(x)}. \nonumber
\end{align}
Notice that our approach benefits from providing generic results because we consider general Feller processes and approximations but also because the functions $\phi$ and $\psi$ are not specified explicitly.

\paragraph{Infinitesimal generator approximation \\}
This section presents the assumption that enables to characterize the limiting distributions of the $a.s.$ tight  sequence $(\nu^{\eta}_n(dx, \omega))_{n \in \mathbb{N}^{\ast}}$.
 We aim to estimate the distance between $\mathcal{V}$ and $\nu^{\eta}_n$ (see (\ref{eq:def_weight_emp_meas})) for $n$ large enough. We thus introduce an hypothesis concerning the distance between $(\widetilde{A}_{\gamma} )_{ \gamma>0}$, the pseudo-generator of $(\mathscr{Q}_{\gamma} )_{ \gamma>0}$, and $A$, the infinitesimal generator of $(P_t)_{t \geqslant 0}$. We assume that there exists $\DomA_0 \subset \DomA$ with $\DomA_0 $ dense in $\mathcal{C}_0(E)$ such that:
%
%
 \begin{align}
\mathcal{E}(\widetilde{A},A,\DomA_0) \quad \equiv \qquad   \forall \gamma \in (0, \overline{\gamma}], & \forall f \in \DomA_0, \forall x \in E,   \nonumber \\
&  \vert \widetilde{A}_{\gamma} f(x) -Af(x)\vert \leqslant  \Lambda_f(x,\gamma),
 \label{hyp:erreur_tems_cours_fonction_test_reg}
\end{align}
where $\Lambda_{f}:E \times \mathbb{R}_+ \to \mathbb{R}_+$ can be represented in the following way: Let $(\tilde{\Omega},\tilde{\mathcal{G}},\tilde{\mathbb{P}})$ be a probability space. Let $g :E\to \mathbb{R}_+^{q}$, $q \in \mathbb{N}$, be a locally bounded Borel measurable function and let $\tilde{\Lambda}_{f}:(E\times \mathbb{R}_+ \times \tilde{\Omega}, \mathcal{B}(E) \otimes \mathcal{B}(\mathbb{R}_+) \otimes \tilde{\mathcal{G}}) \to \mathbb{R}_+^{q}$ be a measurable function such that  $\sup_{i \in \{1,\ldots,q\} } \tilde{\mathbb{E}}[ \sup_{x \in E} \sup_{\gamma \in (0,\overline{\gamma}] } \tilde{\Lambda}_{f,i}(x,\gamma, \tilde{\omega}) ]< + \infty$ and that we have the representation
%
%
%
\begin{align*}
\forall x \in E , \forall \gamma \in (0,\overline{\gamma}], \qquad \Lambda_f(x,\gamma)= \langle g (  x ) ,\tilde{\mathbb{E}} [\tilde{\Lambda}_{f}(x,\gamma, \tilde{\omega})]  \rangle_{\mathbb{R}^q}
\end{align*}
%
%
%
%
%
%
%
Moreover, we assume that for every $i \in \{1,\ldots,q\}$, $\sup_{n \in \mathbb{N}^{\ast}} \nu_n^{\eta}( g_i ,\omega )< + \infty, \; \mathbb{P}(d\omega)-a.s.$, and that $\tilde{\Lambda}_{f,i}$ satisfies one of the following two properties:\\
There exists a measurable function $\underline{\gamma}:(\tilde{\Omega}, \tilde{\mathcal{G}}) \to((0, \overline{\gamma}],\mathcal{B}((0, \overline{\gamma}]) )$ such that:
\begin{enumerate}[label=\textbf{\Roman*)}]
\item \label{hyp:erreur_tems_cours_fonction_test_reg_Lambda_representation_cv_1} 
\inlineequation[hyp:erreur_temps_cours_fonction_test_reg_Lambda_representation_cv_2_1]{
 \tilde{\mathbb{P}}(d\tilde{\omega})-a.s \qquad \left\{
    \begin{array}{l}
  (i)   \quad \; \;  \forall K \in \mathcal{K}_E ,   \quad  \lim\limits_{\gamma \to 0} \sup\limits_{x \in K} \tilde{\Lambda}_{f,i}(x, \gamma,\tilde{\omega})=0, \qquad \quad \\
  (ii) \quad     \lim\limits_{x \to \infty}  \sup\limits_{\gamma \in (0,\underline{\gamma}(\tilde{\omega}) ]} \tilde{\Lambda}_{f,i}(x, \gamma,\tilde{\omega})=0,  \qquad \qquad \quad 
    \end{array}
\right.
}\\
or
\item \label{hyp:erreur_temps_cours_fonction_test_reg_Lambda_representation_cv_2}\inlineequation[hyp:erreur_temps_cours_fonction_test_reg_Lambda_representation_cv_2_2]{
 \tilde{\mathbb{P}}(d\tilde{\omega})-a.s \qquad  \lim\limits_{\gamma \to 0} \sup\limits_{x \in E} \tilde{\Lambda}_{f,i}(x, \gamma,\tilde{\omega}) g_i(x) =0  . \qquad \qquad \qquad \qquad \qquad \; \;}
\end{enumerate}
\begin{remark}
\label{rmk:representation_mesure_infinie}
Let $(F,\mathcal{F},\lambda)$ be a measurable space. Using the exact same approach, the results we obtain hold when we replace the probability space $(\tilde{\Omega},\tilde{\mathcal{G}},\tilde{\mathbb{P}})$ by the product measurable space $(\tilde{\Omega} \times F,\tilde{\mathcal{G}} \otimes \mathcal{F},\tilde{\mathbb{P}}\otimes \lambda)$ in the representation of $\Lambda_f$ and in (\ref{hyp:erreur_temps_cours_fonction_test_reg_Lambda_representation_cv_2_1}) and (\ref{hyp:erreur_temps_cours_fonction_test_reg_Lambda_representation_cv_2_2}) but we restrict to that case for sake of clarity in the writing. This observation can be useful when we study jump process where $\lambda$ can stand for the jump intensity.
\end{remark}
This representation assumption benefits from the fact that the transition functions $(\mathscr{Q}_{\gamma} (x,dy))_{\gamma \in (0, \overline{\gamma}]}$, $x \in E$, can be represented using distributions of random variables which are involved in the computation of $(\overline{X}_{\Gamma_n})_{n \in \mathbb{N}^{\ast}}$. In particular, this approach is well adapted to stochastic approximations associated to a time grid such as numerical schemes for stochastic differential equations with a Brownian part or/and a jump part. 
\paragraph{Growth control and Step Weight assumptions \\}
We conclude with hypotheses concerning the control of the martingale increments of functions of the approximation $(\overline{X}_{\Gamma_n})_{n \in \mathbb{N}^{\ast}}$. Let $\rho \in [1,2]$ and let $\epsilon_{\mathcal{I}} : \mathbb{R}_+ \to \mathbb{R}_+$ an increasing function. For $F \subset \{f,f:(E, \mathcal{B}(E)) \to (\mathbb{R}, \mathcal{B}(\mathbb{R}) ) \}$ and $g:E \to \mathbb{R}_+$ a Borel function, we assume that, for every $n \in \mathbb{N}$,
 \begin{align}
\label{hyp:incr_X_Lyapunov}
\mathcal{GC}_{Q} & (F,g,\rho,\epsilon_{\mathcal{I}}) \;  \equiv \quad  \mathbb{P}-a.s.\quad  \forall f \in F, \nonumber \\
& \mathbb{E}[ \vert  f  ( \overline{X}_{\Gamma_{n+1}})- \mathscr{Q}_{\gamma_{n+1}}f(\overline{X}_{\Gamma_n}) \vert^{\rho}\vert \overline{X}_{\Gamma_n} ]  \leqslant   C_f \epsilon_{\mathcal{I}}(\gamma_{n+1})  g (\overline{X}_{\Gamma_n}) ,
\end{align}
with $C_f>0$ a finite constant which may depend on $f$. 
\begin{remark}\label{rmrk:Accroiss_mes} The reader may notice that $\mathcal{GC}_{Q}(F,g,\rho,\epsilon_{\mathcal{I}}) $ holds as soon as (\ref{hyp:incr_X_Lyapunov}) is satisfied with $\mathscr{Q}_{\gamma_{n+1}}f(\overline{X}_{\Gamma_n})$, $n  \in \mathbb{N}^{\ast} $, replaced by a $\mathcal{F}^{\overline{X}}_n:=\sigma(\overline{X}_{\Gamma_k},k \leqslant n)$- progressively  measurable process $(\mathfrak{X}_n)_{n \in \mathbb{N}^{\ast}}$ since we have $\mathscr{Q}_{\gamma_{n+1}}f(\overline{X}_{\Gamma_n}) =\mathbb{E}[f(\overline{X}_{\Gamma_{n+1}}) \vert \overline{X}_{\Gamma_n}]$ and $\mathbb{E}[ \vert  f  ( \overline{X}_{\Gamma_{n+1}})- \mathscr{Q}_{\gamma_{n+1}}f(\overline{X}_{\Gamma_n}) \vert^{\rho}\vert \overline{X}_{\Gamma_n} ]  \leqslant 2^{\rho} \mathbb{E}[ \vert  f  ( \overline{X}_{\Gamma_{n+1}})- \mathfrak{X}_n \vert^{\rho}\vert \overline{X}_{\Gamma_n} ]$ for every $\mathfrak{X}_n \in \LL^2(\mathcal{F}^{\overline{X}}_n)$.
\end{remark}

We will combine this first assumption with the following step weight related ones:
\begin{equation}
 \label{hyp:step_weight_I_gen_chow}
\mathcal{S}\mathcal{W}_{\mathcal{I}, \gamma,\eta}(g, \rho , \epsilon_{\mathcal{I}}) \quad \equiv \qquad   \mathbb{P}-a.s. \quad   \sum_{n=1}^{\infty} \Big \vert \frac{\eta_n }{H_n \gamma_n } \Big \vert^{\rho} \epsilon_{\mathcal{I}}(\gamma_n)  g(\overline{X}_{\Gamma_n})  < + \infty,
 \end{equation}

 and
 \begin{align}
 \label{hyp:step_weight_I_gen_tens}
\mathcal{S}\mathcal{W}_{\mathcal{II},\gamma,\eta}(F) \quad \equiv \quad \mathbb{P}-a.s. & \quad \forall f \in F, \nonumber \\
&   \sum_{n=0}^{\infty} \frac{(\eta_{n+1} /\gamma_{n+1}-\eta_n /\gamma_n)_+ }{H_{n+1} }  \vert f(\overline{X}_{\Gamma_n}) \vert < + \infty,
 \end{align}
with the convention $\eta_0/\gamma_0=1$. Notice that this last assumption holds as soon as the sequence $(\eta_n / \gamma_n)_{n \in \mathbb{N}^{\ast} }$ is non-increasing. \\

We end this section presenting the main results concerning the almost sure convergence of empirical measures which are used in this article. Those results were established in \cite{Pages_Rey_2020} in an abstract framework that we will extend to establish CLT.

\subsubsection{Almost sure tightness}
From the recursive control assumption, Theorem \ref{th:tightness} establishes the $a.s.$ tightness of the sequence $(\nu^{\eta}_n)_{n \in \mathbb{N}^{\ast}}$ and also provides a uniform control of $(\nu^{\eta}_n)_{n \in \mathbb{N}^{\ast}}$ on a generic class of test functions. 
%
%
%
\begin{mytheo}
\label{th:tightness}
Let $s \geqslant 1$, $\rho \in[1,2]$, $v_{\ast}>0$, and let us consider the Borel functions $V :E \to [v_{\ast},\infty)$, $g:E \to \mathbb{R}_+$, $\psi : [v_{\ast},\infty) \to \mathbb{R}_+ $ and $\epsilon_{\mathcal{I}} : \mathbb{R}_+ \to \mathbb{R}_+$ an increasing function. We have the following properties:
\begin{enumerate}[label=\textbf{\Alph*.}]
\item\label{th:tightness_point_A}  Assume that $\widetilde{A}_{\gamma_n}(\psi \circ V)^{1/s}$ exists for every $n \in \mathbb{N}^{\ast}$, and that $\mathcal{GC}_{Q}((\psi \circ V)^{1/s},g ,\rho ,\epsilon_{\mathcal{I}}) $ (see (\ref{hyp:incr_X_Lyapunov})), $\mathcal{S}\mathcal{W}_{\mathcal{I}, \gamma,\eta}( g,\rho,\epsilon_{\mathcal{I}}) $ (see (\ref{hyp:step_weight_I_gen_chow})) and $\mathcal{S}\mathcal{W}_{\mathcal{II},\gamma,\eta}((\psi \circ V)^{1/s}) $ (see (\ref{hyp:step_weight_I_gen_tens}) hold. Then
\begin{equation}
\label{eq:invariance_mes_emp_Lyap_gen}
\mathbb{P} \mbox{-a.s.} \quad  \sup_{n\in \mathbb{N}^{\ast} } - \frac{1}{H_n} \sum_{k=1}^n \eta_k \widetilde{A}_{\gamma_k} (\psi \circ V)^{1/s}  (\overline{X}_{\Gamma_{k-1}})< + \infty.
\end{equation}
\item\label{th:tightness_point_B}
Let $\alpha>0$ and $\beta \in \mathbb{R}$. Let $\phi:[v_{\ast},\infty )\to \mathbb{R}_+^{\ast}$ be a continuous function such that $C_{\phi}:= \sup_{y \in [v_{\ast},\infty )}\phi(y)/y< \infty$. Assume that (\ref{eq:invariance_mes_emp_Lyap_gen}) holds and
\begin{enumerate}[label=\textbf{\roman*.}]
\item $\mathcal{RC}_{Q,V}(\psi,\phi,\alpha,\beta)$ (see (\ref{hyp:incr_sg_Lyapunov})) holds.
\item $\mbox{L}_{V}$ (see (\ref{hyp:Lyapunov})) holds and $\lim\limits_{y \to +\infty}  \frac{\phi(y) \psi (y)^{1/s}}{y}=+\infty$.
\end{enumerate}
Then,
 \begin{equation}
 \label{eq:tightness_gen}
\mathbb{P} \mbox{-a.s.} \quad \sup_{n \in \mathbb{N}^{\ast}} \nu_n^{\eta}( \tilde{V}_{\psi,\phi,s} ) < + \infty .
\end{equation}
with $\tilde{V}_{\psi,\phi,s}$ defined in (\ref{def:espace_test_function_cv}). Therefore, the sequence $(\nu^{\eta}_n)_{n \in \mathbb{N}^{\ast}}$ is $\mathbb{P}-a.s.$ tight. 
 \end{enumerate}
\end{mytheo}

\subsubsection{Identification of the limit}
In Theorem \ref{th:tightness}, the tightness - and then existence of a weak limiting distribution - of $(\nu_n^{\eta})_{n \in \mathbb{N}^{\ast}}$ is established. From Theorem \ref{th:identification_limit}, it follows that every limiting point of this sequence is an invariant distribution of the Feller process with infinitesimal generator $A$. 
\begin{mytheo}
\label{th:identification_limit}
Let $\rho \in [1,2]$. We have the following properties:
\begin{enumerate}[label=\textbf{\Alph*.}]
\item\label{th:identification_limit_A} 
Let $\DomA_0 \subset \DomA$, with  $\DomA_0$ dense in $\mathcal{C}_0(E)$. We assume that $\widetilde{A}_{\gamma_n}f$ exists for every $f \in \DomA_0$ and every $n \in \mathbb{N}^{\ast}$. Also assume that there exists $g:E \to \mathbb{R}_+$ a Borel function and $\epsilon_{\mathcal{I}} : \mathbb{R}_+ \to \mathbb{R}_+$ an increasing function such that $\mathcal{GC}_{Q}(\DomA_0,g,\rho ,\epsilon_{\mathcal{I}}) $ (see (\ref{hyp:incr_X_Lyapunov})) and $\mathcal{S}\mathcal{W}_{\mathcal{I}, \gamma,\eta}( g,\rho,\epsilon_{\mathcal{I}}) $ (see (\ref{hyp:step_weight_I_gen_chow})) hold and that
 \begin{equation}
 \label{hyp:accroiss_sw_series_2}  \lim\limits_{n \to + \infty} \frac{1}{H_n} \sum_{k =1}^{n} \vert \eta_{k+1}/\gamma_{k+1}-\eta_k /\gamma_k \vert = 0.
\end{equation} Then
\begin{equation}
\label{hyp:identification_limit}
\mathbb{P} \mbox{-a.s.} \quad \forall f \in \DomA_0, \qquad   \lim\limits_{n \to + \infty}  \frac{1}{H_n} \sum_{k=1}^n \eta_k \widetilde{A}_{\gamma_k}f (\overline{X}_{\Gamma_{k-1}})=0.
\end{equation}
\item \label{th:identification_limit_B} 
We assume that (\ref{hyp:identification_limit}) and $\mathcal{E}(\widetilde{A},A,\DomA_0) $ (see (\ref{hyp:erreur_tems_cours_fonction_test_reg})) hold. Then
\begin{align*}
\mathbb{P} \mbox{-a.s.} \quad \forall f \in \DomA_0, \qquad     \lim\limits_{n \to + \infty} \nu_n^{\eta}( Af )=0.
\end{align*}
It follows that, $\mathbb{P}-a.s.$, every weak limiting distribution $\nu^{\eta}_{\infty}$ of the sequence $(\nu_n^{\eta})_{n \in \mathbb{N}^{\ast}}$ belongs to $\mathcal{V}$, the set of the invariant distributions of $(X_t)_{t \geqslant 0}$. Finally, if the hypotheses from Theorem \ref{th:tightness} point \ref{th:tightness_point_B} hold and $(X_t)_{t \geqslant 0}$ has a unique invariant distribution, $i.e.$ $\mathcal{V}=\{\nu\}$, then 
\begin{align}
\label{eq:test_function_gen_cv}
\mathbb{P} \mbox{-a.s.} \quad \forall f \in \mathcal{C}_{\tilde{V}_{\psi,\phi,s}}(E), \quad \lim\limits_{n \to + \infty} \nu_n^{\eta}(f)=\nu(f),
\end{align}
 with $\mathcal{C}_{\tilde{V}_{\psi,\phi,s}}(E)$ defined in (\ref{def:espace_test_function_cv}).
\end{enumerate}
\end{mytheo}
In the particular case where the function $\psi$ is polynomial and $V$ is quadratic, notice that (\ref{eq:test_function_gen_cv}) implies the $a.s.$ convergence of the empirical measures for the $p$-Wasserstein distances for some $p>0$.

\subsection{About Growth control and Step Weight assumptions}
The following Lemma presents a $\mbox{L}_1$-finiteness property that we can obtain under recursive control hypotheses and strongly mean reverting assumptions ($\phi=I_d$). This result is thus useful to prove $\mathcal{S}\mathcal{W}_{\mathcal{I}, \gamma,\eta}(g, \rho , \epsilon_{\mathcal{I}})$ (see (\ref{hyp:step_weight_I_gen_chow})) or $\mathcal{S}\mathcal{W}_{\mathcal{II},\gamma,\eta}(F)  $ (see (\ref{hyp:step_weight_I_gen_tens})) for well chosen $F$ and $g$ in this specific situation.
\begin{lemme}
\label{lemme:mom_psi_V}
Let $v_{\ast}>0$, $V:E \to [v_{\ast},\infty) $, $\psi:[v_{\ast},\infty) \to \mathbb{R}_+, $ such that $\widetilde{A}_{\gamma_n}\psi \circ V$ exists for every $n \in \mathbb{N}^{\ast}$. Let $\alpha>0$ and $\beta \in \mathbb{R}$. We assume that $\mathcal{RC}_{Q,V}(\psi,I_d,\alpha,\beta)$ (see (\ref{hyp:incr_sg_Lyapunov})) holds and that $\mathbb{E}[\psi\circ V (\overline{X}_{\Gamma_{n_0}})]< + \infty$ for every $n_0 \in \mathbb{N}^{\ast}$. Then
\begin{align}
\label{eq:mom_psi_V}
\sup_{n \in \mathbb{N}} \mathbb{E}[\psi \circ V(\overline{X}_{\Gamma_n})] < + \infty
\end{align}
In particular, let $\rho \in [1,2]$ and $\epsilon_{\mathcal{I}} : \mathbb{R}_+ \to \mathbb{R}_+$, an increasing function. It follows that if $\sum_{n=1}^{\infty} \Big \vert \frac{\eta_n }{H_n \gamma_n } \Big \vert^{\rho} \epsilon_{\mathcal{I}}(\gamma_n)   < + \infty$, then $\mathcal{S}\mathcal{W}_{\mathcal{I}, \gamma,\eta}(\psi \circ V, \rho , \epsilon_{\mathcal{I}}) $ holds and if $ \sum_{n=0}^{\infty} \frac{(\eta_{n+1} /\gamma_{n+1}-\eta_n /\gamma_n)_+ }{H_{n+1} } < + \infty$, then $\mathcal{S}\mathcal{W}_{\mathcal{II},\gamma,\eta}(\psi \circ V) $ is satisfied.
\end{lemme}
%
%

Now, we provide a general way to obtain $\mathcal{S}\mathcal{W}_{\mathcal{I}, \gamma,\eta}(g,\rho, \epsilon_{\mathcal{I}})$ and $\mathcal{S}\mathcal{W}_{\mathcal{I}\mathcal{I},\gamma,\eta}(F)$  for some specific $g$ and $F$ as soon as a recursive control with weakly mean reversion assumption holds.
\begin{lemme}
\label{lemme:mom_V}
Let $v_{\ast}>0$, $V:E \to [v_{\ast},\infty) $, $\psi, \phi :[v_{\ast},\infty) \to \mathbb{R}_+, $ such that $\widetilde{A}_{\gamma_n}\psi \circ V$ exists for every $n \in \mathbb{N}^{\ast}$. Let $\alpha>0$ and $\beta \in \mathbb{R}$. We also introduce the non-increasing sequence $(\theta_n)_{n \in \mathbb{N}^{\ast}}$ such that $\sum_{n \geqslant 1} \theta_n \gamma_n < + \infty$. We assume that $\mathcal{RC}_{Q,V}(\psi,\phi,\alpha,\beta)$ (see (\ref{hyp:incr_sg_Lyapunov})) holds and that $\mathbb{E}[\psi\circ V (\overline{X}_{\Gamma_{n_0}})]< + \infty$ for every $n_0 \in \mathbb{N}^{\ast}$. Then
\begin{equation*}
\sum_{n=1}^{\infty} \theta_n \gamma_n \mathbb{E} [\tilde{V}_{\psi,\phi,1}(\overline{X}_{\Gamma_{n-1}}) ] < + \infty
\end{equation*}
 with $ \tilde{V}_{\psi,\phi,1}$ defined in (\ref{def:espace_test_function_cv}). In particular, let $\rho \in [1,2]$ and $\epsilon_{\mathcal{I}} : \mathbb{R}_+ \to \mathbb{R}_+$, an increasing function. If we also assume 
 \begin{align}
 \label{hyp:step_weight_I}
\mathcal{S}\mathcal{W}_{\mathcal{I}, \gamma,\eta}(\rho, \epsilon_{\mathcal{I}})  \quad \equiv \qquad & \Big( \gamma_n^{-1} \epsilon_{\mathcal{I}}(\gamma_n) \big( \frac{\eta_n }{H_n \gamma_n } \big)^{\rho} \Big)_{n \in \mathbb{N}^{\ast}} \mbox{ is non-increasing and } \nonumber \\
& \sum_{n=1}^{\infty} \Big( \frac{\eta_n }{H_n \gamma_n } \Big)^{\rho} \epsilon_{\mathcal{I}}(\gamma_n) < + \infty,
 \end{align}
 then we have $\mathcal{S}\mathcal{W}_{\mathcal{I}, \gamma,\eta}(\tilde{V}_{\psi,\phi,1},\rho,\epsilon_{\mathcal{I}}) $ (see (\ref{hyp:step_weight_I_gen_chow})). Finally,if
  \begin{align}
 \label{hyp:step_weight_II}
\mathcal{S}\mathcal{W}_{\mathcal{II},\gamma,\eta}  \quad \equiv \qquad &\Big( \frac{ \frac{\eta_{n+1} }{(\gamma_{n+1}}-\frac{\eta_n}{\gamma_n})_+ }{ \gamma_n H_n }  \Big)_{n \in \mathbb{N}^{\ast}} \mbox{ is non-increasing and } \nonumber \\
&\sum_{n=1}^{\infty}  \frac{(\eta_{n+1} /\gamma_{n+1}-\eta_n /\gamma_n)_+ }{H_n }  < + \infty,
 \end{align}
  then we have $\mathcal{S}\mathcal{W}_{\mathcal{II},\gamma,\eta}( \tilde{V}_{\psi,\phi,1}) $ (see (\ref{hyp:step_weight_I_gen_tens})).
\end{lemme}

\section{Rate of convergence - A general approach}
\label{Rate of convergence - A general approach}
In this section, we extend the abstract framework from Section \ref{section:convergence_inv_distrib_gnl} to be adapted to establish an ergodic CLT for the empirical measures (\ref{eq:def_weight_emp_meas}) with the highest possible order. The approach we propose is twofold. First we give appropriate weak error type estimations and on second we give suitable step weight assumptions to control the martingale part of increments of functions of the approximation $(\overline{X}_{\Gamma_n})_{n \in \mathbb{N}^{\ast}}$. Notice that, together with the choice of weights $(\eta_{n})_{n \in \mathbb{N}^{\ast}}$, the weak error type estimations are the crucial tool to obtain high order of convergence in the CLT satisfied by the weighted empirical measures (\ref{eq:def_weight_emp_meas}).

\subsection{Assumption on the random measures}
\paragraph{Weak approximation assumption \\}

In this section, our purpose is to propose a new version of $\mathcal{E}(\widetilde{A},A,\DomA_0)$ (see (\ref{hyp:erreur_tems_cours_fonction_test_reg}) in Section \ref{section:convergence_inv_distrib_gnl}) which is adapted to obtain the CLT with order $q \in \mathbb{N}^{\ast}$  $i.e.$ Theorem \ref{th:conv_gnl_ordre_q}. \\
We begin by giving an intuition of this new version of (\ref{hyp:erreur_tems_cours_fonction_test_reg}).  In order to obtain the CLT, even in the case $q=1$, we need a sharper estimate of the error $\widetilde{A}_{\gamma}-A$ than the one provided by (\ref{hyp:erreur_tems_cours_fonction_test_reg}). In particular, we need to identify the dominating term of $\widetilde{A}_{\gamma}-A$ when $\gamma$ is close to $0$ and impose a similar assumption as (\ref{hyp:erreur_tems_cours_fonction_test_reg}) but with $\widetilde{A}_{\gamma}$ replaced by $\widetilde{A}_{\gamma}-A$ and $A$ replaced by the dominating term of $\widetilde{A}_{\gamma}-A$. This approach is well adapted when $q=1$ and for higher values of $q$, the intuition is similar but assumptions are made on the dominating term of $\widetilde{A}_{\gamma}-\sum_{i=1}^{q} \frac{\gamma^i}{i !}A^{i}$.  \\

We now describe rigorously the framework and the hypotheses we need to impose on the dominating term of $\widetilde{A}_{\gamma}-A$  or its higher order counterpart. \\
This new version of (\ref{hyp:erreur_tems_cours_fonction_test_reg}) is expected to act on a more restricted set of test functions so we introduce $F \subset \{f,f:(E, \mathcal{B}(E)) \to (\mathbb{R}, \mathcal{B}(\mathbb{R}) ) \}$ to replace $\DomA_0$.  \\
Moreover, in order to describe the behavior of the dominating term of $\widetilde{A}_{\gamma}-\sum_{i=1}^{q} \frac{\gamma^i}{i !}A^{i}$, let us introduce $\mathfrak{M}_q$, a linear operator acting on $F$, and let us introduce $\tilde{\eta}_{q}:\mathbb{R}_+ \times\{0,\ldots,(q-2)_+\} \to \mathbb{R}_+$ (with notation $b_+=b \vee 0$ for $b \in \mathbb{R}$).  \\
Notice that the CLT we obtain in Theorem \ref{th:conv_gnl_ordre_q} is restricted to test function $Af$ with $f \in F$. Moreover, the rate of the CLT we obtain is monitored by $\tilde{\eta}_{q}$ while the asymptotic mean is built with $\mathfrak{M}_q$.\\

We are now in a position to introduce the new version of (\ref{hyp:erreur_tems_cours_fonction_test_reg}). We assume that the weight sequence $(\tilde{\eta}_{q,n})_{n \in \mathbb{N}^{\ast}}=(\tilde{\eta}_{q}(\gamma_n, n \mod (q-1)))_{n \in \mathbb{N}^{\ast}}$ (with convention $n \mod 0=0$) is decreasing, satisfies (\ref{eq:weight_def}), is such that $\mathbb{P}-a.s.$, $\lim_{n \to \infty}\nu ^{\tilde{\eta}_q}_n(\mathfrak{M}_q f) = \nu(\mathfrak{M}_q f)$ for every $f \in F$ and
 \begin{align}
  \label{hyp:rate_erreur_tems_cours_fonction_test_reg}
\mathcal{E}_q(F,\tilde{A},A, \mathfrak{M}_{q},\tilde{\eta}_q)& \quad  \equiv \quad  \forall f \in F, \forall x \in E, \forall  \gamma \in (0, \overline{\gamma}], \forall e \in \{0,\ldots,(q-2)_+\},\\
&   \Big\vert \mathcal{R}_{q} f (x,\gamma,e) - \tilde{\eta}_q(\gamma,e) \mathfrak{M}_qf(x) \Big \vert \leqslant \tilde{\eta}_q(\gamma,e)\Lambda_{f,q}(x,\gamma), \nonumber
\end{align}

where $\mathcal{R}_qf$ is defined in the following way. \\

We fix the weights of the order one and two discretization, $c^{1}=1$, $c^{2}=(\frac{1}{2},\frac{1}{2})$. When $q > 2$, we consider $\mathscr{C}(q) \in \mathbb{N}$, $\mathscr{C}(q) \geqslant q$ with $\mathscr{C}(1)=1$ and $\lambda=(\lambda^{1},\ldots,\lambda^{\mathscr{C}(q)-2}) \in  \mathbb{R}^{\mathscr{C}(q)-2} \times \mathbb{R}^{\mathscr{C}(q)-3} \times \ldots \times \mathbb{R}$ and for $p \in \mathbb{N}^{\ast}$, $2 \leqslant p \leqslant q$, the quantity
$\zeta^{q,p}(\lambda) =(\zeta^{q,p}_{0}(\lambda),\ldots,\zeta^{q,p}_{p-1}(\lambda))\in \mathbb{R}^{\mathscr{C}(q)-\mathscr{C}(q-p+1)+1}$. We assume that $\zeta ^{q,q}(\lambda) $ satisfies
\begin{align*}
\zeta_{u}^{q,q}(\lambda) = \frac{1}{q!}\mathds{1}_{u<\mathscr{C}(q)-1}+\sum_{i=0}^{\mathscr{C}(q)-3}\sum_{l=1}^{\mathscr{C}(q)-2-i} \frac{\sum_{h=l}^{\mathscr{C}(q)-2-i}\lambda^{h}_{i}}{(q+l)!} (-1)^{l-u+i}\binom{l}{u-i} \mathds{1}_{0 \leqslant u-i \leqslant l},
\end{align*}

for $u \in \{0,\ldots,q-1\}$. Then, we define
\begin{align*}
\boldsymbol{\lambda}^{q} := \{  \lambda \in  \mathbb{R}^{\mathscr{C}(q)-2}  \times \ldots \times \mathbb{R} , 
\forall p=3,\ldots,q, \exists \zeta^{q,p-1}(\lambda)  \in \mathbb{R}^{p-1} , M^{q,p}\zeta^{q,p-1}(\lambda)=N^{q,p}(\zeta^{q,p}(\lambda)) \} 
\end{align*}

where
\begin{align*}
N^{q,p}:\mathbb{R}^{\mathscr{C}(q)-\mathscr{C}(q-p+1)+1} & \to \mathbb{R}^{\mathscr{C}(q)} \\
\zeta &\mapsto   \frac{1}{(p-1) !} \mathds{1}_{e < \mathscr{C}(q)-1 } +\zeta_{e-(\mathscr{C}(q-p+1)-2)_+-1}^{q,p}(\lambda) \mathds{1}_{\mathscr{C}(q)-\mathscr{C}(q-p+1)+1  \geqslant e-(\mathscr{C}(q-p+1)-2)_+  \geqslant 1} \\
 & \qquad -\zeta_{e}^{q,p}(\lambda) \mathds{1}_{e < (\mathscr{C}(q)-\mathscr{C}(q-p+1)+1) \wedge (\mathscr{C}(q)-1) }
\end{align*}

 and $M^{q,p}\in \mathbb{R}^{\mathscr{C}(q) \times \mathscr{C}(q-p+1)}$ is defined by $M^{q,p}_{j,u}=   c_{j-u}^{q-p+2} \mathds{1}_{j-u \in \{ 0,\ldots, \mathscr{C}(q-p+2)-1\} }$. 
We then choose arbitrarily $\mathscr{C}(q) \geqslant q$ and $\hat{ \lambda} \in \boldsymbol{\lambda}^{q} $ as soon as it is not empty. Notice that for $q \leqslant 3$, it is sufficient to take $\mathscr{C}(q)=q$ and $\lambda=0$.
The weights for the $q$ order discretization are then given by
\begin{align}
\label{def:coeff_c}
c_{j}^q=& \mathds{1}_{j \leqslant \mathscr{C}(q)-2} + \sum_{u=0}^{\mathscr{C}(q)-\mathscr{C}(q-1)} \zeta^{q,2}_{u}(\hat{\lambda}) ( \mathds{1}_{j =(\mathscr{C}(q-1)-2)_{+}+u+1} - \mathds{1}_{j=u} ) ,
\end{align}

for every $j \in \{0,\ldots,\mathscr{C}(q)-1\}$. Moreover, for every $x,y \in E$, $\gamma \geqslant 0$, $e \in \{0,(\mathscr{C}(q)-2)_+\}$, we let

\begin{align*}
\mathcal{R}_1f(x,\gamma,e):=&\mathcal{R}_1f(x,\gamma):= -\tilde{\mathcal{R}}_{1} f(x,\gamma),  \quad \mbox{and for }q \in \mathbb{N}^{\ast},\\ 
\mathcal{R}_qf(x,\gamma,e):= & -(1-\sum_{l=1}^{\mathscr{C}(q)-2-e} \hat{\lambda}^{l}_{e})  \tilde{\mathcal{R}}_{q} f(x,\gamma)  - \sum_{l=1}^{\mathscr{C}(q)-2-e}  \hat{\lambda}^{l}_{e}  \tilde{\mathcal{R}}_{q+l} f(x,\gamma)   \\
&-\sum_{k=0}^{\mathscr{C}(q)-2}\sum_{l=1}^{\mathscr{C}(q)-2-k} \frac{\sum_{h=l}^{\mathscr{C}(q)-2-k}\hat{\lambda}^{h}_{k}}{(q+l)!}  \sum_{i=0}^{l-1} (-1)^{l-i-e+k-1}\binom{l-i-1}{e-k} \gamma_{n_{q}+1}^{q+i} \mathcal{R}_{1}A^{q+i}f(x,\gamma) \mathds{1}_{0 \leqslant e-k \leqslant l-i }\\
&  -\sum_{l=1}^{q-2} \gamma^{q-l} \sum_{u=(e-(\mathscr{C}(l)-2)_+)_{+}}^{e \wedge (\mathscr{C}(q)-\mathscr{C}(l))}  \zeta_u^{q,q-l+1}(\hat{\lambda})   \mathcal{R}_{l}A^{q-l}f(x,\gamma,e-u ) \\
& -\gamma \sum_{u=0}^{\mathscr{C}(q)-\mathscr{C}(q-1)} \zeta_u^{q,2}(\hat{\lambda})   \mathcal{R}_{q-1} Af(x,\gamma,e-u) \mathds{1}_{u \leqslant e \leqslant (q-3)_++u} .
\end{align*}

where, for every $m \in \{1,\ldots,q\}$, the measurable functions
\begin{align*}
\begin{array}{crcl}
\tilde{\mathcal{R}}_{m}f  & : E\times \mathbb{R}_+  & \to &  \mathbb{R} \\
 &(x, \gamma ) & \mapsto &   \gamma \tilde{A}_{\gamma}f(x)-\sum_{i=1}^{m} \frac{\gamma^i}{i !}A^{i} f(x),
\end{array}
\end{align*}
are supposed to be well defined for every $f \in F$.\\

In addition, we also assume that $\Lambda_{f,q}:E \times \mathbb{R}_+ \to \mathbb{R}_+$ can be represented in the following way:  Let $(\tilde{\Omega},\tilde{\mathcal{G}},\tilde{\mathbb{P}})$ be a probability space. Let $g :E\to \mathbb{R}_+^{l}$, $l \in \mathbb{N}^{\ast}$, be a locally bounded Borel measurable function and let $\tilde{\Lambda}_{f,q}:(E\times \mathbb{R}_+ \times \tilde{\Omega}, \mathcal{B}(E) \otimes \mathcal{B}(\mathbb{R}_+) \otimes \tilde{\mathcal{G}}) \to \mathbb{R}_+^{l}$ be a measurable function such that  
\begin{align}
\label{eq:domination_erreur_L1}
\sup_{i \in \{1,\ldots,l\} } \tilde{\mathbb{E}}[ \sup_{x \in E} \sup_{\gamma \in (0,\overline{\gamma}] } \tilde{\Lambda}_{f,q,i}(x,\gamma, \tilde{\omega}) ]< + \infty
\end{align}
 and that the following representation assumption holds
%
%
%
\begin{align*}
\forall x \in E , \forall \gamma \in (0,\overline{\gamma}], \qquad \Lambda_{f,q}(x,\gamma)= \langle g (  x ) ,\tilde{\mathbb{E}} [\tilde{\Lambda}_{f,q}(x,\gamma, \tilde{\omega})]  \rangle_{\mathbb{R}^l}.
\end{align*}
 Moreover, we assume that for every $i \in \{1,\ldots,l \}$, $\sup_{n \in \mathbb{N}^{\ast}} \nu_n^{ \tilde{\eta}_q }( g_i ,\omega )< + \infty, \; \mathbb{P}-a.s.$, and that $\tilde{\Lambda}_{f,q,i}$ satisfies one of the two following properties.\\
There exists a measurable function $\underline{\gamma}:(\tilde{\Omega}, \tilde{\mathcal{G}}) \to((0, \overline{\gamma}],\mathcal{B}((0, \overline{\gamma}]) )$ such that:
\begin{enumerate}[label=\textbf{\Roman*)}]
\item \label{hyp:erreur_tems_cours_fonction_test_reg_Lambda_representation_1} 
\inlineequation[hyp:erreur_temps_cours_fonction_test_reg_Lambda_representation_2_1]{
 \tilde{\mathbb{P}}(d\tilde{\omega})-a.s \qquad \left\{
    \begin{array}{l}
  (i)   \quad \; \;  \forall K \in \mathcal{K}_E ,   \quad  \lim\limits_{\gamma \to 0} \sup\limits_{x \in K} \tilde{\Lambda}_{f,q,i}(x, \gamma,\tilde{\omega})=0, \\
  (ii) \quad     \lim\limits_{\vert x \vert \to \infty}  \sup\limits_{\gamma \in (0,\underline{\gamma}(\tilde{\omega}) ]} \tilde{\Lambda}_{f,q,i}(x, \gamma,\tilde{\omega})=0,  \qquad \qquad \qquad \qquad \qquad \quad
    \end{array}
\right.
}
 or
\item \label{hyp:erreur_temps_cours_fonction_test_reg_Lambda_representation_2}\inlineequation[hyp:erreur_temps_cours_fonction_test_reg_Lambda_representation_2_2]{
 \tilde{\mathbb{P}}(d\tilde{\omega})-a.s \qquad  \lim\limits_{\gamma \to 0} \sup\limits_{x \in E} \tilde{\Lambda}_{f,q,i}(x, \gamma,\tilde{\omega}) g_i(x) =0  . \qquad \qquad \qquad \qquad \qquad \qquad \qquad \qquad \; \;}
\end{enumerate}
\begin{remark}
\label{rmk:representation_mesure_infinie_rate}
Let $(F,\mathcal{F},\lambda)$ be a measurable space. Using the exact same approach, the results we obtain hold when we replace the probability space $(\tilde{\Omega},\tilde{\mathcal{G}},\tilde{\mathbb{P}})$ by the product measurable space $(\tilde{\Omega} \times F,\tilde{\mathcal{G}} \otimes \mathcal{F},\tilde{\mathbb{P}}\otimes \lambda)$ in the representation of $\Lambda_{f,q}$ and in (\ref{hyp:erreur_temps_cours_fonction_test_reg_Lambda_representation_2_1}) and (\ref{hyp:erreur_temps_cours_fonction_test_reg_Lambda_representation_2_2}). It is a similar observation as in the study of the convergence as pointed out in Remark \ref{rmk:representation_mesure_infinie}.
\end{remark}

\paragraph{Growth assumption \\}

In this section we introduce a new version of $\mathcal{GC}_{Q}$ (see (\ref{hyp:incr_X_Lyapunov})) adapted to prove CLT of order $q \in \mathbb{N}^{\ast}$. In particular, this new version is crucial to identify the asymptotic variance in the CLT we obtain in Theorem \ref{th:conv_gnl_ordre_q}.\\

We denote by $\mathcal{P}_{\overline{X},2}$ the set of  $\mathcal{F}^{\overline{X}}_n:=\sigma(\overline{X}_{\Gamma_k},k \leqslant n)$- progressively measurable processes $(\mathscr{X}_n)_{n \in \mathbb{N}^{\ast}}$ with $\mathscr{X}_{n+1} \in \mbox{L}^2(\mathcal{F}^{\overline{X}}_n)$ and $\mathbb{E}[\mathscr{X}_{n+1} \vert \overline{X}_{\Gamma_n}]=0 $ for every $n \in \mathbb{N}$. Let $\rho \in [1,2]$ and let $\epsilon_{\mathfrak{X}}, \epsilon_{\mathcal{G}\mathcal{C}} : \mathbb{R}_+ \to \mathbb{R}_+$ be two increasing functions such that the weight sequence $(\epsilon_{\mathfrak{X},n})_{n \in \mathbb{N}^{\ast}}=(\epsilon_{\mathfrak{X}}(\gamma_n))_{n \in \mathbb{N}^{\ast}}$ satisfies (\ref{eq:weight_def}). Let $F \subset \{f,f:(E, \mathcal{B}(E)) \to (\mathbb{R}, \mathcal{B}(\mathbb{R}) ) \}$ and $g:E \to \mathbb{R}_+$ be a Borel measurable function. Finally, we introduce a linear operator $\mathfrak{V}$ defined on $F$ that is the main ingredient to compute the asymptotic variance in Theorem \ref{th:conv_gnl_ordre_q}. In particular, we identify the asymptotic variance in the CLT as the $a.s.$ limit of weighted (with weights $(\epsilon_{\mathfrak{X},n})_{n \in \mathbb{N}^{\ast}}$) empirical measures (\ref{eq:def_weight_emp_meas}) applied to functions taking form $\mathfrak{V}f$, $f \in F$.\\

We are now in a position to introduce the new version of (\ref{hyp:incr_X_Lyapunov}). We assume that $A^m f$ is well defined for every $m \in \{0,\ldots,q-1\}$ and every $f \in F$ and that

 \begin{align}
\label{hyp:incr_X_Lyapunov_vitesse}
&\mathcal{GC}_{Q,q}   (F,g,\rho,\epsilon_{\mathscr{X}},  \epsilon_{\mathcal{G}\mathcal{C}},  \mathfrak{V})  \quad \equiv \qquad   \mathbb{P}-a.s. \quad  \forall f \in F, \exists \mathscr{X}_{f} \in \mathcal{P}_{\overline{X},2}  \nonumber \\
 & \mathbb{E}\Big[\Big \vert  B_{q}f(\overline{X}_{\Gamma_n} ,\overline{X}_{\Gamma_{n+1}} ,\gamma_{n+1}, n \mbox{\; mod \;} (q-1)) - \mathscr{X}_{f,n+1}  \Big\vert^{\rho} \Big\vert \overline{X}_{\Gamma_n} \Big]  \leqslant    C_{f}\epsilon_{\mathcal{G}\mathcal{C}}(\gamma_{n+1})  g (\overline{X}_{\Gamma_n}) .
\end{align}

with, for every $x,y \in E$, $\gamma \geqslant 0$, $e \in \{0,(\mathscr{C}(q)-2)_+\}$,
\begin{align*}
B_{1} &f(x,y,\gamma,e):=B_{1}f(x,y,\gamma) :=  \mathscr{Q}_{\gamma}f(x) - f(y),  \quad \mbox{and for }q \in \mathbb{N}^{\ast},\\
 B_{q} & f(x,y,\gamma,e) :=  B_1 f (x,y,\gamma) \nonumber \\
 &-\sum_{k=0}^{\mathscr{C}(q)-2}\sum_{l=1}^{\mathscr{C}(q)-2-k} \frac{\sum_{h=l}^{\mathscr{C}(q)-2-k}\hat{\lambda}^{h}_{k}}{(q+l)!}  \sum_{i=0}^{l-1} (-1)^{l-i-e+k-1}\binom{l-i-1}{e-k} \gamma_{n_{q}+1}^{q+i} B_{1}A^{q+i}f(x,y,\gamma) \mathds{1}_{0 \leqslant e-k \leqslant l-i\}}\\
 &-\sum_{l=1}^{q-2} \gamma^{q-l} \sum_{u=(e-(\mathscr{C}(q)-2)_+)_{+}}^{e \wedge (\mathscr{C}(q)-\mathscr{C}(l))}  \zeta_u^{q,q-l+1}(\hat{\lambda})  B_{l }A^{q-l}f(x,y,\gamma,e-u ) \\
& -\gamma \sum_{u=0}^{\mathscr{C}(q)-\mathscr{C}(q-1)} \zeta_u^{q,2}(\hat{\lambda})   B_{q-1} Af(x,y,\gamma,e-u) \mathds{1}_{u \leqslant e \leqslant (q-3)_++u},
\end{align*}


and $\mathbb{E}[\vert \mathscr{X}_{f,n+1}\vert^2 \vert \overline{X}_{\Gamma_n} ]= \epsilon_{\mathscr{X}}(\gamma_{n+1})\mathfrak{V}f(\overline{X}_{\Gamma_n} )$ with for every $f \in F$, $\lim_{n \in \mathbb{N}^{\ast}} \nu_n^{\epsilon_{\mathscr{X}}}( \mathfrak{V}f ,\omega )= \nu (\mathfrak{V}f ), \; \mathbb{P}-a.s.$, and

%
%
%

\begin{align}
\label{hyp:Lindeberg_CLT_scheme}
\forall \mathfrak{E} >0 , \quad  \lim_{n \to \infty}  \frac{ 1}{H_{\epsilon_{\mathscr{X}},n}}  \sum_{k=0}^{n-1}\mathbb{E}[  \vert \mathscr{X}_{f,k+1} \vert^2\mathds{1}_{ \vert   \mathscr{X}_{f,k+1} \vert > \sqrt{H_{\epsilon_{\mathscr{X}},n}} \mathfrak{E} } \vert \overline{X}_{\Gamma_{k}}] \overset{\mathbb{P}}{=} 0.
\end{align}

\begin{remark}\label{rmrk:Accroiss_mes_vitesse} The reader may notice that $\mathcal{GC}_{Q,q}(F,q,\rho,\epsilon_{\mathscr{X}},\epsilon_{\mathcal{G}\mathcal{C}},\mathfrak{V}) $ holds as soon as (\ref{hyp:incr_X_Lyapunov}) is satisfied with $\mathscr{Q}_{\gamma_{n+1}}A^mf(\overline{X}_{\Gamma_n})$, $n  \in \mathbb{N}^{\ast} $, $m  \in \mathbb{N}^{\ast} $ replaced by a $\mathcal{F}^{\overline{X}}_n:=\sigma(\overline{X}_{\Gamma_k},k \leqslant n)$- progressively  measurable process $(\mathfrak{X}_{m,n})_{n \in \mathbb{N}^{\ast}}$, since $\rho \in [1,2]$ and we have $\mathscr{Q}_{\gamma_{n+1}}A^mf(\overline{X}_{\Gamma_n}) =\mathbb{E}[A^mf(\overline{X}_{\Gamma_{n+1}}) \vert \overline{X}_{\Gamma_n}]$ and $\mathbb{E}[\mathscr{X}_{f,n+1} \vert \overline{X}_{\Gamma_n}]=0$.
\end{remark}


In the following, we will combine this assumption with 
\begin{equation}
 \label{hyp:step_weight_I_gen_chow_rate}
\mathcal{S}\mathcal{W}_{\mathcal{G}\mathcal{C}, \gamma}(g, \rho , \epsilon_{\mathscr{X}},\epsilon_{\mathcal{G}\mathcal{C}})  \quad \equiv \qquad \quad  \mathbb{P}-a.s. \quad   \sum_{n=1}^{\infty} \frac{\epsilon_{\mathcal{G}\mathcal{C}}(\gamma_n)    }{H_{\epsilon_{\mathscr{X}},n} ^{\rho/2} }  g(\overline{X}_{\Gamma_n})  < + \infty.
\end{equation}

Notice that, as a consequence of Lemma \ref{lemme:mom_V}, if we suppose that  $\mathcal{RC}_{Q,V}(\psi,\phi,\alpha,\beta)$ (see (\ref{hyp:incr_sg_Lyapunov})) holds, that $\mathbb{E}[\psi\circ V (\overline{X}_{\Gamma_{n_0}})]< + \infty$ for every $n_0 \in \mathbb{N}^{\ast}$ and that

 \begin{equation}
 \label{hyp:step_weight_I_gen_chow_rate_sans_g}
\mathcal{S}\mathcal{W}_{\mathcal{G}\mathcal{C}, \gamma}( \rho , \epsilon_{\mathscr{X}},\epsilon_{\mathcal{G}\mathcal{C}})  \quad \equiv  \qquad \Big(\frac{\epsilon_{\mathcal{G}\mathcal{C}}(\gamma_n)    }{\gamma_n H_{\epsilon_{\mathscr{X}},n} ^{\rho/2} } \Big)_{n \in \mathbb{N}^{\ast}} \mbox{ is non increasing and } \sum_{n=1}^{\infty} \frac{\epsilon_{\mathcal{G}\mathcal{C}}(\gamma_n)    }{H_{\epsilon_{\mathscr{X}},n} ^{\rho/2} }< + \infty,
 \end{equation}
holds, then we have $\mathcal{S}\mathcal{W}_{\mathcal{G}\mathcal{C}, \gamma}(\tilde{V}_{\psi,\phi,1}, \rho , \epsilon_{\mathscr{X}},\epsilon_{\mathcal{G}\mathcal{C}}) $ (see (\ref{hyp:step_weight_I_gen_chow_rate})) with $\tilde{V}_{\psi,\phi,1}$ defined in (\ref{def:espace_test_function_cv}).

\subsection{Convergence rate results}
 We begin with some preliminary results.
\begin{lemme}{\textbf{(Kronecker).}}
\label{lemme:Kronecker}
 Let $(a_n)_{n \in \mathbb{N}^{\ast}}$ and $(b_n)_{n \in \mathbb{N}^{\ast}}$ be two sequences of real numbers. If $(b_n)_{n \in \mathbb{N}^{\ast}}$ is non-decreasing, strictly positive, with $\lim\limits_{n \to + \infty}b_n=+\infty$ and $\sum\limits_{n \geqslant 1} a_n /b_n$ converges in $\mathbb{R}$, then
 \begin{equation*}
 \lim\limits_{n \to +\infty} \frac{1}{b_n} \sum_{k=1}^n a_k=0.
 \end{equation*}
\end{lemme}
\begin{mytheo}{\textbf{(Chow} (see \cite{Hall_Heyde_1980}, Theorem 2.17)\textbf{).}}
\label{th:chow}
Let $(M_n)_{n \in \mathbb{N}^{\ast}}$ be a real valued martingale with respect to some filtration $\mathcal{F}=(\mathcal{F}_n)_{n \in \mathbb{N}}$. Then
\begin{align*}
\quad \lim\limits_{n \to +\infty} M_n =M_{\infty} \in \mathbb{R} \quad  a.s. & \quad   \mbox{on the event} \\
&\bigcup_{r \in [0,1]} \Big \{ \sum_{n=1}^{\infty} \mathbb{E} [ \vert M_n -M_{n-1} \vert^{1+r} \vert \mathcal{F}_{n-1} ] < + \infty \Big \}.
\end{align*}
\end{mytheo}

Now, we give a general CLT result from \cite{Hall_Heyde_1980} (Corollary 3.1) which applies to martingale arrays.  

\begin{myprop}
\label{prop:CLT_Hall_Heyde}
Let $(\tilde{M}_{k,n})_{k \in \{1,..,n\},n \in \mathbb{N}}$ be a $\mathbb{R}$-valued martingale array and define $\mathcal{F}^{\tilde{M}}_{k,n} = \sigma(\tilde{M}_{i,n},i \in \{0, \ldots,k \})$.We assume that $(\tilde{M_n})_{n \in \mathbb{N}}$ satisfies the Lindeberg condition:
\begin{align}
\label{hyp:Lindeberg_CLT}
\forall \mathfrak{E} >0 , \quad  \lim_{n \to \infty} \sum_{k=0}^{n-1}\mathbb{E}[  \vert \tilde{M}_{k+1,n}-\tilde{M}_{k,n} \vert^2 \mathds{1}_{ \vert \tilde{M}_{k+1,n}-\tilde{M}_{k,n} \vert > \mathfrak{E} } \vert \mathcal{F}^{\tilde{M}}_{k,n} ] \overset{\mathbb{P}}{=} 0
\end{align}
and that

\begin{align*}
 \lim_{n \to \infty} \sum_{k=0}^{n-1}\mathbb{E}[  \vert \tilde{M}_{k+1,n}-\tilde{M}_{k,n} \vert^2  \vert \mathcal{F}^{\tilde{M}}_{k,n} ] \overset{\mathbb{P}}{=}  \zeta_{\tilde{M}}^2
\end{align*}
with $ \zeta_{\tilde{M}}^2$ an almost sure finite random variable. Then

\begin{align*}
\lim_{n \to \infty} \tilde{M}_{n,n}  \overset{law}{=} \tilde{\mathcal{N}}(\zeta_{\tilde{M}}^2 ),
\end{align*}
where $\tilde{\mathcal{N}}(\zeta_{\tilde{M}}^2 )$ is a random variable with Laplace transform  $\mathbb{E}[\exp ( v \tilde{\mathcal{N}}(\zeta_{\tilde{M}}^2 ) ] = \mathbb{E}[\exp(v^2 \zeta_{\tilde{M}}^2  /2 ))]$ for every $v \in \mathbb{R}$.
\end{myprop}

\subsubsection{The $q$-order ergodic CLT}
When we consider the $q$-order weak approximation $(\overline{X}_{\Gamma_n})_{n\in \mathbb{N}}$ of a Feller process $(X_t)_{t \geqslant 0}$, it can be possible to obtain convergence of some weighted empirical measures. Moreover, the rate of convergence can be improved as $q$ grows. A crucial point to obtain a faster rate is to consider a specific weight sequence when we build the weighted empirical measures (\ref{eq:def_weight_emp_meas}). We begin by a alternative result which is crucial for the choice of the weights and for the proof of the q-order CLT established in Theorem \ref{th:conv_gnl_ordre_q}.

\begin{myprop}
\label{prop:dvpt_Af}
Let $q \in \mathbb{N}^{\ast}$ such that $\boldsymbol{\lambda}^{q}$ is not empty and $n_{q} \in \mathbb{N}$ such that $\gamma_{n_{q}+1+e}=\gamma_{n_{q}+1}$ for $e\in \{0,\ldots,(\mathscr{C}(q)-2)_+\}$. Considering $(c_{e}^{q})_{e \in \{0,\ldots,\mathscr{C}(q)-1\}}$ defined in (\ref{def:coeff_c}), then, for every $f \in \DomA$,

\begin{align}
\label{eq:dvpt_Af}
\gamma_{n_{q}+1}  \sum_{e=0}^{\mathscr{C}(q)-1}  c_{e}^{q} Af(\overline{X}_{\Gamma_{n_{q}+e}}) =&\sum_{e=0}^{(\mathscr{C}(q)-2)_+}  f(\overline{X}_{\Gamma_{n_{q}+e+1}})-f(\overline{X}_{\Gamma_{n_{q}+e}})\\
& \quad + B_q f (\overline{X}_{\Gamma_{n_{q}+e}},\overline{X}_{\Gamma_{n_{q}+e+1}},\gamma_{n_{q}+1},e)  \nonumber \\
& \quad  + \mathcal{R}_{q} f(\overline{X}_{\Gamma_{n_{q}+e}},\gamma_{n_{q}+1},e) . \nonumber
\end{align}

with $B_q f$ defined in (\ref{hyp:incr_X_Lyapunov_vitesse}) and $\mathcal{R}_{q} f$ defined in (\ref{hyp:rate_erreur_tems_cours_fonction_test_reg}), as soon as those quantities are well defined.

\end{myprop}

\begin{proof}
We prove the result by recurrence on $q$. The case $q=1$ follows directly from the definitions of $B_1 f$ and $\mathcal{R}_{1} f$ remembering that $c^{1}_{0}=1$. Assume that $q \geqslant 2$.
\\

\noindent
\textbf{Step 1.}
Since $\gamma_{n_{q}+1+e}=\gamma_{n_{q}+1}$ for $e\in \{0,\ldots,\mathscr{C}(q)-2\}$, then for every $l \in \mathbb{N}^{\ast}$ with $l \leqslant \mathscr{C}(q)-1-e$, we have,
\begin{align*}
\gamma_{n_{q}+1}^{l}A^{q+l}f(\overline{X}_{\Gamma_{n_{q}+e}})=&\sum_{j=0}^{l} (-1)^{l-j}\binom{l}{j}A^{q}f(\overline{X}_{\Gamma_{n_{q}+e+j}}) \\
& + \sum_{i=0}^{l-1} \sum_{j=0}^{l-i-1} (-1)^{l-i-j-1}\binom{l-i-1}{j} \gamma_{n_{q}+1}^{i}B_{1}A^{q+i}f(\overline{X}_{\Gamma_{n_{q}+e+j+1}},\overline{X}_{\Gamma_{n_{q}+e+j}},\gamma_{n_{q}+1})\\
& +\sum_{i=0}^{l-1} \sum_{j=0}^{l-i-1} (-1)^{l-i-j-1}\binom{l-i-1}{j} \gamma_{n_{q}+1}^{i} \mathcal{R}_{1}A^{q+i}f(\overline{X}_{\Gamma_{n_{q}+e+j}},\gamma_{n_{q}+1}).
\end{align*}
In particular, for $ \mathscr{C}(q) \geqslant q$ and $\lambda=(\lambda^{1},\ldots,\lambda^{\mathscr{C}(q)-2}) \in \boldsymbol{\lambda}^{q}  \subset \mathbb{R}^{\mathscr{C}(q)-2} \times \mathbb{R}^{\mathscr{C}(q)-3} \times \ldots \times \mathbb{R}$ 
\begin{align*}
 \sum_{e=0}^{\mathscr{C}(q)-2}& f(\overline{X}_{\Gamma_{n_{q}+e+1}})-f(\overline{X}_{\Gamma_{n_{q}+e}})+B_1 f (\overline{X}_{\Gamma_{n_{q}+e+1}},\overline{X}_{\Gamma_{n_{q}+e}},\gamma_{n_{q}+1}) - \sum_{e=0}^{\mathscr{C}(q)-2}   \sum_{i=1}^{q-1} \frac{\gamma_{n_{q}+1}^i}{i !}A^if(\overline{X}_{\Gamma_{n_{q}+e}}) \\
= & \sum_{e=0}^{\mathscr{C}(q)-2} \gamma_{n_{q}+1}^{q}  \frac{1}{q!} A^{q}f(\overline{X}_{\Gamma_{n_{q}+e}}) +\sum_{l=1}^{\mathscr{C}(q)-2-e} \gamma_{n_{q}+1}^{q+l} \frac{\sum_{h=l}^{\mathscr{C}(q)-2-e}\lambda^{h}_{e}}{(q+l)!} A^{q+l}f(\overline{X}_{\Gamma_{n_{q}+e}}) \\
 &+(1-\sum_{l=1}^{ \mathscr{C}(q)-2-e} \lambda^{l}_{e})  \tilde{\mathcal{R}}_{q} f(\overline{X}_{\Gamma_{n_{q}+e}},\gamma_{n_{q}+1})  + \sum_{l=1}^{ \mathscr{C}(q)-2-e}  \lambda^{l}_{e}  \tilde{\mathcal{R}}_{q+l} f(\overline{X}_{\Gamma_{n_{q}+e}},\gamma_{n_{q}+1})  \\
=& \sum_{e=0}^{ \mathscr{C}(q)-2} \gamma_{n_{q}+1}^{q}  \frac{1}{q!} A^{q}f(\overline{X}_{\Gamma_{n_{q}+e}})  +\sum_{l=1}^{ \mathscr{C}(q)-2-e} \frac{\sum_{h=l}^{ \mathscr{C}(q)-2-e}\lambda^{h}_{e}}{(q+l)!} \sum_{j=0}^{l} (-1)^{l-j}\binom{l}{j} \gamma_{n_{q}+1}^{q}A^{q}f(\overline{X}_{\Gamma_{n_{q}+e+j}})  \\
&+\sum_{l=1}^{ \mathscr{C}(q)-2-e} \frac{\sum_{h=l}^{\mathscr{C}(q)-2-e}\lambda^{h}_{e}}{(q+l)!}  \sum_{i=0}^{l-1} \sum_{j=0}^{l-i} (-1)^{l-i-j-1}\binom{l-i-1}{j} \gamma_{n_{q}+1}^{q+i} \mathcal{R}_{1}A^{q+i}f(\overline{X}_{\Gamma_{n_{q}+e+j}},\gamma_{n_{q}+1}) \\
&+\sum_{l=1}^{\mathscr{C}(q)-2-e} \frac{\sum_{h=l}^{\mathscr{C}(q)-2-e}\lambda^{h}_{e}}{(q+l)!}  \gamma_{n_{q}+1}^{q+1}  \sum_{i=0}^{l-1} \sum_{j=0}^{l-i-1} (-1)^{l-i-j-1}\binom{l-i-1}{j} \gamma_{n_{q}+1}^{q+i}  B_{1}A^{q+i}f(\overline{X}_{\Gamma_{n_{q}+e+j+1}},\overline{X}_{\Gamma_{n_{q}+e+j}},\gamma_{n_{q}+1}) \\
& +(1-\sum_{l=1}^{\mathscr{C}(q)-2-e} \lambda^{l}_{e})  \tilde{\mathcal{R}}_{q} f(\overline{X}_{\Gamma_{n_{q}+e}},\gamma_{n_{q}+1})  + \sum_{l=1}^{\mathscr{C}(q)-2-e}  \lambda^{l}_{e}  \tilde{\mathcal{R}}_{q+l} f(\overline{X}_{\Gamma_{n_{q}+e}},\gamma_{n_{q}+1}) .
\end{align*}

\noindent \textbf{Step 2.} Let us prove that for $p \in \{2,\ldots,q\}$

\begin{align}
\label{eq:dvpt_cumul}
\sum_{e=0}^{\mathscr{C}(q)-2} & \sum_{i=p}^{q} \frac{\gamma_{n_{q}+1}^i}{i !}  A^if(\overline{X}_{\Gamma_{n_{q}+e}}) +\sum_{l=1}^{\mathscr{C}(q)-2-e} \frac{\sum_{h=l}^{\mathscr{C}(q)-2-e}\lambda^{h}_{e}}{(q+l)!} \sum_{j=0}^{l} (-1)^{l-j}\binom{l}{j} \gamma_{n_{q}+1}^{q}A^{q}f(\overline{X}_{\Gamma_{n_{q}+e+j}})    \nonumber \\
    =& \gamma_{n_{q}+1}^p \sum_{u=0}^{\mathscr{C}(q)-\mathscr{C}(q-p+1)}  \zeta_u^{q,p}(\lambda)( \sum_{e=0}^{\mathscr{C}(q-p+1)-1} c_{e}^{q-p+1 }A^pf(\overline{X}_{\Gamma_{n_{q}+e+u}})  ) +D_{q,q-p},
\end{align}
 
with for $b \in \{0,\ldots,q-2\}$,

\begin{align*}
  D_{q,b}=&   \sum_{l=1}^{b}  \gamma_{n_{q}+1}^{q-l} \sum_{u=0}^{\mathscr{C}(b)-\mathscr{C}(l)}  \zeta_u^{q,q-l+1}(\lambda)  \sum_{e=0}^{(\mathscr{C}(l)-2)_+} B_{l}A^{q-l}f(\overline{X}_{\Gamma_{n_{q}+e+u}},\overline{X}_{\Gamma_{n_{q}+e+u+1}}, \gamma_{n_{q}+1},e   ) \\
& + \gamma_{n_{q}+1}^{q-l} \sum_{u=0}^{\mathscr{C}(b)-\mathscr{C}(l)}  \zeta_u^{q,q-l+1}(\lambda)  \sum_{e=0}^{(\mathscr{C}(l)-2)_+} \mathcal{R}_{l}A^{q-l}f(\overline{X}_{\Gamma_{n_{q}+e+u}},\gamma_{n_{q}+1},e ) .
\end{align*}

Notice that $c^{1}_{0}=1$,  and since

\begin{align*}
\zeta_{u}^{q,q}(\lambda) = \frac{1}{q!}\mathds{1}_{u<\mathscr{C}(q)-1}+\sum_{i=0}^{\mathscr{C}(q)-3}\sum_{l=1}^{\mathscr{C}(q)-2-i} \frac{\sum_{h=l}^{\mathscr{C}(q)-2-i}\lambda^{h}_{i}}{(q+l)!} (-1)^{l-u+i}\binom{l}{u-i} \mathds{1}_{0 \leqslant u-i \leqslant l}
\end{align*}

and $D_{q,0}= 0$, then (\ref{eq:dvpt_cumul}) is true for $p=q$. Assume that (\ref{eq:dvpt_cumul}) is true for some $p \in \{3,\ldots,q\}$, and let us show that it still holds with $p$ replaced by $p-1$.  We first apply (\ref{eq:dvpt_cumul}) and then (\ref{eq:dvpt_Af}) with $f$ replaced by $A^{p-1}f$ and $q$ replaced by $(q-p+1)$ which yields

\begin{align*}
\sum_{e=0}^{\mathscr{C}(q)-2} & \sum_{i=p}^{q-1} \frac{\gamma_{n_{q}+1}^i}{i !}  A^if(\overline{X}_{\Gamma_{n_{q}+e}})+ \frac{1}{q!} \gamma_{n_{q}+1}^{q} A^{q}f(\overline{X}_{\Gamma_{n_{q}+e}}) \nonumber \\
& +\sum_{l=1}^{\mathscr{C}(q)-2-e} \frac{\sum_{h=l}^{\mathscr{C}(q)-2-e}\lambda^{h}_{e}}{(q+l)!} \sum_{j=0}^{l} (-1)^{l-j}\binom{l}{j} \gamma_{n_{q}+1}^{q}A^{q}f(\overline{X}_{\Gamma_{n_{q}+e+j}})   \\
=& \gamma_{n_{q}+1}^{p-1} \sum_{u=0}^{\mathscr{C}(q)-\mathscr{C}(q-p+1)}  \zeta_u^{q,p}(\lambda) \sum_{e=0}^{(\mathscr{C}(q-p+1)-2)_{+}}  (A^{p-1}f(\overline{X}_{\Gamma_{n_{q}+e+u+1}})  - A^{p-1}f(\overline{X}_{\Gamma_{n_{q}+e+u}})   ) \\
& + \gamma_{n_{q}+1}^{p-1} \sum_{u=0}^{\mathscr{C}(q)-\mathscr{C}(q-p+1)}  \zeta_u^{q,p}(\lambda) \sum_{e=0}^{(\mathscr{C}(q-p+1)-2)_{+}} B_{q-p+1}A^{p-1}f(\overline{X}_{\Gamma_{n_{q}+e+u}}, \overline{X}_{\Gamma_{n_{q}+e+u+1}},\gamma_{n_{q}+1},e   ) \\
& + \gamma_{n_{q}+1}^{p-1} \sum_{u=0}^{\mathscr{C}(q)-\mathscr{C}(q-p+1)}  \zeta_u^{q,p}(\lambda) \sum_{e=0}^{(\mathscr{C}(q-p+1)-2)_{+}}\mathcal{R}_{q-p+1}A^{p-1}f(\overline{X}_{\Gamma_{n_{q}+e+u}},\gamma_{n_{q}+1},e )+D_{q,q-p}
\\
=& \gamma_{n_{q}+1}^{p-1}  \sum_{u=0}^{\mathscr{C}(q)-\mathscr{C}(q-p+1)}  \zeta_u^{q,p}(\lambda) \sum_{e=0}^{(\mathscr{C}(q-p+1))_{+}}  (A^{p-1}f(\overline{X}_{\Gamma_{n_{q}+e+u+1}})  - A^{p-1}f(\overline{X}_{\Gamma_{n_{q}+e+u}})   ) + D_{q,q-p+1}.
\end{align*}

Moreover, observe that, with notation from (\ref{hyp:rate_erreur_tems_cours_fonction_test_reg}) and using that $\zeta_{\mathscr{C}(q)-1} ^{q,q}(\lambda)=0$,
\begin{align*}
& \sum_{u=0}^{\mathscr{C}(q)-\mathscr{C}(q-p+1)}  \zeta_u^{q,p}(\lambda)  \sum_{e=0}^{(\mathscr{C}(q-p+1)-2)_+} (A^{p-1}f(\overline{X}_{\Gamma_{n_q+e+u+1}})  - A^{p-1}f(\overline{X}_{\Gamma_{n_q+e+u}})   ) +\sum_{e=0}^{\mathscr{C}(q)-2}   \frac{1}{(p-1) !}  A^{p-1}f(\overline{X}_{\Gamma_{n_q+e}})\\
& = \sum_{u=0}^{(\mathscr{C}(q)-\mathscr{C}(q-p+1))\wedge (\mathscr{C}(q)-2)}  \zeta_u^{q,p}(\lambda)  \sum_{e=0}^{(\mathscr{C}(q-p+1)-2)_+} (A^{p-1}f(\overline{X}_{\Gamma_{n_q+e+u+1}})  - A^{p-1}f(\overline{X}_{\Gamma_{n_q+e+u}})   ) \\
&+\sum_{e=0}^{\mathscr{C}(q)-2}   \frac{1}{(p-1) !}  A^{p-1}f(\overline{X}_{\Gamma_{n_q+e}})\\
 &=\sum_{e=0}^{\mathscr{C}(q)-1} \Big(   \frac{1}{(p-1) !} \mathds{1}_{e < \mathscr{C}(q)-1 } +\zeta_{e-(\mathscr{C}(q-p+1)-2)_+-1}^{q,p}(\lambda) \mathds{1}_{(\mathscr{C}(q)-\mathscr{C}(q-p+1)+1) \wedge (\mathscr{C}(q)-1) \geqslant e-(\mathscr{C}(q-p+1)-2)_+  \geqslant 1} \\
 & \qquad -\zeta_{e}^{q,p}(\lambda) \mathds{1}_{e < (\mathscr{C}(q)-\mathscr{C}(q-p+1)+1) \wedge (\mathscr{C}(q)-1) } \Big) A^{p-1}f(\overline{X}_{\Gamma_{n_q+e}})\\
&=\sum_{e=0}^{\mathscr{C}(q)-1}N^{q,p}(\zeta^{q,p}(\lambda))_{e} A^{p-1}f(\overline{X}_{\Gamma_{n_q+e}}),
\end{align*}
and 
\begin{align*}
\sum_{u=0}^{\mathscr{C}(q)-\mathscr{C}(q-p+2)} & \zeta_u^{q,p-1}(\lambda) \sum_{e=0}^{\mathscr{C}(q-p+2)-1} c_{e}^{q-p+2}A^{p-1}f(\overline{X}_{\Gamma_{n_q+e+u}}) 
\\
=&\sum_{e=0}^{\mathscr{C}(q)-1} \left( \sum_{u=0}^{\mathscr{C}(q)-\mathscr{C}(q-p+2)} \zeta^{q,p-1}_u(\lambda)  c_{e-u}^{q-p+2} \mathds{1}_{e-u \in \{ 0,\ldots,\mathscr{C}(q-p+2)-1\} }\right) A^{p-1}f(\overline{X}_{\Gamma_{n_q+e}}) \\
=&\sum_{e=0}^{\mathscr{C}(q)-1} (M^{q,p} \zeta^{q,p-1}(\lambda))_e  A^{p-1}f(\overline{X}_{\Gamma_{n_q+e}}).
\end{align*}

Using that $M^{q,p}\zeta^{q,p-1}(\lambda)=N^{q,p}(\zeta^{q,p}(\lambda))$ and gathering all the terms together completes the proof of \textbf{Step 2}.\\

\noindent \textbf{Step 3}. We are now in a position to complete the proof. We apply (\ref{eq:dvpt_Af}) with $f$ replaced by $Af$ and $q$ replaced by  $q-1$, and for $u \in \{0,1\}$, we obtain, 
\begin{align*}
\gamma_{n_{q}+1}  \sum_{j=0}^{\mathscr{C}(q-1)-1} c_{j}^{q-1}A^2f(\overline{X}_{\Gamma_{n_{q}+u+j}})  = &  \sum_{j=0}^{(\mathscr{C}(q-1)-2)_+}  Af(\overline{X}_{\Gamma_{n_{q}+u+j+1}})-Af(\overline{X}_{\Gamma_{n_{q}+u+j}})\\
& \quad + B_{q-1} Af (\overline{X}_{\Gamma_{n_{q}+u+j}},\overline{X}_{\Gamma_{n_{q}+u+j+1}},\gamma_{n_{q}+1},j)  \nonumber \\
& \quad  + \mathcal{R}_{q-1} Af(\overline{X}_{\Gamma_{n_{q}+u+j}},\gamma_{n_{q}+1},j) . \nonumber
\end{align*}

Combining this expansion with the decomposition established in \textbf{Step 1} and the result from \textbf{Step 2} for $p=2$  yields

\begin{align*}
 \gamma_{n_{q}+1}  & \sum_{e=0}^{\mathscr{C}(q)-2}  Af(\overline{X}_{\Gamma_{n_{q}+e}}) +  \gamma_{n_{q}+1}    \sum_{u=0}^{\mathscr{C}(q)-\mathscr{C}(q-1)}\zeta_u^{q,2}(\lambda)  \sum_{j=0}^{(\mathscr{C}(q-1)-2)_+}  Af(\overline{X}_{\Gamma_{n_{q}+j+u+1}})-Af(\overline{X}_{\Gamma_{n_{q}+j+u}}) \\
=&\sum_{e=0}^{\mathscr{C}(q)-2} \left. f(\overline{X}_{\Gamma_{n_{q}+e+1}})-f(\overline{X}_{\Gamma_{n_{q}+e}})+B_1 f (\overline{X}_{\Gamma_{n_{q}+e}},\overline{X}_{\Gamma_{n_{q}+e+1}},\gamma_{n_{q}+1}) \right. \\
&-\sum_{l=1}^{\mathscr{C}(q)-2-e} \frac{\sum_{h=l}^{\mathscr{C}(q)-2-e}\lambda^{h}_{e}}{(q+l)!}  \gamma_{n_{q}+1}^{q+1}  \sum_{i=0}^{l-1} \sum_{j=0}^{l-i-1} (-1)^{l-i-j-1}\binom{l-i-1}{j} \gamma_{n_{q}+1}^{q+i}  B_{1}A^{q+i}f(\overline{X}_{\Gamma_{n_{q}+e+j+1}},\overline{X}_{\Gamma_{n_{q}+e+j}},\gamma_{n_{q}+1}) \\
&-  \gamma_{n_{q}+1}   \sum_{j=0}^{(\mathscr{C}(q-1)-2)_+}  \sum_{u=0}^{\mathscr{C}(q)-\mathscr{C}(q-1)}  \zeta_u^{q,2}(\lambda) B_{q-1} Af (\overline{X}_{\Gamma_{n_{q}+u+j}},\overline{X}_{\Gamma_{n_{q}+u+j+1}},\gamma_{n_{q}+1},j)  \nonumber \\
&-\sum_{l=1}^{\mathscr{C}(q)-2-e} \frac{\sum_{h=l}^{\mathscr{C}(q)-2-e}\lambda^{h}_{e}}{(q+l)!}  \sum_{i=0}^{l-1} \sum_{j=0}^{l-i} (-1)^{l-i-j-1}\binom{l-i-1}{j} \gamma_{n_{q}+1}^{q+i} \mathcal{R}_{1}A^{q+i}f(\overline{X}_{\Gamma_{n_{q}+e+j}},\gamma_{n_{q}+1}) \\
& -(1-\sum_{l=1}^{\mathscr{C}(q)-2-e} \lambda^{l}_{e})  \tilde{\mathcal{R}}_{q} f(\overline{X}_{\Gamma_{n_{q}+e}},\gamma_{n_{q}+1})  - \sum_{l=1}^{\mathscr{C}(q)-2-e}  \lambda^{l}_{e}  \tilde{\mathcal{R}}_{q+l} f(\overline{X}_{\Gamma_{n_{q}+e}},\gamma_{n_{q}+1})  \\
& -  \gamma_{n_{q}+1}   \sum_{j=0}^{(\mathscr{C}(q-1)-2)_+}   \sum_{u=0}^{\mathscr{C}(q)-\mathscr{C}(q-1)}  \zeta_u^{q,2}(\lambda)   \mathcal{R}_{q-1} Af(\overline{X}_{\Gamma_{n_{q}+u+j}},\gamma_{n_{q}+1},j) . \nonumber \\
& -D_{q,q-2} .
\end{align*}

To complete the proof, it simply boils down to apply the definition of $(c_{j}^{q})_{j \in \{0,\ldots,q-1\}}$ (see (\ref{def:coeff_c})),  $B_q f$ (see (\ref{hyp:incr_X_Lyapunov_vitesse})) and $\mathcal{R}_{q} f$ (see (\ref{hyp:rate_erreur_tems_cours_fonction_test_reg})).

\end{proof}

We are now in a position to state the $q$-order ergodic CLT. Before that, we introduce the step size and the weight sequences. In particular, we assume that
\begin{align}
\label{def:weight_q_order}
\forall n \in \mathbb{N} , e \in \{0,\ldots,(\mathscr{C}(q)-2)_+\},& \nonumber\\
 \gamma_{(\mathscr{C}(q)-1)n+1+e}=&\gamma_{(\mathscr{C}(q)-1)n+1}, \\
 \eta_{(\mathscr{C}(q)-1)n+1+e}=&C_{\gamma,\eta}\left( c_{e}^{q} \gamma_{(\mathscr{C}(q)-1)n+1}+ \mathds{1}_{e=0}c_{q-1}^{q} \gamma_{(\mathscr{C}(q)-1)(n-1)+1} \right)  , \nonumber
\end{align}
with $C_{\gamma,\eta} \in \mathbb{R}^{\ast}$ and the convention $\gamma_{-l}=0$ for $l \in \mathbb{N}^{\ast}$.
\begin{mytheo}
\label{th:conv_gnl_ordre_q}
Let $q \in \mathbb{N}^{\ast}$ such that $\boldsymbol{\lambda}^{q}$ is not empty, let $F \subset \{f,f:(E, \mathcal{B}(E)) \to (\mathbb{R}, \mathcal{B}(\mathbb{R}) ) , Af \in \mathcal{C}_b(E)  \}$, $g:E \to \mathbb{R}_+$ a Borel function, $\epsilon_{\mathfrak{X}},\epsilon_{\mathcal{G}\mathcal{C}}: \mathbb{R}_+ \to \mathbb{R}_+$ be two increasing functions, let $\tilde{\eta}_q: \mathbb{R}_+ \times \{0,\ldots,(q-2)_+\} \to \mathbb{R}_+$ and let $\mathfrak{M}_q$ and $\mathfrak{V}$ be two linear operators defined on $F$. Finally let $\gamma_n$ and $\eta_{n}$, $n \in \mathbb{N}$, be the time step and the weight sequences satisfying (\ref{def:weight_q_order}).\\

 Assume that  $\mathcal{E}_q(F,\tilde{A},A, \mathfrak{M}_q,\tilde{\eta}_q)$ (see (\ref{hyp:rate_erreur_tems_cours_fonction_test_reg})), $\mathcal{GC}_{Q,q}(F,g,\rho,\epsilon_{\mathscr{X}},\epsilon_{\mathcal{G}\mathcal{C}},\mathfrak{V})$ (see (\ref{hyp:incr_X_Lyapunov}) and (\ref{hyp:Lindeberg_CLT})) and $\mathcal{S}\mathcal{W}_{\mathcal{G}\mathcal{C}, \gamma}(g, \rho , \epsilon_{\mathscr{X}}, \epsilon_{\mathcal{G}\mathcal{C}})$ (see (\ref{hyp:step_weight_I_gen_chow_rate})) hold. \\
 
 Then, for every $f \in F$, we have the following properties:

\begin{enumerate}[label=\textbf{\Alph*.}]

\item If $\lim_{n \to \infty} \sqrt{H_{\epsilon_{\mathfrak{X}},n}}  /  H_{\tilde{\eta}_{q},n} =+\infty$, then
\begin{align*}
\lim_{n \to \infty} \frac{H_n}{ C_{\gamma,\eta} \sqrt{H_{\epsilon_{\mathfrak{X}},n}}}   \nu ^{\eta}_n(Af) \overset{law}{=} \mathcal{N}(0, \nu( \mathfrak{V}f )).
\end{align*}
\item \label{eq:TCL_mes_pond_moy_var} 
If $\lim_{n \to \infty}   \sqrt{H_{\epsilon_{\mathfrak{X}},n}}  / H_{\tilde{\eta}_{q},n} =\hat{l} \in \mathbb{R}^{\ast}_+$, then
\begin{align*}
\lim_{n \to \infty}  \frac{H_n}{ C_{\gamma,\eta} \sqrt{H_{\epsilon_{\mathfrak{X}},n}}}  \nu ^{\eta}_n(A f) \overset{law}{=} \mathcal{N}(\hat{l}^{-1} \nu(\mathfrak{M}_q f), \nu( \mathfrak{V}f )).
\end{align*}

\item If $\lim_{n \to \infty}   \sqrt{H_{\epsilon_{\mathfrak{X}},n}} /  H_{\tilde{\eta}_{q},n}=0$, then
\begin{align*}
\lim_{n \to \infty} \frac{H_n}{C_{\gamma,\eta} H_{\tilde{\eta}_{q},n}} \nu ^{\eta}_n(Af) \overset{\mathbb{P}}{=} \nu(\mathfrak{M}_q f)
\end{align*}
Moreover, when $\mathfrak{V}=0$ this convergence is almost sure.

\end{enumerate}

\end{mytheo}

\begin{remark}
\label{rmk:rate}
Notice that if we take $\gamma_n=1/n^{\xi}$ (for $n=n_{q}(q-1)$ with $n_q \in \mathbb{N}$), $\xi \in (0,1/(q+1))$ and denote

 \begin{eqnarray*}
\forall n \in \mathbb{N}^{\ast}, \qquad \mathfrak{r}_{q,n}= 
& \quad  \left\{
    \begin{array}{l}
   \frac{H_n}{ \sqrt{H_{\epsilon_{\mathfrak{X}},n}}} = \sqrt{ \Gamma_{n} }  \quad \mbox{if} \quad \lim_{n \to \infty}   \sqrt{ \Gamma_{n} }  /   H_{\tilde{\eta}_{q},n}=+\infty,\\
  \frac{H_n}{  \sqrt{H_{\epsilon_{\mathfrak{X}},n}}}  =\sqrt{ \Gamma_{n} }   \quad \mbox{if} \quad   \lim_{n \to \infty} \sqrt{ \Gamma_{n} }  /   H_{\tilde{\eta}_{q},n}  =\hat{l} ,\\
 \frac{H_n}{ H_{\tilde{\eta}_{q},n}} \quad \mbox{if} \quad \lim_{n \to \infty}   \sqrt{ \Gamma_{n} }  /   H_{\tilde{\eta}_{q},n}=0,
    \end{array}
\right.
\end{eqnarray*}
the rate of convergence of $(\nu ^{\eta_q}_n(A f))_{n \in \mathbb{N}^{\ast}}$ in the $q$-order ergodic CLT, then we have 
\begin{align*}
\mathfrak{r}_{q,n} \underset{n \to + \infty}{\sim} C n^{(q\xi) \wedge (1/2-\xi/2)}.
\end{align*}
The highest rate of convergence is thus achieved for $\xi=1/(2q+1)$ and is given by $\mathfrak{r}_{q,n} \underset{n \to + \infty}{\sim} C n^{q/(2q+1)}$.

\end{remark}

\begin{proof}
 We assume $q\geqslant2$. The case $q=1$ is similar but simpler so we leave it out.\\
 
 \noindent
\textbf{Step 1.} Let $n \in \mathbb{N}$. We begin by noticing that the following decomposition holds

\begin{align*}
 \nu ^{\eta}_n(A f) =& \frac{1}{H_n} \sum_{k=1}^n \eta_{k}A f(\overline{X}_{\Gamma_{k-1}}) \\
 =& \frac{C_{\gamma,\eta}}{H_n} \sum_{k=0}^{N_{q,n} } \gamma_{(q-1)k+1}\sum_{e=0}^{\mathscr{C}(q)-1} c_{e}^{q} Af(\overline{X}_{\Gamma_{(q-1)k+e}}) \\
 &+ \frac{1}{H_n} \sum_{k=(q-1)(N_{q,n}+1)+1}^{n} \eta_k A f(\overline{X}_{\Gamma_{k-1}}) \\
 &- \frac{ C_{\gamma,\eta} \gamma_{(q-1)N_{q,n}+1} c_{q-1}^{q}}{H_n} Af(\overline{X}_{\Gamma_{(q-1)(N_{q,n}+1)}}) 
\end{align*}

with the notation $N_{q,n}=\lfloor (n-1)/(\mathscr{C}(q)-1) \rfloor-1.$ Since $Af$ is a bounded function, the second and third terms of the $r.h.s.$ of the above equation multiplied by $\frac{H_n}{ C_{\gamma,\eta} \sqrt{H_{\epsilon_{\mathfrak{X}},n}}} $ or $\frac{H_n}{C_{\gamma,\eta} H_{\tilde{\eta}_{q},n}} $ both converge to zero when $n$ tends to infinity. We study the first term of the $r.h.s.$ of the above equation.
By Proposition \ref{prop:dvpt_Af} (with $n_{q}=(q-1)k$),

\begin{align*}
\gamma_{(q-1)k+1}  \sum_{e=0}^{\mathscr{C}(q)-1}  c_{e}^{q} Af(\overline{X}_{\Gamma_{(q-1)k+e}}) =&\sum_{e=0}^{(\mathscr{C}(q)-2)_+}  f(\overline{X}_{\Gamma_{(q-1)k+e+1}})-f(\overline{X}_{\Gamma_{(q-1)k+e}})\\
& \quad + B_q f (\overline{X}_{\Gamma_{(q-1)k+e}},\overline{X}_{\Gamma_{(q-1)k+e+1}},\gamma_{(q-1)k+1},e)  \nonumber \\
& \quad  + \mathcal{R}_{q} f(\overline{X}_{\Gamma_{(q-1)k+e}},\gamma_{(q-1)k+1},e) . \nonumber
\end{align*}

\noindent
\textbf{Step 2.} In this part, we prove that

\begin{align*}
\lim_{n \to \infty} \frac{1}{\sqrt{H_{\epsilon_{\mathfrak{X}},n}}} \sum_{k=0}^{N_{q,n} } \sum_{e=0}^{(\mathscr{C}(q) -2)_{+}}  B_q f (\overline{X}_{\Gamma_{(q-1)k+e}},\overline{X}_{\Gamma_{(q-1)k+e+1}},\gamma_{(q-1)k+1},e) \overset{law}{=} \mathcal{N}(0, \nu( \mathfrak{V}f )).
\end{align*}

From Proposition \ref{prop:CLT_Hall_Heyde}, since (\ref{hyp:Lindeberg_CLT_scheme}) holds and $\lim_{n \in \mathbb{N}^{\ast}} \nu_n^{\epsilon_{\mathscr{X}}}( \mathfrak{V}f ,\omega )= \nu (\mathfrak{V}f ), \; \mathbb{P}-a.s.$, we have

\begin{align*}
\lim_{n \to \infty} \frac{1}{\sqrt{H_{\epsilon_{\mathfrak{X}},n}}} \sum_{k=1}^n 
\mathscr{X}_{f,k} 
\overset{law}{=} \mathcal{N}(0, \nu( \mathfrak{V}f )) 
\end{align*}

Notice that when $\mathfrak{V}=0$ the $l.h.s.$ of the above equation is $\mathbb{P}-a.s.$ equal to zero for every $f \in F$. Now, to obtain the convergence in law, we are going to show that $\mathbb{P}-a.s$, for every $f \in F$,
\begin{equation*}
. \quad \lim_{n \to + \infty}  \frac{1}{\sqrt{H_{\epsilon_{\mathfrak{X}},n}}}  \sum_{k=0}^{N_{q,n} } \sum_{e=0}^{(\mathscr{C}(q)-2)_+}  B_q f (\overline{X}_{\Gamma_{(q-1)k+e}},\overline{X}_{\Gamma_{(q-1)k+e+1}},\gamma_{(q-1)k+1},e)  -\mathscr{X}_{f,(q-1)k+1+e} =0.
\end{equation*}

This last result is a consequence of Kronecker's Lemma as soon as we prove the $a.s.$ convergence of the martingale $(M_n)_{n \in \mathbb{N}^{\ast}}$ defined by $M_0:=0$ and 

\begin{equation*}
M_n:=  \sum_{k=0}^{N_{q,n} } \sum_{e=0}^{(\mathscr{C}(q)-2)_+}  \frac{B_q f (\overline{X}_{\Gamma_{(q-1)k+e}},\overline{X}_{\Gamma_{(q-1)k+e+1}},\gamma_{(q-1)k+1},e) -\mathscr{X}_{f,(q-1)k+1+e} }{\sqrt{H_{\epsilon_{\mathfrak{X}},(q-1)k+e}}}.
\end{equation*}

From the Chow's theorem (see Theorem \ref{th:chow}), this $a.s.$ convergence is a direct consequence of the $a.s.$ finiteness of the series
\begin{equation*}
\sum_{k=1}^{n}  \mathbb{E}[\vert M_k-M_{k-1}\vert^{\rho} \vert \overline{X}_{\Gamma_{k-1}}  ] ,
\end{equation*}
which follows from $\mathcal{GC}_{Q,q}(F,g,\rho,\epsilon_{\mathfrak{X}},\epsilon_{\mathcal{G}\mathcal{C}},\mathfrak{V})$ (see (\ref{hyp:incr_X_Lyapunov})) together with $\mathcal{S}\mathcal{W}_{\mathcal{G}\mathcal{C}, \gamma}(g, \rho ,\epsilon_{\mathfrak{X}}, \epsilon_{\mathcal{G}\mathcal{C}})$ (see (\ref{hyp:step_weight_I_gen_chow_rate})). \\

\noindent
\textbf{Step 3.} To complete the proof, let us show that
\begin{align*}
\mathbb{P}-a.s. \quad \forall f \in F  \quad \lim_{n \to \infty} \frac{1}{H_{\tilde{\epsilon}_{q},n}} \sum_{k=0}^{N_{q,n} } \sum_{e=0}^{(\mathscr{C}(q)-2)_+}  \mathcal{R}_q f (\overline{X}_{\Gamma_{(q-1)k+e}}\gamma_{(q-1)k+1},e)  = \nu(\mathfrak{M}_q f).
\end{align*}

As a direct consequence of $\mathcal{E}_q(F,\tilde{A},A, \mathfrak{M}_q,\tilde{\eta}_q)$ (see (\ref{hyp:rate_erreur_tems_cours_fonction_test_reg})), since $\mathbb{P}-a.s.$, $\lim_{n \to \infty}\nu ^{\tilde{\eta}_q}_n(\mathfrak{M}_q f) = \nu(\mathfrak{M}_q f)$ for every $f \in F$, we only have to prove that 
 \begin{align*}
\mathbb{P}&-a.s. \quad \forall f \in F  \\
& \lim_{n \to \infty} \frac{1}{H_{\tilde{\epsilon}_{q},n}} \sum_{k=0}^{N_{q,n} } \sum_{e=0}^{(\mathscr{C}(q)-2)_+}  \mathcal{R}_q f (\overline{X}_{\Gamma_{(q-1)k+e}},\gamma_{(q-1)k+1},e) - \tilde{\eta}_{q,(q-1)k+1+e} \mathfrak{M}_q f(\overline{X}_{\Gamma_{(q-1)k+e}}) = 0.
\end{align*}

which holds as soon as 
 \begin{align}
 \label{eq:preuve_conv_ident_vitess}
\mathbb{P}-a.s. \quad \forall f \in F  \quad \lim_{n \to \infty} \frac{1}{H_{\tilde{\eta}_{q},n}} \sum_{k=1}^n  \tilde{\eta}_{q,k} \Lambda_{f,q}(\overline{X}_{\Gamma_{k-1}},\gamma_{k}) = 0,  
\end{align}

We recall that we have the following decomposition

\begin{align*}
\forall f \in F, \forall x \in E , \forall \gamma \in [0,\overline{\gamma}], \qquad \Lambda_{f,q}(x,\gamma)= \langle g (  x ) ,\tilde{\mathbb{E}} [\tilde{\Lambda}_{f,q}(x,\gamma)]  \rangle_{\mathbb{R}^l}
\end{align*}
with $g : (E, \mathcal{B}(E))\to \mathbb{R}_+^{l}$, $l \in \mathbb{N}^{\ast}$, a locally bounded Borel measurable function and $\tilde{\Lambda}_{f,w}:(E\times \mathbb{R}_+ \times \tilde{\Omega}, \mathcal{B}(E) \otimes \mathcal{B}(\mathbb{R}_+) \otimes \tilde{\mathcal{G}}) \to \mathbb{R}_+^{l}$ a measurable function such that  $\sup_{i \in \{1,\ldots,l \} ,x \in E,\gamma \in (0,\overline{\gamma}] } \tilde{\mathbb{E}}[\tilde{\Lambda}_{f,q,i}(x,\gamma) ]< + \infty$. Since for every $i \in \{1,\ldots,l\}$, $\sup_{n \in \mathbb{N}^{\ast}} \nu_n^{\tilde{\eta}_q}( g_i ,\omega )< + \infty$, $\mathbb{P}(d \omega)-a.s.$, (\ref{eq:preuve_conv_ident_vitess}) follows from the following result:\\

 Let $(\overline{x}_{_n})_{n \in \mathbb{N}} \in E^{\otimes \mathbb{N}} $. Assume that $\sup_{i \in \{1,\dots,l \}}\sup_{n \in \mathbb{N}^{\ast}} \frac{1}{H_{\tilde{\eta}_q  ,n}}\sum_{k=1}^n \tilde{\eta}_{q,k}   g_i(\overline{x}_{_{k-1}})<+ \infty$, then, for every $f \in F$, 
\begin{align*}
\lim\limits_{n\to \infty} \frac{1}{H_{\tilde{\eta}_q,n}}\sum_{k=1}^n  \tilde{\eta}_{q,k} \Lambda_{f,q}(\overline{x}_{_{k-1}},\gamma_k) =0.
\end{align*}
 In order to obtain this result, we are going to show that, for every $f \in F$, every $i \in \{1,\ldots, l \}$, and every $(\overline{x}_{n})_{n \in \mathbb{N}} \in E^{\otimes \mathbb{N}} $, then
\begin{equation*}
\tilde{\mathbb{P}}(d \tilde{\omega})-a.s. \quad \lim\limits_{n \to \infty} \frac{1}{H_{\tilde{\eta}_q  ,n}}\sum_{k=1}^n \tilde{\eta}_{q,k} \tilde{\Lambda}_{f,q,i}(\overline{x}_{{k-1}},\gamma_k,\tilde{\omega}) g_i (\overline{x}_{{k-1}}) =0  ,
\end{equation*}
and the result will follow from the Dominated Convergence theorem since for every $n \in \mathbb{N}^{\ast}$, 

\begin{align*}
 \frac{1}{H_{\tilde{\eta}_q ,n}}\sum_{k=1}^n  & \tilde{\eta}_{q,k}  \tilde{\Lambda}_{f,q,i}(\overline{x}_{{k-1}},\gamma_k,\omega) g_i (\overline{x}_{{k-1}}) \\
 &\leqslant \sup_{x \in E} \sup_{\gamma \in (0,\overline{\gamma}]}\tilde{\Lambda}_{f,2,i} (x,\gamma,\tilde{\omega})    \sup_{n \in \mathbb{N}^{\ast}} \frac{1}{H_{\tilde{\eta}_q ,n}}\sum_{k=1}^n \tilde{\eta}_{q,k} g_i(\overline{x}_{{k-1}})< +\infty .
\end{align*}
with $ \tilde{\mathbb{E} } [  \sup_{x \in E} \sup_{\gamma \in (0,\overline{\gamma}]}\tilde{\Lambda}_{f,q,i} (x,\gamma,\tilde{\omega})   ]< +\infty$ and $\sup_{n \in \mathbb{N}^{\ast}} \frac{1}{H_{\tilde{\eta}_q ,n}}\sum_{k=1}^n\tilde{\eta}_{q,k}  g_i(\overline{x}_{{k-1}})<+ \infty$. We fix $f \in F$, $i \in \{1,\ldots,N \}$ and $(\overline{x}_{n})_{N \in \mathbb{N}} \in E^{\otimes \mathbb{N}} $ and we assume that $\mathcal{E}_q(\tilde{A},A,\mathfrak{M}_q,\tilde{\eta}_q)$ \ref{hyp:erreur_tems_cours_fonction_test_reg_Lambda_representation_1} (see (\ref{hyp:erreur_temps_cours_fonction_test_reg_Lambda_representation_2_1})) holds for $\tilde{\Lambda}_{f,q,i}$. If instead $\mathcal{E}_q(\tilde{A},A,\mathfrak{M}_q)$ \ref{hyp:erreur_temps_cours_fonction_test_reg_Lambda_representation_2} (see (\ref{hyp:erreur_temps_cours_fonction_test_reg_Lambda_representation_2_2})) is satisfied, the proof is similar but simpler so we leave it to the reader.

Let $\underline{n}(\tilde{\omega}) :=\inf\{n \in \mathbb{N}^{\ast},\sup_{k \geqslant n} \gamma_{k} \leqslant \underline{\gamma}(\tilde{\omega}) \}$. By assumption $\mathcal{E}_q(F,\tilde{A},A, \mathfrak{M}_q,\tilde{\eta}_q)$ \ref{hyp:erreur_tems_cours_fonction_test_reg_Lambda_representation_1} (ii)(see (\ref{hyp:erreur_temps_cours_fonction_test_reg_Lambda_representation_2_2})), $\tilde{\mathbb{P}}(d\tilde{\omega})-a.s$, for every $R>0$, there exists $K_R(\tilde{\omega}) \in \mathcal{K}_E$ such that 
\begin{align*}
 \sup_{x \in K_{R}^c(\tilde{\omega}) }\sup_{\gamma \in (0, \underline{\gamma}(\tilde{\omega}) ]} \tilde{\Lambda}_{f,q,i}(x, \gamma,\tilde{\omega})<1/R.
 \end{align*}
  Moreover,
\begin{align*}
 \sup_{n \geqslant \underline{n}(\tilde{\omega}) } \frac{1}{H_{\tilde{\eta}_q ,n}} \sum_{k=\underline{n}(\tilde{\omega}) }^n \tilde{\eta}_{q,k} \tilde{\Lambda}_{f,2,i}(\overline{x}_{{k-1}},\gamma_k,\tilde{\omega}) & g (\overline{x}_{{k-1}})   \mathds{1}_{K_{R}^c(\tilde{\omega})}(\overline{x}_{k-1}) \\
  \leqslant  &  \sup_{x \in K_{R}^c(\tilde{\omega})} \sup_{\gamma \in (0,\underline{\gamma}(\tilde{\omega})]}  \tilde{\Lambda}_{f,q,i} (x,\gamma,\tilde{\omega})  \sup_{n \in \mathbb{N}^{\ast}} \frac{1}{H_{\tilde{\eta}_q  ,n}}\sum_{k=1}^n \tilde{\eta}_{q,k}  g_i(\overline{x}_{k-1}).
\end{align*}
We let $R$ tends to infinity and since $\sup_{n \in \mathbb{N}^{\ast}} \frac{1}{H_{\tilde{\eta}_q ,n}}\sum_{k=1}^n \tilde{\eta}_{q,k}  g_i(\overline{x}_{k-1})< + \infty$, the $l.h.s.$ of the above equation converges $\tilde{\mathbb{P}}(d\tilde{\omega})-a.s.$ to 0. Finally, since $\underline{n}(\tilde{\omega}) $ is $\tilde{\mathbb{P}}(d\tilde{\omega})-a.s.$ finite, we also have
\begin{align*}
 \tilde{\mathbb{P}}(d\tilde{\omega})-a.s. \quad \forall R>0, \quad\lim_{n \to \infty}\frac{1}{H_{\tilde{\eta}_q  ,n}}  \sum_{k=1}^{\underline{n}(\tilde{\omega}) -1} \tilde{\eta}_{q,k}  \tilde{\Lambda}_{f,2,i}(\overline{x}_{{k-1}},\gamma_k,\tilde{\omega}) g (\overline{x}_{{k-1}})  \mathds{1}_{K_{R}^c(\tilde{\omega})}(\overline{x}_{{k-1}}) =0.
\end{align*}

Moreover, from $\mathcal{E}_q(F,\tilde{A},A, \mathfrak{M}_q,\tilde{\eta}_q)$ \ref{hyp:erreur_tems_cours_fonction_test_reg_Lambda_representation_1} (i)(see (\ref{hyp:erreur_temps_cours_fonction_test_reg_Lambda_representation_2_1})), we derive that, $\tilde{\mathbb{P}}(d\tilde{\omega})-a.s.$, for every $R>0$, $\lim\limits_{n \to \infty} \tilde{\Lambda}_{f,q,i}(\overline{x}_{{n-1}},\gamma_n ,\tilde{\omega} ) \mathds{1}_{  K_ R(\tilde{\omega}) }(\overline{x}_{_{k-1}})=0, \;  $ Then, since $g_i $ is a locally bounded function, as an immediate consequence of the Cesaro's lemma, we obtain
\begin{equation*}
\tilde{\mathbb{P}}(d\tilde{\omega}) \quad \forall R>0, \quad \lim\limits_{n \to \infty} \frac{1}{H_{\tilde{\eta}_q ,n}} \sum_{k=1}^n \tilde{\eta}_{q,k} \tilde{\Lambda}_{f,q,i}(\overline{x}_{{k-1}},\gamma_k,\tilde{\omega}) g_i (\overline{x}_{{k-1}}) \mathds{1}_{   K_ R(\tilde{\omega}) }(\overline{x}_{k-1}) =0  
\end{equation*}
Applying the same approach for every $i \in \{1,\ldots,q \}$, the Dominated Convergence Theorem yields:
\begin{align*}
\forall (\overline{x}_{n})_{n \in \mathbb{N}} \in E^{\otimes \mathbb{N}} , \forall f \in F, \qquad \lim\limits_{n \to \infty} \frac{1}{H_{\tilde{\eta}_q ,n}} \sum_{k=1}^n  \Lambda_{f,q}(\overline{x}_{{k-1}},\gamma_k) = 0.
\end{align*}

Finally, since for every $i \in \{1,\ldots,l\}$, $\sup_{n \in \mathbb{N}^{\ast}} \nu_n^{\tilde{\eta}_q}( g_i ,\omega )< + \infty, \; \mathbb{P}-a.s.$, then (\ref{eq:preuve_conv_ident_vitess}) follows. We gather all the terms together and the proof is completed.

\end{proof}

\subsection{Example - The Euler scheme}
Using this abstract approach, we recover the results obtained in \cite{Lamberton_Pages_2002} or \cite{Panloup_2008_rate} concerning the study of the Euler scheme of a $d$-dimensional Brownian diffusion under weakly mean reverting properties. We consider a $N$-dimensional Brownian motion $(W_t)_{t \geqslant 0}$. We are interested in the strong solution - assumed to exist and to be unique - of the $d$-dimensional stochastic equation
\begin{align}
\label{eq:EDS_Euler_dim_d}
X_t= x+\int_0^t b(X_{s})ds + \int_0^t  \sigma(X_{s})  dW_s
\end{align}
where $b:\mathbb{R}^d \to \mathbb{R}^d$, $\sigma:\mathbb{R}^d \to \mathbb{R}^{d \times N}$. Let $V:\mathbb{R} \to [1,+\infty)$, the Lyapunov function of this system such that $\mbox{L}_V$ (see (\ref{hyp:Lyapunov})) holds with $E=\mathbb{R}^d $, and
\begin{align*}
\vert \nabla V \vert^2 \leqslant C_V V, \qquad \Vert D^2 V \Vert_{\infty} < + \infty.
\end{align*}
Moreover, we assume that for every $x \in \mathbb{R},  \quad  \vert b(x) \vert^2 + \mbox{Tr}[\sigma \sigma^{T}(x)]   \leqslant V^a(x)$ for some $a \in (0,1]$. Finally, for $p \geqslant 1$, we introduce the following $\mbox{L}_p$-mean reverting property of $V$,
\begin{align*}
\exists \alpha > 0, \beta \in \mathbb{R}, &\forall x \in \mathbb{R},  \nonumber \\   
&\langle \nabla V(x), b(x) \rangle+ \frac{1}{2}   \Vert \lambda_{p} \Vert_{\infty} 2^{(2p-3)_+} \mbox{Tr}[\sigma \sigma^{T}(x)]   \leqslant \beta - \alpha V^a (x )
\end{align*}
with for every $x \in \mathbb{R}^d$, $\lambda_p(x):= \sup\{\lambda_{p,1}(x),\ldots,\lambda_{p,d}(x),0 \}$, with $\lambda_{p,i}(x)$ the $i$-th eigenvalue of the matrix $D^2V(x)+2(p-1) \nabla V(x)^{\otimes 2} / V(x) $. We now introduce the Euler scheme of $(X_t)_{t \geqslant 0}$. Let $\rho \in [1,2]$ and $\epsilon_{\mathcal{I}}(\gamma)=\gamma^{\rho/2}$ and assume that (\ref{hyp:accroiss_sw_series_2}), $\mathcal{S}\mathcal{W}_{\mathcal{I}, \gamma,\eta}(\rho, \epsilon_{\mathcal{I}}) $ (see (\ref{hyp:step_weight_I})) and $\mathcal{S}\mathcal{W}_{\mathcal{II},\gamma,\eta}$ (see (\ref{hyp:step_weight_II})) hold. Let $(U_n)_{n}$ be a sequence of $\mathbb{R}^N$-valued centered independent and identically distributed random variables with covariance identity and bounded moments of order $2p$. We define the Euler scheme with decreasing step  $(\gamma_n)_{n \in \mathbb{N}^{\ast}}$, $(\overline{X}_{\Gamma_n} )_{n \in \mathbb{N}}$ of $(X_t)_{t \geqslant 0}$ (\ref{eq:EDS_Euler_dim_d}) on the time grid $\{\Gamma_n=\sum_{k=1}^n \gamma_k, n \in \mathbb{N} \}$ by
\begin{align*}
\forall n \in \mathbb{N} , \qquad  \overline{X}_{\Gamma_{n+1}}  =& \overline{X}_{\Gamma_n} + \gamma_{n+1} b(\overline{X}_{\Gamma_{n}}) + \sqrt{\gamma_{n+1}} \sigma(\overline{X}_{\Gamma_{n}})U_{n+1},\quad \overline{X}_{0} =x.
\end{align*}
We consider $ (\nu^{\eta}_n(dx,\omega))_{n \in \mathbb{N}^{\ast}}$ defined as in (\ref{eq:def_weight_emp_meas}) with $(\overline{X}_{\Gamma_n} )_{n \in \mathbb{N}}$ defined above. Now, we specify the measurable functions $\psi,\phi:[1,+\infty) \to [1,+\infty)$ as $\psi_p(y)=y^{p}$ and $\phi(y)=y^a$. Moreover, let $s \geqslant 1$ such that $a\,p\rho /s \leqslant p+a-1$, $p/s+a-1>0$ and $\Tr[\sigma \sigma^{T}] \leqslant CV^{p/s+a-1}$. Then, it follows from Theorem \ref{th:identification_limit} that there exists an invariant distribution $\nu$ for $(X_t)_{t \geqslant 0}$. Moreover, $ (\nu^{\eta}_n(dx,\omega))_{n \in \mathbb{N}^{\ast}}$ $a.s.$ weakly converges towards $\mathcal{V}$, the set of invariant distributions of $(X_t)_{t \geqslant 0}$ and when it is unique $i.e.$ $\mathcal{V}=\{\nu\}$, we have
\begin{align*}
\mathbb{P}-a.s. \qquad \lim_{n \to + \infty}\nu^{\eta}_n(f)= \nu(f),
\end{align*}
for every $\nu-a.s.$ continuous function $f\in \mathcal{C}_{\tilde{V}_{\psi_p,\phi,s}}(\mathbb{R}^d)$ defined in (\ref{def:espace_test_function_cv}).  Taking $V:x \mapsto 1+ \vert x \vert^{2}$, we obtain the $a.s.$ convergence for the $2p/s+2a-2$-Wasserstein distance. In addition to that $\mathbb{P}-a.s. $ weak converge result we can also establish a first order CLT. Let $\tilde{\rho}_1 \in [1,2]$, let $C_{\gamma,\eta}>0$ and let us define $\eta_{1,n}=C_{\gamma,\eta}\gamma_n$, $n \in \mathbb{N}^{\ast}$ and 
 \begin{align*}
 F_1 = \{f \in \mathcal{C}^{4}(\mathbb{R}^d;\mathbb{R}), \forall l\in \{2,\ldots,4\},  D^l f \in \mathcal{C}_0(\mathbb{R}^d;\mathbb{R})  \},
 \end{align*}
and the linear operator $\mathfrak{M}_1$ defined on $\mathcal{C}^4(\mathbb{R}^d;\mathbb{R})$ such that for every $f \in \mathcal{C}^4(\mathbb{R}^d;\mathbb{R})$,

\begin{align*}
\mathfrak{M}_1f (x) =& -\frac{1}{2} \big(D^2f(x);   b(x)^{\otimes 2}  \big)  - \mathbb{E} \Big[ \frac{1}{2} \big(D^3f(x); (\sigma(x)U)^{\otimes 2} \otimes b(x) \big)  + \frac{1}{4!} \big(D^4f(x); (\sigma(x)U)^{\otimes 4}  \big) \Big] . \nonumber
\end{align*}
 Let $\tilde{\eta}_1(\gamma)=\gamma^{2}$.  Assume that (\ref{eq:weight_def}), (\ref{hyp:accroiss_sw_series_2}), $\mathcal{S}\mathcal{W}_{\mathcal{I}, \gamma,\eta}(\rho, \epsilon_{\mathcal{I}}) $ (see (\ref{hyp:step_weight_I})) and $\mathcal{S}\mathcal{W}_{\mathcal{II},\gamma,\eta}$ (see (\ref{hyp:step_weight_II})) hold with $\eta$ replaced by $\tilde{\eta}_1$ and by $\gamma$.  Now, we introduce necessary assumptions concerning the random variables that are used to build this scheme.  Let $q \in \mathbb{N}^{\ast}$, $p \geqslant 0$. Now let $(U_n)_{n \in \mathbb{N}^{\ast}} $ be a sequence of $\mathbb{R}^N$-valued independent and identically distributed random variables such that
\begin{align}
\label{hyp:matching_normal_moment_ordre_q_va_schema_Talay}
M_{\mathcal{N},q}(U) \quad \equiv \qquad \forall n \in \mathbb{N}^{\ast} ,\forall \tilde{q} \in \{1, \ldots, q\} , \quad \mathbb{E}[(U_n)^{\otimes \tilde{q}}]=\mathbb{E}[(\mathcal{N}(0,I_d))^{\otimes \tilde{q}}],
\end{align}
and
\begin{align}
\label{hyp:moment_ordre_p_va_schema_Talay}
M_p(U) \qquad \sup_{n \in \mathbb{N}^{\ast}} \mathbb{E}[\vert U_n \vert^{2p} ] < + \infty.
\end{align}
We assume that the sequence $(U_n)_{n \in \mathbb{N}^{\ast}}$ satisfies $M_{\mathcal{N},3}(U)$ (see (\ref{hyp:matching_normal_moment_ordre_q_va_schema_Talay}))) and $M_{2}(U)$ (see (\ref{hyp:moment_ordre_p_va_schema_Talay})) and that $\mathcal{S}\mathcal{W}_{\mathcal{G}\mathcal{C}, \gamma}( \tilde{\rho}_1 , \gamma,\gamma)$  (see (\ref{hyp:step_weight_I_gen_chow_rate_sans_g})) holds. \\
Also assume that $g_{\sigma,1}\leqslant C  V^{p/s+a-1}$, with $g_{\sigma,1}=\Tr[ \sigma \sigma^{T} ]^{4}+\vert b \vert^{2}$, that $\Tr[ \sigma \sigma^{T} ]=o_{\vert x \vert \to +\infty}(V^{p/s+a-1})$ and that $\nu$ is unique. Finally assume that for every $f \in F_1$, $\vert \sigma^{T} Df \vert^2 \in \mathcal{C}_{\tilde{V}_{\psi_p,\phi,s}}(\mathbb{R}^d)$ and $\mathfrak{M}_1 f \in \mathcal{C}_{\tilde{V}_{\psi_p,\phi,s}}(\mathbb{R}^d)$. \\

Then, for every $f \in F_1$,

\begin{enumerate}[label=\textbf{\roman*.}]
\item If $\lim_{n \to \infty} \sqrt{ \Gamma_{n} }  /  H_{\tilde{\eta}_{1},n} =+\infty$,
\begin{align*}
\lim_{n \to \infty}   \sqrt{ \Gamma_{n} }  \nu ^{\eta_1}_n(Af) \overset{law}{=} \mathcal{N}(0, \nu( \vert \sigma^{T} Df \vert^2 )).
\end{align*}
\item \label{e1:TCL_mes_pond_moy_var} 
If $\lim_{n \to \infty}   \sqrt{ \Gamma_{n} }  /   H_{\tilde{\eta}_{1},n}  =\hat{l} \in \mathbb{R}^{\ast}_+$, 
\begin{align*}
\lim_{n \to \infty}  \sqrt{ \Gamma_{n} }  \nu ^{\eta_1}_n(A f) \overset{law}{=} \mathcal{N}(\hat{l}^{-1} \nu(\mathfrak{M}_1 f), \nu( \vert \sigma^{T} Df \vert^2 )).
\end{align*}

\item If $\lim_{n \to \infty}   \sqrt{ \Gamma_{n} }  /   H_{\tilde{\eta}_{1},n}=0$, 
\begin{align*}
\lim_{n \to \infty} \frac{H_n}{ H_{\tilde{\eta}_{1},n}} \nu ^{\eta_1}_n(Af) \overset{\mathbb{P}}{=} \nu(\mathfrak{M}_1 f).
\end{align*}
\end{enumerate}

This result was initially obtained in \cite{Lamberton_Pages_2002} but under strongly mean reverting assumption $i.e.$ $a=1$. The extension of this result to the weak mean reverting setting was developed in \cite{Panloup_2008_rate}. Notice that,  for $f \in F_{1}$ the Ito's Lemma yields, 
\begin{align*}
  \mathbb{E}[f(X_{t})^{2}] =  &  \mathbb{E}[f(X_{0})^{2}]  + \int_{0}^{t}\mathbb{E}[ f(X_{s})Af(X_{s}) +\vert \sigma^{T} Df \vert^2(X_{s})]ds .
\end{align*}
In particular, choosing $X_{0} \sim \nu$, we obtain $\nu( \vert \sigma^{T} Df \vert^2 ))=-2 \nu(fAf)$ and the asymptotic variance of our first order CLT is the same as in the continuous case \cite{Bhattacharya_1982}.
}
\begin{remark}
\label{remark:ordre_cv_1_appli}
Notice that if we take $\gamma_n=1/n^{\xi}$, $\xi \in (0,1/2)$ and 
$\eta=\gamma$, the mentioned step weight assumptions are satisfied (take $\rho \in (1/(1-\xi),2]$ and $\tilde{\rho}_1 \in(2/(1+\xi),2]$). Then, if we define by

 \begin{eqnarray*}
\forall n \in \mathbb{N}^{\ast}, \qquad \mathfrak{r}_n= 
& \quad  \left\{
    \begin{array}{l}
    \sqrt{ \Gamma_{n} }  \quad \mbox{if} \quad \lim_{n \to \infty}   \sqrt{ \Gamma_{n} }  /   H_{\tilde{\eta}_{1},n}=+\infty,\\
 \sqrt{ \Gamma_{n} }   \quad \mbox{if} \quad   \lim_{n \to \infty} \sqrt{ \Gamma_{n} }  /   H_{\tilde{\eta}_{1},n}  =\hat{l} ,\\
 \frac{H_n}{ H_{\tilde{\eta}_{1},n}} \quad \mbox{if} \quad \lim_{n \to \infty}   \sqrt{ \Gamma_{n} }  /   H_{\tilde{\eta}_{1},n}=0,
    \end{array}
\right.
\end{eqnarray*}
the rate of convergence of $(\nu ^{\eta_1}_n(A f))_{n \in \mathbb{N}^{\ast}}$, we have 
\begin{align*}
\mathfrak{r}_n \underset{n \to + \infty}{\sim} C n^{\xi \wedge (1/2-\xi/2)}.
\end{align*}
The highest rate of convergence is thus achieved for $\xi=1/3$ and is given by $\mathfrak{r}_n \underset{n \to + \infty}{\sim} C n^{1/3}$.
\end{remark}

\section{Application - The Talay second weak order scheme}
\label{Application - The Talay second weak order scheme}
\paragraph{Notations. \\}
 In the sequel we will use the following notations. First, for $\alpha \in (0,1]$ and $f$ an $\alpha$-H\"older function we denote $[f]_{\alpha}=\sup_{x \neq y}\vert f(y)-f(x) \vert / \vert y -x \vert^{\alpha}$.\\
Now, let $d \in \mathbb{N}$. For any $ \mathbb{R}^{d \times d}$-valued symmetric matrix $S$, we define $\lambda_S:= \sup\{\lambda_{S,1},\ldots, \lambda_{S,d},0 \}$, with $\lambda_{S,i}$ the $i$-th eigenvalue of $S$.\\

\paragraph{Presentation of the main result. \\}
In this section we study the second order convergence of the weighted empirical measures of a scheme designed in \cite{Talay_1990} and adapted to the case of decreasing time steps. We consider a $N$-dimensional Brownian motion $(W_t)_{t \geqslant 0}$. We are interested in the solution - assumed to exist and to be unique - of the $d$-dimensional stochastic equation
\begin{align*}
X_t= x+\int_0^t b(X_{s})ds + \int_0^t  \sigma(X_{s})  dW_s,
\end{align*}
where $b: \mathbb{R}^d \to \mathbb{R}^d$ and $\sigma:\mathbb{R}^{d } \to \mathbb{R}^{d \times N},$ are locally bounded functions. The infinitesimal generator of this process is given by
\begin{align*}
Af(x) =&\langle b(x) , \nabla f(x) \rangle + \frac{1}{2}\sum_{i,j=1}^d (\sigma \sigma^{T})_{i,j}(x) \frac{\partial^2f}{\partial x_i \partial x_j}(x)   
\end{align*} 

and its domain $\DomA$ contains $\DomA_0 =\mathcal{C}^2_K(\mathbb{R}^d)$. Notice that $\DomA_0 $ is dense in $\mathcal{C}_0(E)$. Now, we present the Talay's scheme, introduced in \cite{Talay_1990}, of $(X_t)_{t \geqslant 0}$ adapted to the case of decreasing time steps. \\


Moreover, let $(\kappa_n)_{n \in \mathbb{N}^{\ast}}$ be a sequence of $\mathbb{R}^{N \times N}$-valued independent and identically distributed random variables such that for every $ n \in \mathbb{N}^{\ast}$, $\kappa_n$ is made of $N\times N$ independent components and for every $(i,j) \in \{1, \ldots, N\}^2$, $\mathbb{P}(\kappa^{i,j}_{n}=-1/2)=\mathbb{P}(\kappa^{i,j}_{n}=1/2)=1/2$. At this point we define the sequence $(\mathcal{W}_n)_{n \in \mathbb{N}^{\ast}}$ of $\mathbb{R}^{N \times N}$-valued random variables such that for every $ n \in \mathbb{N}^{\ast}$,
\begin{align*}
\forall i,j \in \{1,\ldots,N \}, \qquad \mathcal{W}^{i,i}_n=\frac{1}{2}( \vert U^i_n \vert^2-1) \quad \mbox{and} \quad \mathcal{W}^{i,j}_n=  \frac{1}{2} U^i_n U^j_n -\kappa^{i \wedge j, i \vee j }_n \quad \mbox{for} \; i \neq j.
\end{align*}


For every $n \in \mathbb{N}$, the Talay's scheme with decreasing step  is defined by
\begin{align*}
\overline{X}_{\Gamma_{n+1}}  =& \overline{X}_{\Gamma_n} + \sqrt{\gamma_{n+1}}\sigma(\overline{X}_{\Gamma_{n}})U_{n+1}+ \gamma_{n+1}  \Big( b(\overline{X}_{\Gamma_{n}}) + (D \sigma( \overline{X}_{\Gamma_n} ) ; \sigma( \overline{X}_{\Gamma_n} )\mathcal{W}_{n+1}^{T}  )\Big)  \\
& + \gamma_{n+1}^{3/2} \tilde{\sigma}( \overline{X}_{\Gamma_n} ) U_{n+1}  + \gamma_{n+1}^{2} \frac{1}{2} Ab( \overline{X}_{\Gamma_n} ) ,  \nonumber
\end{align*}
with, for every $i \in \{1,\ldots,N\}$, and $j \in \{1, \ldots,d\}$, $\tilde{\sigma}_{j,i}=(\tilde{\sigma}_i)_j$ where
\begin{align*}
\begin{array}{crcl}
\tilde{\sigma}_i & :\mathbb{R}^d & \to &  \mathbb{R}^d\\
 &x& \mapsto &  \frac{1}{2}  \sum\limits_{l=1}^d \Big( \partial_{x_l} b( x ) \sigma_{l,i}( x ) +\partial_{x_l} \sigma_{l,i} ( x ) b( x ) + \frac{1}{2} \sum\limits_{j=1}^d (\sigma \sigma^{T})_{l,j}(x) \frac{\partial^2 \sigma_i }{\partial x_l \partial x_j}( x ) \Big).
\end{array}
\end{align*}
with, for every $i \in \{1,\ldots,N\}$, $\sigma_i:\mathbb{R}^d \to \mathbb{R}^d, x \mapsto \sigma_i(x)= (\sigma_{1,i}(x), \ldots , \sigma_{d,i}(x) )$.


We will also denote $\Delta \overline{X}_{n+1}=\overline{X}_{\Gamma_{n+1}}-\overline{X}_{\Gamma_{n}}$ and
\begin{align}
\label{def:incr_Talay}
& \Delta \overline{X}^1_{n+1} =  \gamma_{n+1}^{1/2} \sigma (\overline{X}_{\Gamma_{n}}) U_{n+1} = \gamma_{n+1}^{1/2} \sum_{i=1}^N \sigma_i (\overline{X}_{\Gamma_{n}}) U^i_{n+1}  , \quad  \Delta \overline{X}^2_{n+1} =   \gamma_{n+1}b(\overline{X}_{\Gamma_{n}} )   , \\
&\Delta \overline{X}^3_{n+1} =  (D \sigma( \overline{X}_{\Gamma_n} ) ; \sigma( \overline{X}_{\Gamma_n} )\mathcal{W}_{n+1}^{T}  ) =  \gamma_{n+1} \sum_{i,j=1}^N \sum_{l=1}^d \partial_{x_l} \sigma_i ( \overline{X}_{\Gamma_n} )  \sigma_{l,j}( \overline{X}_{\Gamma_n} ) \mathcal{W}^{i,j}_{n+1}     , \nonumber \\
& \Delta \overline{X}^4_{n+1} =  \gamma_{n+1}^{3/2} \tilde{\sigma}( \overline{X}_{\Gamma_n} ) U_{n+1}  =\gamma_{n+1}^{3/2} \sum_{i=1}^N \tilde{\sigma}_i( \overline{X}_{\Gamma_n} ) U^i_{n+1}   \nonumber \\
&\Delta \overline{X}^5_{n+1} =   \gamma_{n+1}^2 \frac{1}{2} Ab( \overline{X}_{\Gamma_n} )  \nonumber
\end{align}


and $\overline{X}_{\Gamma_{n+1}}^i=\overline{X}_{\Gamma_n}+ \sum_{j=1}^i \Delta \overline{X}^i_{n+1}$. Now, we assume the existence of a Lyapunov function $V: \mathbb{R}^d \to [v_{\ast}, \infty)$, $v_{\ast}>0$, satisfying $\mbox{L}_V$ (see (\ref{hyp:Lyapunov})) and which is essentially quadratic:
\begin{align}
\label{hyp:Lyapunov_control_Talay}
\vert \nabla V \vert^2 \leqslant C_V V, \qquad \sup_{x \in \mathbb{R}^d} \vert D^2 V(x) \vert < + \infty
\end{align}
It remains to introudce the mean-reverting property of $V$. We define
\begin{align}
\label{def:lambda_psi_Talay}
\forall x \in \mathbb{R}^d , \quad \lambda_{\psi}(x):= \lambda_{D^2V(x)+2\nabla V(x)^{\otimes 2} \psi''\circ V(x) \psi'\circ V(x)^{-1}}  .
\end{align}
When $\psi(y)=\psi_p(y)=y^{p}$, we will also use the notation $\lambda_p$ instead of $\lambda_{\psi}$. Now, let $\phi:[v_{\ast}, + \infty) \to \mathbb{R}_+$, and assume that for every $x \in \mathbb{R}^d$,
\begin{align}
\label{hyp:controle_coefficients_Talay}
\mathfrak{B}(\phi) \; \equiv   \; \vert b(x) \vert^2 + \Tr[ \sigma \sigma^{T}(x) ] +\vert D \sigma( x ) \vert^2   \Tr[ \sigma \sigma^{T}(x) ]+ \vert \tilde{\sigma}(x) \vert^2 + \vert Ab(x) \vert^2 \leqslant C  \phi \circ V (x).
\end{align}

We are now able to introduce the $\LL_p$ mean-reverting property of $V$. Let $p \geqslant 0$. Let $\beta \in \mathbb{R}$, $\alpha>0$. We assume that $\liminf\limits_{y \to \infty} \phi(y)>\beta/\alpha$ and 
\begin{align}
\label{hyp:recursive_control_param_Talay}
\mathcal{R}_p(\alpha,\beta, \phi, V) \quad \equiv \qquad \forall x \in \mathbb{R}^d,  \quad  \langle \nabla V(x), b(x) \rangle+   \frac{1}{2} \chi_{p}(x) \leqslant \beta - \alpha \phi \circ V (x),
\end{align}
with

\begin{equation}
\label{hyp:recursive_control_param_terme_ordre_sup_Talay}
\chi_{p}(x) =  \left\{
      \begin{aligned}
        & \Vert \lambda_{1} \Vert_{\infty}  \mbox{Tr}[\sigma \sigma^{T}(x)]  & & \quad \mbox{if } p \leqslant 1\\
        & \Vert \lambda_{p} \Vert_{\infty} 2^{(2p-3)_+} \mbox{Tr}[\sigma \sigma^{T}(x)]  &   & \quad \mbox{if } p >1.
      \end{aligned}
    \right.
\end{equation}

Finally we consider the linear operator $\mathfrak{M}_1$ defined on $\mathcal{C}^4(\mathbb{R}^d;\mathbb{R})$ such that for every $f \in \mathcal{C}^4(\mathbb{R}^d;\mathbb{R})$,


\begin{align}
\label{eq:def_M1}
\mathfrak{M}_1f (x) =& - \frac{1}{2} \big(Df(x);Ab(x) \big) \\
& - \mathbb{E} \Big[  \frac{1}{2} \big(D^2f(x);   b(x)^{\otimes 2} + 2 b(x) \otimes (D \sigma(x); \sigma(x) \mathcal{W}^{T} )+(D \sigma(x); \sigma(x) \mathcal{W}^{T} )^{\otimes 2} \big)  \nonumber  \\
&+ \frac{1}{2} \big(D^3f(x); (\sigma(x)U)^{\otimes 2} \otimes (b(x)+(D \sigma(x); \sigma(x) \mathcal{W}^{T} ))+ (\sigma(x)U) \otimes (\tilde{\sigma}(x)U)\big)  \nonumber \\
&+ \frac{1}{4!} \big(D^4f(x); (\sigma(x)U)^{\otimes 4}  \big) \Big] . \nonumber
\end{align}

We also consider the linear operator $\mathfrak{M}_2$ defined on $\mathcal{C}^6(\mathbb{R}^d;\mathbb{R})$ such that for every $f \in \mathcal{C}^6(\mathbb{R}^d;\mathbb{R})$, $\mathfrak{M}_2f=\tilde{\mathfrak{M}}_2f - \frac{1}{2} \mathfrak{M}_1Af$ with


\begin{align}
\label{eq:def_M2_tilde}
\tilde{\mathfrak{M}}_2f (x) =&- \mathbb{E} \Big[ \big( D^2f(x);  \frac{1}{2} (\tilde{\sigma}(x) U)^{\otimes 2})+\frac{1}{2} b(x)\otimes Ab(x) \big) \\
&+ \frac{1}{2} \big(D^3f(x);\frac{1}{3}(D \sigma(x); \sigma(x) \mathcal{W}^{T} )^{\otimes 3}+ b(x)^{\otimes 2} \otimes (D \sigma(x); \sigma(x) \mathcal{W}^{T} ) +\frac{1}{2} (\sigma(x)U)^{\otimes 2}\otimes Ab(x)   \nonumber  \\
& \qquad \qquad \qquad \qquad + (\sigma(x)U) \otimes (b(x)+  (D \sigma(x); \sigma(x) \mathcal{W}^{T} ))\otimes (\tilde{\sigma}(x)U) + \frac{1}{3}b(x)^{\otimes 3} \big)   \nonumber  \\
&+ \frac{1}{2} \big( D^4f(x); \frac{1}{2}(\sigma(x)U)^{\otimes 2} \otimes (b(x)^{\otimes 2}+2 b(x) \otimes (D \sigma(x); \sigma(x) \mathcal{W}^{T} )  + (D \sigma(x); \sigma(x) \mathcal{W}^{T} )^{\otimes 2}  )  \nonumber\\
& \qquad \qquad \qquad \qquad  + \frac{1}{3} (\sigma(x)U)^{\otimes 3}\otimes (\tilde{\sigma}(x)U) \big) \nonumber   \\
&+ \frac{1}{4!} \big( D^5f(x);(\sigma(x)U)^{\otimes 4} \otimes (b(x)+(D \sigma(x); \sigma(x) \mathcal{W}^{T} )) \big)  \nonumber  \\
&+ \frac{1}{6!} \big( D^6f(x);(\sigma(x)U)^{\otimes 6} \big)\Big] . \nonumber
\end{align}

We are now in a position to provide our main result concerning convergence of weighted empirical measures of the Talay's scheme. This first part of this result concerns the $\mathbb{P}-a.s.$ weak convergence while the second part establishes first and second order CLT.
\begin{mytheo}
\label{th:cv_was_Talay}
Let $p >0,a \in (0,1]$, $s \geqslant 1, \rho \in [1,2]$ and, $\psi_p(y)=y^p$, $\phi(y)=y^a$ and $\epsilon_{\mathcal{I}}(\gamma)=\gamma^{\rho/2}$. Let $\alpha>0$ and $\beta \in \mathbb{R}$. 
 \begin{enumerate}[label=\textbf{\Alph*.}]
 \item \label{th:cv_was_Talay_point_A}
Assume that the sequence $(U_n)_{n \in \mathbb{N}^{\ast}}$ satisfies $M_{\mathcal{N},2}(U)$ (see (\ref{hyp:matching_normal_moment_ordre_q_va_schema_Talay})) and $M_{(2p) \vee (2p \rho/s) \vee 2}(U)$ (see (\ref{hyp:moment_ordre_p_va_schema_Talay})). 
Also assume that (\ref{hyp:Lyapunov_control_Talay}), $\mathfrak{B}(\phi)$ (see (\ref{hyp:controle_coefficients_Talay})), $\mathcal{R}_p(\alpha,\beta, \phi, V)$ (see (\ref{hyp:recursive_control_param_Talay})), $\mbox{L}_{V}$ (see (\ref{hyp:Lyapunov}), $\mathcal{S}\mathcal{W}_{\mathcal{I}, \gamma,\eta}(\rho, \epsilon_{\mathcal{I}})$ (see (\ref{hyp:step_weight_I})), $\mathcal{S}\mathcal{W}_{\mathcal{II},\gamma,\eta}(V^{p/s}) $ (see (\ref{hyp:step_weight_I_gen_tens})) and (\ref{hyp:accroiss_sw_series_2}) also hold and that $ap\rho/s \leqslant p+a-1$.\\

 Then, if $p/s+a-1>0$, $(\nu_n^{\eta})_{n \in \mathbb{N}^{\ast}}$ is $\mathbb{P}-a.s.$ tight and
\begin{align}
 \label{eq:tightness_Talay}
\mathbb{P} \mbox{-a.s.} \quad \sup_{n \in \mathbb{N}^{\ast}} \nu_n^{\eta}( V^{p/s+a-1} ) < + \infty .
\end{align}

Moreover, assume also that $b$, $\sigma$, $\vert D\sigma \vert \Tr[\sigma \sigma^{T}]^{1/2}$, $\tilde{\sigma}$ and $Ab$ have sublinear growth and that $g_{\sigma}\leqslant C  V^{p/s+a-1}$, with $g_{\sigma}=\Tr[ \sigma \sigma^{T} ]+ \vert D\sigma \vert \Tr[\sigma \sigma^{T}]^{1/2}+\Tr[ \tilde{\sigma} \tilde{\sigma}^{T} ]^{1/2} $. Then, every weak limiting distribution $\nu$ of $(\nu_n^{\eta})_{n \in \mathbb{N}^{\ast}}$ is an invariant distribution of $(X_t)_{t \geqslant 0}$ and when $\nu$ is unique, we have
\begin{align}
\label{eq:cv_was_Talay}
\mathbb{P} \mbox{-a.s.} \quad  \forall f \in \mathcal{C}_{\tilde{V}_{\psi_p,\phi,s}}(\mathbb{R}^d), \quad \lim\limits_{n \to + \infty} \nu_n^{\eta}(f)=\nu(f),
\end{align}
 with $\mathcal{C}_{\tilde{V}_{\psi_p,\phi,s}}(\mathbb{R}^d)$ defined in (\ref{def:espace_test_function_cv}). Notice that when $p/s \leqslant p\vee 1 +a-1$, the assumption $\mathcal{S}\mathcal{W}_{\mathcal{II},\gamma,\eta}(V^{p/s}) $ (see (\ref{hyp:step_weight_I_gen_tens})) can be replaced by $\mathcal{S}\mathcal{W}_{\mathcal{II},\gamma,\eta} $ (see (\ref{hyp:step_weight_II})). \\

\item \label{th:cv_was_Talay_point_B} 
 Let $q \in \{1,2\}$, let $\tilde{\rho}_q \in [1,2]$, let $C_{\gamma,\eta}>0$ and let us define $\eta_{1,n}=C_{\gamma,\eta}\gamma_n$, $\eta_{2,n+1}=C_{\gamma,\eta}(\gamma_n+\gamma_{n+1})/2$, $n \in \mathbb{N}^{\ast}$ (with $\gamma_0=0$) and 
 \begin{align*}
 F_q = \{f \in \mathcal{C}^{2(q+1)}(\mathbb{R}^d;\mathbb{R}), \forall l\in \{1,\ldots,2(q+1)\},  D^l f \in \mathcal{C}_0(\mathbb{R}^d;\mathbb{R}) , Af \in F_1 \; \mbox{if } \;q=2 \}.
 \end{align*}
  Finally let $\tilde{\eta}_q(\gamma)=\gamma^{q+1}$.\\
  
Assume that the sequence $(U_n)_{n \in \mathbb{N}^{\ast}}$ satisfies $M_{\mathcal{N},2q+1}(U)$ (see (\ref{hyp:matching_normal_moment_ordre_q_va_schema_Talay})) and $M_{q+1}(U)$ (see (\ref{hyp:moment_ordre_p_va_schema_Talay})) and that $\mathcal{S}\mathcal{W}_{\mathcal{G}\mathcal{C}, \gamma}( \tilde{\rho}_q , \gamma,\gamma)$  (see (\ref{hyp:step_weight_I_gen_chow_rate_sans_g})) holds. \\
Also assume that $g_{\sigma,q}\leqslant C  V^{p/s+a-1}$, with $g_{\sigma,q}=\Tr[ \sigma \sigma^{T} ]^{2(q+1)}+\vert b \vert^{q+1}+ \vert D\sigma \vert^{q+1} \Tr[\sigma \sigma^{T}]^{(q+1)/2}+\Tr[ \tilde{\sigma} \tilde{\sigma}^{T} ] +\vert Ab \vert^q$, that $\Tr[ \sigma \sigma^{T} ]=o_{\vert x \vert \to +\infty}(V^{p/s+a-1})$, that $\nu$ is unique and that (\ref{eq:weight_def}) and the hypotheses from point \ref{th:cv_was_Talay_point_A} hold with $\eta$ replaced by $\tilde{\eta}_q$ and by $\gamma$. Finally assume that for every $f \in F_q$, $\vert \sigma^{T} Df \vert^2 \in \mathcal{C}_{\tilde{V}_{\psi_p,\phi,s}}(\mathbb{R}^d)$ and $\mathfrak{M}_q f \in \mathcal{C}_{\tilde{V}_{\psi_p,\phi,s}}(\mathbb{R}^d)$. \\

Then, for every $f \in F_q$, we have

\begin{enumerate}[label=\textbf{\roman*.}]
\item If $\lim_{n \to \infty} \sqrt{ \Gamma_{n} }  /  H_{\tilde{\eta}_{q},n} =+\infty$, ,
\begin{align*}
\lim_{n \to \infty}  \sqrt{ \Gamma_{n} }  \nu ^{\eta_q}_n(Af) \overset{law}{=} \mathcal{N}(0, \nu( \vert \sigma^{T} Df \vert^2 )).
\end{align*}
\item \label{eq:TCL_mes_pond_moy_var} 
If $\lim_{n \to \infty}   \sqrt{ \Gamma_{n} }  /   H_{\tilde{\eta}_{q},n}  =\hat{l} \in \mathbb{R}^{\ast}_+$, 
\begin{align*}
\lim_{n \to \infty} \sqrt{ \Gamma_{n} }  \nu ^{\eta_q}_n(A f) \overset{law}{=} \mathcal{N}(\hat{l}^{-1} \nu(\mathfrak{M}_q f), \nu( \vert \sigma^{T} Df \vert^2 )).
\end{align*}

\item If $\lim_{n \to \infty}   \sqrt{ \Gamma_{n} }  /   H_{\tilde{\eta}_{q},n}=0$,
\begin{align*}
\lim_{n \to \infty} \frac{H_n}{ H_{\tilde{\eta}_{q},n}} \nu ^{\eta_q}_n(Af) \overset{\mathbb{P}}{=} \nu(\mathfrak{M}_q f)
\end{align*}
\end{enumerate}
\end{enumerate}

\end{mytheo}

\begin{remark}
\label{rmk:rate_Talay}
Notice that if we take $\gamma_n=1/n^{\xi}$, $\xi \in (0,1/(q+1))$, the mentioned step weight assumptions of Theorem \ref{th:cv_was_Talay} point \ref{th:cv_was_Talay_point_B} are satisfied (take $\rho \in (1/(1-\xi),2]$ and $\tilde{\rho}_q \in(2/(1+\xi),2]$). Then, if we define by

 \begin{eqnarray*}
\forall n \in \mathbb{N}^{\ast}, \qquad \mathfrak{r}_{q,n}= 
& \quad  \left\{
    \begin{array}{l}
    \sqrt{ \Gamma_{n} }  \quad \mbox{if} \quad \lim_{n \to \infty}   \sqrt{ \Gamma_{n} }  /   H_{\tilde{\eta}_{q},n}=+\infty,\\
 \sqrt{ \Gamma_{n} }   \quad \mbox{if} \quad   \lim_{n \to \infty} \sqrt{ \Gamma_{n} }  /   H_{\tilde{\eta}_{q},n}  =\hat{l} ,\\
 \frac{H_n}{ H_{\tilde{\eta}_{q},n}} \quad \mbox{if} \quad \lim_{n \to \infty}   \sqrt{ \Gamma_{n} }  /   H_{\tilde{\eta}_{q},n}=0,
    \end{array}
\right.
\end{eqnarray*}
the rate of convergence of $(\nu ^{\eta_q}_n(A f))_{n \in \mathbb{N}^{\ast}}$, we have 
\begin{align*}
\mathfrak{r}_{q,n} \underset{n \to + \infty}{\sim} C n^{(q\xi) \wedge (1/2-\xi/2)}.
\end{align*}
The highest rate of convergence is thus achieved for $\xi=1/(2q+1)$ and is given by $\mathfrak{r}_{q,n} \underset{n \to + \infty}{\sim} C n^{q/(2q+1)}$. In particular in the first order case ($q=1$) we have $\mathfrak{r}_{1,n} \underset{n \to + \infty}{\sim} C n^{1/3}$ which is, as expected, the same rate as for the Euler scheme (see Remark \ref{remark:ordre_cv_1_appli}). However, for the second order case ($q=2$) we obtain a faster rate of convergence since $\mathfrak{r}_{2,n} \underset{n \to + \infty}{\sim} C n^{2/5}$. This rate can be achieved because $(\overline{X}_{\Gamma_n})_{n \in \mathbb{N}}$ is a second weak order scheme but also because the step sequence  $(\eta_{2,n})_{n \in \mathbb{N}^{\ast}}$ is well chosen.\\
\end{remark}

The next part of this Section is dedicated to the proof of Theorem \ref{th:cv_was_Talay}.

\subsection{Recursive control}

\begin{myprop}
\label{prop:recursive_control_Talay}
Let $v_{\ast}>0$, and let $\phi:[v_{\ast},\infty )\to \mathbb{R}_+$ be a continuous function such that $C_{\phi}:= \sup_{y \in [v_{\ast},\infty )}\phi(y)/y<+ \infty$.  Now let $p >0$ and define $\psi_p(y)=y^p$. Let $\alpha>0$ and $\beta \in \mathbb{R}$. \\

Assume that $(U_n)_{n \in \mathbb{N}^{\ast}} $ is a sequence of independent random variables such that $U$ satisfies $M_{\mathcal{N},2}(U) $ (see (\ref{hyp:matching_normal_moment_ordre_q_va_schema_Talay})) and $M_{(2p) \vee 2} (U)$ (see (\ref{hyp:moment_ordre_p_va_schema_Talay})).  \\
Also assume that (\ref{hyp:Lyapunov_control_Talay}), $\mathfrak{B}(\phi)$ (see (\ref{hyp:controle_coefficients_Talay})), $\mathcal{R}_p(\alpha,\beta, \phi, V)$ (see (\ref{hyp:recursive_control_param_Talay})), are satisfied. \\

Then, for every $\tilde{\alpha}\in (0, \alpha)$, there exists $n_0 \in \mathbb{N}^{\ast}$, such that 

\begin{align}
\label{eq:recursive_control_Talayt_fonction_pol}
 \forall n   \geqslant n_0, \forall x \in \mathbb{R}^d,  \quad\tilde{A}_{\gamma_n} \psi_{p} \circ V(x)\leqslant \frac{ \psi_p \circ V(x)}{V(x)}p(\beta - \tilde{\alpha} \phi\circ V(x)). 
\end{align}

Then $\mathcal{RC}_{Q,V}(\psi,\phi,p\tilde{\alpha},p\beta)$ (see (\ref{hyp:incr_sg_Lyapunov})) holds for every $\tilde{\alpha}\in (0, \alpha)$ such that $\liminf\limits_{y \to + \infty} \phi(y) > \beta / \tilde{\alpha}$. Moreover, when $\phi=Id$ we have

\begin{align}
\label{eq:mom_pol_Talay}
\sup_{n \in \mathbb{N}} \mathbb{E}[V^p(\overline{X}_{\Gamma_{n}})] < + \infty.
\end{align}
\end{myprop}

\begin{proof} 
We distinguish the cases $p \geqslant 1$ and $p \in(0,1)$.
\paragraph{Case $p\geqslant 1$. }
 First, we focus on the case $p \geqslant 1$.  From the Taylor's formula and the definition of $\lambda_{\psi_p}=\lambda_p$ (see (\ref{def:lambda_psi_Talay})), we have

\begin{align}
\label{eq:taylor_preuve_RC_pol_Talay}
\psi_{p} \circ V(\overline{X}_{\Gamma_{n+1}})=& \psi_{p} \circ V(\overline{X}_{\Gamma_n})+ \langle \overline{X}_{\Gamma_{n+1}}-\overline{X}_{\Gamma_n}, \nabla V(\overline{X}_{\Gamma_n}) \rangle \psi_{p}'\circ V(\overline{X}_{\Gamma_n}) \nonumber \\
&+ \frac{1}{2} \big( D^2 V(\Upsilon_{n+1}) \psi_{p}' \circ  V(\Upsilon_{n+1})+\nabla V (\Upsilon_{n+1} )^{\otimes 2} \psi_{p}'' \circ V(\Upsilon_{n+1}) ;
 ( \overline{X}_{\Gamma_{n+1}}-\overline{X}_{\Gamma_n} )^{\otimes 2} \big) \nonumber \\
 \leqslant & \psi_{p} \circ V(\overline{X}_{\Gamma_n})+ \langle \overline{X}_{\Gamma_{n+1}}-\overline{X}_{\Gamma_n}, \nabla V(\overline{X}_{\Gamma_n}) \rangle \psi_{p}'\circ V(\overline{X}_{\Gamma_n})  \nonumber \\
&+ \frac{1}{2} \lambda_{p} (\Upsilon_{n+1} )  \psi_{p}'\circ V(\Upsilon_{n+1}) \vert  \overline{X}_{\Gamma_{n+1}}-\overline{X}_{\Gamma_n}  \vert^{2}. 
\end{align}

with $\Upsilon_{n+1}  \in (\overline{X}_{\Gamma_n}, \overline{X}_{\Gamma_{n+1}})$. First, from (\ref{hyp:Lyapunov_control_Talay}), we have $\sup_{x \in \mathbb{R}^d} \lambda_{p} (x)  < + \infty$.

 Since $U$ and $\mathcal{W}$ are made of centered random variables, we deduce from $M_{\mathcal{N},2}(U) $ (see (\ref{hyp:matching_normal_moment_ordre_q_va_schema_Talay})) and $M_{4} (U)$ (see (\ref{hyp:moment_ordre_p_va_schema_Talay})) that


\begin{align*}
& \mathbb{E}[\overline{X}_{\Gamma_{n+1}}-\overline{X}_{\Gamma_n} \vert \overline{X}_{\Gamma_n} ]= \gamma_{n+1}  b(\overline{X}_{\Gamma_n} ) +  \gamma_{n+1}^{2} \frac{1}{2} Ab( \overline{X}_{\Gamma_n} )\\
& \mathbb{E}[ \vert  \overline{X}_{\Gamma_{n+1}}-\overline{X}_{\Gamma_n} \vert^{2} \vert \overline{X}_{\Gamma_n} ]\leqslant \gamma_{n+1}\mbox{Tr}[\sigma \sigma^{T}(\overline{X}_{\Gamma_n} )]  + \gamma_{n+1}^{3/2} C \Big(\mbox{Tr}[\sigma \sigma^{T}(\overline{X}_{\Gamma_n} )] + \vert b(\overline{X}_{\Gamma_n} ) \vert^{2}\\
& \qquad \qquad \qquad \qquad \qquad  \qquad  +\vert D \sigma( \overline{X}_{\Gamma_n} ) \vert^2   \Tr[ \sigma \sigma^{T}(\overline{X}_{\Gamma_n}) ]  +   \vert \tilde{\sigma}(x) \vert^2 + \vert Ab(x) \vert^2 \Big),
\end{align*}

with $C$ a positive constant. Assume first that $p=1$. Using $\mathfrak{B}(\phi)$ (see (\ref{hyp:controle_coefficients_Talay})), for every $\tilde{\alpha} \in (0, \alpha)$, there exists $n_0(\tilde{\alpha})$ such that for every $n\geqslant n_0(\tilde{\alpha})$,


\begin{align}
\label{eq:recursive_control_remaider_Id_Talay}
 \gamma_{n+1}^{2} \frac{1}{2}Ab( \overline{X}_{\Gamma_n} ) +  &\frac{1}{2}\Vert \lambda_{1} \Vert_{\infty}    \gamma_{n+1}^{3/2} C \Big(\mbox{Tr}[\sigma \sigma^{T}(\overline{X}_{\Gamma_n} )] + \vert b(\overline{X}_{\Gamma_n} ) \vert^{2}\\
&  +\vert D \sigma( \overline{X}_{\Gamma_n} ) \vert^2   \Tr[ \sigma \sigma^{T}(\overline{X}_{\Gamma_n}) ]  +   \vert \tilde{\sigma}(x) \vert^2 + \vert Ab(x) \vert^2 \Big) \leqslant \gamma_{n+1}(\alpha- \tilde{\alpha})\phi \circ V(\overline{X}_{\Gamma_n} ).  \nonumber 
\end{align}

From assumption $\mathcal{R}_p(\alpha,\beta, \phi, V)$ (see (\ref{hyp:recursive_control_param_Talay}) and (\ref{hyp:recursive_control_param_terme_ordre_sup_Talay})), we conclude that
\begin{align*}
\tilde{A}_{\gamma_n} \psi_1 \circ V(x) \leqslant \beta - \tilde{\alpha} \phi \circ V(x)
\end{align*}
Assume now that $p>1$.Since $\vert \nabla V \vert \leqslant C_V V$ (see (\ref{hyp:Lyapunov_control_Talay})), then $\sqrt{V}$ is Lipschitz. Now, we use the following inequality: Let $l \in \mathbb{N}^{\ast}$. We have

\begin{align*}
\forall \alpha >0 ,\forall u_i \in \mathbb{R}^d, i=1,\ldots,l, \qquad  \big \vert  \sum_{i=1}^l u_i \big \vert^{\alpha} \leqslant l^{(\alpha-1)_+}  \sum_{i=1}^l \vert u_i  \vert^{\alpha} .
\end{align*}
\begin{align*}
V^{p-1} (\Upsilon_{n+1} )  \leqslant & \big( \sqrt{V}(\overline{X}_{\Gamma_n})+[\sqrt{V}]_1 \vert \overline{X}_{\Gamma_{n+1}}-\overline{X}_{\Gamma_n} \vert \big)^{2p-2} \\
\leqslant & 2^{(2p-3)_+} (V^{p-1}(\overline{X}_{\Gamma_n}) + [\sqrt{V}]_1^{2p-2} \vert \overline{X}_{\Gamma_{n+1}}-\overline{X}_{\Gamma_n} \vert^{2p-2})
\end{align*}
To study the `remainder' of (\ref{eq:taylor_preuve_RC_pol_Talay}), we multiply the above inequality by $\vert \overline{X}_{\Gamma_{n+1}}-\overline{X}_{\Gamma_n} \vert^{2}$. First, we study the second term which appears in the $r.h.s.$ and using $\mathfrak{B}(\phi)$ (see (\ref{hyp:controle_coefficients_Talay})), for every $p \geqslant 1$,
\begin{align*}
\vert  \overline{X}_{\Gamma_{n+1}}- \overline{X}_{\Gamma_n} \vert^{2p} \leqslant C \gamma_{n+1}^p \phi \circ V(\overline{X}_{\Gamma_{n}})^p(1+\vert U_{n+1} \vert^{4p}).
\end{align*}

 Let $\hat{\alpha} \in (0,\alpha)$. Then, we deduce from $M_{2p } (U)$ (see (\ref{hyp:moment_ordre_p_va_schema_Talay})) that there exists $n_0(\hat{\alpha}) \in \mathbb{N}$ such that for any $n \geqslant n_0(\hat{\alpha})$, we have
\begin{align*}
 \mathbb{E}[ \vert \overline{X}_{\Gamma_{n+1}}- \overline{X}_{\Gamma_n} \vert^{2p} \vert \overline{X}_{\Gamma_n} ] \leqslant \gamma_{n+1} \phi \circ V (\overline{X}_{\Gamma_{n}})^{p} \frac{\alpha- \hat{\alpha} }{\Vert \phi /I_d \Vert_{\infty}^{p-1} \Vert \lambda_{p} \Vert_{\infty} 2^{(2p-3)_+} [\sqrt{V}]_1^{2p-2} } 
\end{align*}

To treat the other term of the `remainder' of (\ref{eq:taylor_preuve_RC_pol_Talay}) we proceed as in (\ref{eq:recursive_control_remaider_Id_Talay}) with $\Vert \lambda_{1} \Vert_{\infty}$ replaced by $\Vert \lambda_{p} \Vert_{\infty} 2^{2p-3} [\sqrt{V}]_1^{2p-2}  $, $\alpha$ replace by $ \hat{\alpha}$ and $\tilde{\alpha} \in (0, \hat{\alpha})$. We gather all the terms together and using (\ref{hyp:recursive_control_param_terme_ordre_sup_Talay}), for every $n \geqslant n_0(\tilde{\alpha}) \vee n_0(\hat{\alpha}) $, we obtain

\begin{align*}
\mathbb{E}[V^p (\overline{X}_{\Gamma_{n+1}})- V^p(\overline{X}_{\Gamma_{n}}) \vert \overline{X}_{\Gamma_{n}}] \leqslant & \gamma_{n+1}p V^{p-1}(\overline{X}_{\Gamma_{n}})( \beta - \alpha \phi \circ V (\overline{X}_{\Gamma_{n}})  ) \\
& + \gamma_{n+1}p V^{p-1}(\overline{X}_{\Gamma_{n}}) \Big(  \phi \circ V (\overline{X}_{\Gamma_{n}}) (\hat{\alpha} -\tilde{\alpha}  ) \\
& \qquad \qquad \qquad \qquad + (\alpha-\hat{\alpha})  \frac{ V^{1-p}(\overline{X}_{\Gamma_{n}})  \phi \circ V (\overline{X}_{\Gamma_{n}})^{p} }{\Vert \phi /I_d \Vert_{\infty}^{p-1}}    \Big) \\
\leqslant&\gamma_{n+1} V^{p-1}(\overline{X}_{\Gamma_{n}})( \beta p- \tilde{\alpha} p  \phi \circ V (\overline{X}_{\Gamma_{n}})  ). \\
\end{align*}
which is exactly the recursive control for $p>1$. 

\paragraph{Case $p\in(0,1)$. }
Now, let $p \in (0,1)$ so that $x \mapsto x^{p}$ is concave. it follows that
\begin{align*}
V^{p}(\overline{X}_{\Gamma_{n+1}} ) - V^{p}(\overline{X}_{\Gamma_{n}} ) \leqslant pV^{p-1}(\overline{X}_{\Gamma_{n}}  ) (V(\overline{X}_{\Gamma_{n+1}} ) - V(\overline{X}_{\Gamma_{n}} ) )
\end{align*}
We have just proved that we have the recursive control $\mathcal{RC}_{Q,V}(\psi , \phi, \alpha, \beta)$ holds for $\psi=I_d$ (with $\alpha$ replaced by $\tilde{\alpha} >0$), and since $V$ takes positive values, we obtain

\begin{align*}
\mathbb{E}[V^{p}(\overline{X}_{\Gamma_{n+1}} ) - V^{p}(\overline{X}_{\Gamma_{n}} )\vert \overline{X}_{\Gamma_{n}} ] \leqslant  & pV^{p-1}(\overline{X}_{\Gamma_{n}}  ) \mathbb{E}[V(\overline{X}_{\Gamma_{n+1}} ) - V(\overline{X}_{\Gamma_{n}} ) \vert \overline{X}_{\Gamma_{n}} ] \\
 \leqslant & \gamma_{n+1} V^{p-1}(\overline{X}_{\Gamma_{n}}  ) (p\beta - p\tilde{\alpha} \phi \circ V (\overline{X}_{\Gamma_{n}}  ) ),
\end{align*}
which completes the proof of (\ref{eq:recursive_control_Talayt_fonction_pol}). The proof of (\ref{eq:mom_pol_Talay}) is an immediate application of Lemma \ref{lemme:mom_psi_V} as soon as we notice that the increments of the Talay's scheme have finite polynomial moments which implies (\ref{eq:mom_psi_V}).

\end{proof}

\subsection{Infinitesimal approximation}

\begin{myprop}
\label{prop:Talay_infinitesimal_approx}
Assume that $b$, $\sigma$, $\vert D\sigma \vert \Tr[\sigma \sigma^{T}]^{1/2}$, $\tilde{\sigma}$ and $Ab$ have sublinear growth. We have the following properties:
\begin{enumerate}[label=\textbf{\Alph*.}] 
\item  \label{prop:Talay_infinitesimal_approx_point_A}
Assume that the sequence $(U_n)_{n \in \mathbb{N}^{\ast}}$ satisfies $M_{\mathcal{N},2}(U)$ (see (\ref{hyp:matching_normal_moment_ordre_q_va_schema_Talay})) and that $\sup_{n \in \mathbb{N}^{\ast}} \nu_n^{\eta}( \Tr[ \sigma \sigma^{T}]  )< + \infty$, $\sup_{n \in \mathbb{N}^{\ast}} \nu_n^{\eta}(  \vert D\sigma \vert \Tr[\sigma \sigma^{T}]^{1/2})< + \infty$ and $\sup_{n \in \mathbb{N}^{\ast}} \nu_n^{\eta}( \Tr[ \tilde{\sigma} \tilde{\sigma}^{T}]^{1/2}  )< + \infty$. \\

Then, $\mathcal{E}(\tilde{A},A,\DomA_0) $ (see (\ref{hyp:erreur_tems_cours_fonction_test_reg})) is satisfied.

\item \label{prop:Talay_infinitesimal_approx_point_B}
Let $F_1 = \{f \in \mathcal{C}^4(\mathbb{R}^d;\mathbb{R}), \forall q\in \{1,\ldots,4\},  D^qf \in \mathcal{C}_0(\mathbb{R}^d;\mathbb{R}) \}$, let $\mathfrak{M}_1$ defined in (\ref{eq:def_M1}) and let $\tilde{\eta}_1(\gamma)=\gamma^2$. \\

Assume that the sequence $(U_n)_{n \in \mathbb{N}^{\ast}}$ satisfies $M_{\mathcal{N},3}(U)$ (see (\ref{hyp:matching_normal_moment_ordre_q_va_schema_Talay})) and $M_2(U)$ (see (\ref{hyp:moment_ordre_p_va_schema_Talay})) and that $\sup_{n \in \mathbb{N}^{\ast}} \nu_n^{\tilde{\eta}_1}(g_1)< + \infty$, with  $g_1: \mathbb{R}^d \to \mathbb{R}$ such that for every $x \in \mathbb{R}^d$, $g_1(x)=\Tr[\sigma \sigma^{T}(x)]^2+\vert b(x) \vert^2+ \vert D \sigma(x)  \vert^2 \Tr[\sigma \sigma^{T}(x)]+\Tr[\tilde{\sigma} \tilde{\sigma}^{T}(x)]+\vert Ab(x) \vert$. Finally assume that $\mathbb{P}-a.s.$, for every $f \in F_1$, $\lim_{n \to \infty}\nu ^{\tilde{\eta}_{1}}_n(\mathfrak{M}_1 f) =\nu(\mathfrak{M}_1 f)$. \\

Then $\mathcal{E}_1(F_1,\tilde{A},A, \mathfrak{M}_1,\tilde{\eta}_1)$ (see (\ref{hyp:rate_erreur_tems_cours_fonction_test_reg})) is satisfied.

\item \label{prop:Talay_infinitesimal_approx_point_C}
Let $F_2 = \{f \in \mathcal{C}^6(\mathbb{R}^d;\mathbb{R}), \forall q\in \{2,\ldots,6\},  D^qf \in \mathcal{C}_0(\mathbb{R}^d;\mathbb{R}), Af \in F_1 \}$, let $\mathfrak{M}_2$ defined in (\ref{eq:def_M2_tilde}) and let $\tilde{\eta}_2(\gamma)=\gamma^3$. \\

Assume that the sequence $(U_n)_{n \in \mathbb{N}^{\ast}}$ satisfies $M_{\mathcal{N},5}(U)$ (see (\ref{hyp:matching_normal_moment_ordre_q_va_schema_Talay})) and $M_3(U)$ (see (\ref{hyp:moment_ordre_p_va_schema_Talay}))  and that $\sup_{n \in \mathbb{N}^{\ast}} \nu_n^{\tilde{\eta}_2}( g_2 )< + \infty$ with $g_2: \mathbb{R}^d \to \mathbb{R}$ such that for every $x \in \mathbb{R}^d$, $ g_2(x) =\Tr[\sigma \sigma^{T}(x)]^3+\vert b(x) \vert^3+ \vert D \sigma(x)  \vert^3 \Tr[\tilde{\sigma} \tilde{\sigma}^{T}(x)]^{3/2} +\Tr[\tilde{\sigma} \tilde{\sigma}^{T}(x)]+\vert Ab(x) \vert^2$. Finally assume that $\mathbb{P}-a.s.$, for every $f \in F_2$, $\lim_{n \to \infty}\nu ^{\tilde{\eta}_{2}}_n(\mathfrak{M}_2 f) = \nu(\mathfrak{M}_2 f)$. \\

Then $\mathcal{E}_2(F_2,\tilde{A},A, \mathfrak{M}_2,\tilde{\eta}_2)$ (see (\ref{hyp:rate_erreur_tems_cours_fonction_test_reg})) is satisfied.

\end{enumerate}
\end{myprop}

\begin{proof} 
The proof of point \ref{prop:Talay_infinitesimal_approx_point_A} is very similar to the proof of point \ref{prop:Talay_infinitesimal_approx_point_B} and point \ref{prop:Talay_infinitesimal_approx_point_C} but simpler and thus left to the reader. The proof of point \ref{prop:Talay_infinitesimal_approx_point_B} and point \ref{prop:Talay_infinitesimal_approx_point_C} is a direct consequence of the following Lemma.
\begin{lemme}
\label{lemme:Talay_infinitesimal_approx}
Assume that $b$, $\sigma$, $\vert D\sigma \vert \Tr[\sigma \sigma^{T}]^{1/2}$, $\tilde{\sigma}$ and $Ab$ have sublinear growth. We have the following properties:
\begin{enumerate}[label=\textbf{\Alph*.}] 

\item \label{lemme:Talay_infinitesimal_approx_point_A} Assume that the sequence $(U_n)_{n \in \mathbb{N}^{\ast}}$ satisfies $M_{\mathcal{N},3}(U)$ (see (\ref{hyp:matching_normal_moment_ordre_q_va_schema_Talay})) and $M_2(U)$ (see (\ref{hyp:moment_ordre_p_va_schema_Talay})). \\

Then, for every $f \in \mathcal{C}^4(\mathbb{R}^d;\mathbb{R})$ such that $D^qf \in \mathcal{C}_0(\mathbb{R}^d;\mathbb{R})$ for $q\in \{1,\ldots,4\}$, then

\begin{align*}
\Big\vert \mathbb{E}[ f(\overline{X}_{\Gamma_{n+1}})- f(\overline{X}_{\Gamma_{n}}) \vert \overline{X}_{\Gamma_{n}}  ] -\gamma_{n+1} Af(\overline{X}_{\Gamma_{n}})+ & \gamma_{n+1}^{2} \mathfrak{M}_1f(\overline{X}_{\Gamma_{n}}) \Big\vert \\
 \leqslant & \gamma_{n+1}^2 \Lambda_{f,1}(\overline{X}_{\Gamma_{n}},\gamma_{n+1})  ,
\end{align*}
with, given $l \in \mathbb{N}^{\ast}$ and a probability space $(\tilde{\Omega},\tilde{\mathcal{G}},\tilde{\mathbb{P}})$,
\begin{align*}
\forall x \in \mathbb{R}^d , \forall \gamma \in (0,\overline{\gamma}], \qquad \Lambda_{f,1}(x,\gamma)= \langle g_1 (  x ) ,\tilde{\mathbb{E}} [\tilde{\Lambda}_{f,1}(x,\gamma, \tilde{\omega})]  \rangle_{\mathbb{R}^l},
\end{align*}
with $\tilde{\Lambda}_{f,1}$ satisfying (\ref{eq:domination_erreur_L1}) and (\ref{hyp:erreur_temps_cours_fonction_test_reg_Lambda_representation_2_1}), $\mathfrak{M}_1$ defined in (\ref{eq:def_M1}) and $g_1: \mathbb{R}^d \to \mathbb{R}^l$, such that for every $x \in \mathbb{R}^d$, $\vert g_1(x) \vert \leqslant   1+\Tr[\sigma \sigma^{T}(x)]^2+\vert b(x) \vert^2+ \vert D \sigma(x)  \vert^2 \Tr[\sigma \sigma^{T}(x)]+\Tr[\tilde{\sigma} \tilde{\sigma}^{T}(x)]+\vert Ab(x) \vert$.

\item \label{lemme:Talay_infinitesimal_approx_point_B} Assume that the sequence $(U_n)_{n \in \mathbb{N}^{\ast}}$ satisfies $M_{\mathcal{N},5}(U)$ (see (\ref{hyp:matching_normal_moment_ordre_q_va_schema_Talay})) and $M_3(U)$ (see (\ref{hyp:moment_ordre_p_va_schema_Talay})). \\

Then, for every $f \in \mathcal{C}^6(\mathbb{R}^d;\mathbb{R})$ such that $D^qf \in \mathcal{C}_0(\mathbb{R}^d;\mathbb{R})$ for $q\in \{2,\ldots,6\}$, then

\begin{align*}
\Big\vert \mathbb{E}[ f(\overline{X}_{\Gamma_{n+1}})- f(\overline{X}_{\Gamma_{n}}) \vert \overline{X}_{\Gamma_{n}}  ] -\gamma_{n+1} Af(\overline{X}_{\Gamma_{n}})-\frac{\gamma_{n+1}^2}{2} A^2f(\overline{X}_{\Gamma_{n}})+ & \gamma_{n+1}^{3} \tilde{\mathfrak{M}}_2f(\overline{X}_{\Gamma_{n}}) \Big\vert \\
\leqslant & \gamma_{n+1}^3 \Lambda_{f,2}(\overline{X}_{\Gamma_{n}},\gamma_{n+1})  ,
\end{align*}
with, given $l \in \mathbb{N}^{\ast}$ and a probability space $(\tilde{\Omega},\tilde{\mathcal{G}},\tilde{\mathbb{P}})$,
\begin{align*}
\forall x \in \mathbb{R}^d , \forall \gamma \in (0,\overline{\gamma}], \qquad \Lambda_{f,2}(x,\gamma)=  \langle g_2 (  x ) ,\tilde{\mathbb{E}} [\tilde{\Lambda}_{f,2}(x,\gamma, \tilde{\omega})]  \rangle_{\mathbb{R}^l},
\end{align*}
with $\tilde{\Lambda}_{f,2}$ satisfying (\ref{eq:domination_erreur_L1}) and (\ref{hyp:erreur_temps_cours_fonction_test_reg_Lambda_representation_2_1}) and $\tilde{\mathfrak{M}}_2$ defined in (\ref{eq:def_M2_tilde}) and $g_2: \mathbb{R}^d \to \mathbb{R}^l$, such that for every $x \in \mathbb{R}^d$, $\vert g_2(x) \vert \leqslant  1+\Tr[\sigma \sigma^{T}(x)]^3+\vert b(x) \vert^3+ \vert D \sigma(x)  \vert^3 \Tr[\sigma \sigma^{T}(x)]^{3/2} +\Tr[\tilde{\sigma} \tilde{\sigma}^{T}(x)]+\vert Ab(x) \vert^2$.

\end{enumerate}
\end{lemme}

Notice that to obtain Proposition \ref{prop:Talay_infinitesimal_approx} point \ref{prop:Talay_infinitesimal_approx_point_B}, we use Lemma \ref{lemme:Talay_infinitesimal_approx} point \ref{lemme:Talay_infinitesimal_approx_point_A} and to obtain Proposition \ref{prop:Talay_infinitesimal_approx} point \ref{prop:Talay_infinitesimal_approx_point_C}, we combine Lemma \ref{lemme:Talay_infinitesimal_approx} point \ref{lemme:Talay_infinitesimal_approx_point_A} (with $f$ replaced by $Af$) and Lemma \ref{lemme:Talay_infinitesimal_approx} point \ref{lemme:Talay_infinitesimal_approx_point_B}
\begin{proof}[Proof of Lemma \ref{lemme:Talay_infinitesimal_approx}]
We simply prove point point \ref{lemme:Talay_infinitesimal_approx_point_B}. The proof of point point \ref{lemme:Talay_infinitesimal_approx_point_A} is similar but simpler. The first step consists in writing the following decomposition
\begin{align*}
f(\overline{X}_{\Gamma_{n+1}})- f(\overline{X}_{\Gamma_{n}}) =\sum_{j=0}^4 f( \overline{X}^{j}_{\Gamma_{n}} ) - f( \overline{X}^{j-1}_{\Gamma_{n}} )
\end{align*}
with notations (\ref{def:incr_Talay}) and $\overline{X}^{0}_{\Gamma_{n}}=\overline{X}_{\Gamma_{n}}$. At this point it remains to study each term of the sum of the $r.h.s.$ of the above equation. For $j=1$, we use Taylor expansion at order 6 and it follows that
\begin{align*}
\vert \mathbb{E}[f( \overline{X}^{1}_{\Gamma_{n}} ) \vert \overline{X}_{\Gamma_{n}} ] -  f( \overline{X}_{\Gamma_{n}} ) \vert \leqslant & \sum_{ i=1}^{6} \frac{\gamma_{n+1}^{i/2}( D^i f( \overline{X}_{\Gamma_{n}} ) ; \mathbb{E}[(\sigma( \overline{X}_{\Gamma_{n}}) U_{n+1})^{\otimes i} ) \vert \overline{X}_{\Gamma_{n}} ] )}{i ! } \\
&+ \gamma_{n+1}^3 \Lambda_{f,2,1}(\overline{X}_{\Gamma_{n}},\gamma_{n+1}) \\
\end{align*}
with $\Lambda_{f,2,1}(x, \gamma)=g_{2,1} (x)  \tilde{\mathbb{E}}[\tilde{\Lambda}_{f,2,1}(x,z,\gamma)]$ where $\tilde{\Lambda}_{f,2,1} ( x,\gamma)= \tilde{\mathcal{R}}_{f,2,1}(x,z,\gamma,U, \Theta)$ with $U \sim \mathbb{P}_U$, $\Theta \sim \mathcal{U}_{[0,1]}$ under $\tilde{\mathbb{P}}$, $g_{2,1}(x)=  \Tr[ \sigma \sigma^{T}(x) ]^3$ and
 \begin{align*}
\begin{array}{crcl}
\tilde{\mathcal{R}}_{f,2,1} & :  \mathbb{R}^d  \times \mathbb{R}_+ \times \mathbb{R}^{N }  \times [0,1] & \to & \mathbb{R}_+ \\
 &( x, \gamma , u, \theta ) & \mapsto &  \tilde{\mathcal{R}}_{f,2,1} ( x, \gamma , u, \theta )  ,
\end{array}
\end{align*}
with
\begin{align*}
\tilde{\mathcal{R}}_{f,2,1} ( x, \gamma , u, \theta )  =\frac{\vert u \vert^6}{5!}  (1- \theta)^{5} \vert D^6  f( x + \theta \sqrt{\gamma} \sigma(  x) u)- D^6  f( x ) \vert .
\end{align*}

We are going to prove that $\tilde{\Lambda}_{f,2,1}$ satisfies (\ref{hyp:erreur_temps_cours_fonction_test_reg_Lambda_representation_2_1}). We fix $u \in \mathbb{R}^N$ and $\theta \in [0,1]$. Now, since the function $ \sigma$ has sublinear growth, there exists $C_{ \sigma} \geqslant 0$ such that $ \vert \sigma(x) \vert \leqslant C_{ \sigma}(1+ \vert x \vert ) $ for every $x \in \mathbb{R}^d$. Therefore, since $f$ has compact support, there exists $\underline{\gamma}(u,\theta)>0$ and $R>0$ such that 
\begin{align*}
\sup_{\vert x \vert >R} \sup_{\gamma \leqslant \underline{\gamma}(u,\theta)} \tilde{\mathcal{R}}_{f,2,1}(x,\gamma,u,\theta)=0.
\end{align*}
 It follows that (\ref{hyp:erreur_temps_cours_fonction_test_reg_Lambda_representation_2_1}) (ii) holds. Moreover since $D^6 f$ is bounded, and $M_3(U)$ (see (\ref{hyp:moment_ordre_p_va_schema_Talay})) holds, $\tilde{\Lambda}_{f,2}$ also satisfies (\ref{eq:domination_erreur_L1}).\\
 
 The rest of the proof is completely similar and involves heavy calculus so we just give the sketch to follow for $j=2$ and invite the reader to follow the same line for $j \in \{3,4,5\}$. For $j=2$, we use Taylor expansion at order 3 and it follows that
\begin{align*}
\vert \mathbb{E}[f( \overline{X}^{2}_{\Gamma_{n}} ) \vert \overline{X}_{\Gamma_{n}} ] -  f( \overline{X}^{1}_{\Gamma_{n}} ) \vert \leqslant & \sum_{ i=1}^{3} \frac{\gamma_{n+1}^{i}  \mathbb{E}[( D^i f( \overline{X}^{1}_{\Gamma_{n}} ) ;(b( \overline{X}_{\Gamma_{n}}) )^{\otimes i} ) \vert \overline{X}_{\Gamma_{n}}  ]}{i ! } \\
&+ \frac{\gamma_{n+1}^3}{2} (D^3f(\overline{X}_{\Gamma_{n}});b(\overline{X}_{\Gamma_{n}})^{\otimes 3}) + \gamma_{n+1}^3 \Lambda_{f,2,2}(\overline{X}_{\Gamma_{n}},\gamma_{n+1}) 
\end{align*}
 
with $\Lambda_{f,2,2}(x, \gamma)=g_{2,2} (x)  \tilde{\mathbb{E}}[\tilde{\Lambda}_{f,2,2}(x,z,\gamma)]$ where $\tilde{\Lambda}_{f,2,2} ( x,\gamma)= \tilde{\mathcal{R}}_{f,2,2}(x,z,\gamma,U, \Theta)$ with $U \sim \mathbb{P}_U$, $\Theta \sim \mathcal{U}_{[0,1]}$ under $\tilde{\mathbb{P}}$, $g_{2,2}(x)=  \vert b(x) \vert^3 $ and
 \begin{align*}
\begin{array}{crcl}
\tilde{\mathcal{R}}_{f,2,1} & :  \mathbb{R}^d  \times \mathbb{R}_+ \times \mathbb{R}^{N }  \times [0,1] & \to & \mathbb{R}_+ \\
 &( x, \gamma , u, \theta ) & \mapsto &  \tilde{\mathcal{R}}_{f,2,1} ( x, \gamma , u, \theta )  ,
\end{array}
\end{align*}
with
\begin{align*}
\tilde{\mathcal{R}}_{f,2,2} ( x, \gamma , u, \theta )  =\frac{1}{2}  (1- \theta)^{2} \vert D^3  f( x + \sqrt{\gamma} \sigma(  x) u+\theta \gamma b(x))- D^3  f( x ) \vert .
\end{align*}
Following the same approach as for the case $j=1$ we can show that $\tilde{\Lambda}_{f,2,2}$ satisfies (\ref{hyp:erreur_temps_cours_fonction_test_reg_Lambda_representation_2_1}) and (\ref{eq:domination_erreur_L1}).

To complete the study for $j=1$, we replace $ D^i f( \overline{X}^{1}_{\Gamma_{n}}$, $i \in \{1,2\}$ by an upper bound of their Taylor expansion at order $2(3-i)$ and at point $\overline{X}^{j-1}_{\Gamma_{n}}=\overline{X}_{\Gamma_{n}}$, that is

\begin{align*}
\vert \mathbb{E}[D^i f( \overline{X}^{1}_{\Gamma_{n}} )\vert \overline{X}_{\Gamma_{n}}   ] - D^i f( \overline{X}_{\Gamma_{n}} ) \vert\leqslant &  \sum_{ \bar{i}=1}^{2(3-i)}  \frac{\gamma_{n+1}^{\bar{i}/2}( D^{\bar{i}+i} f( \overline{X}_{\Gamma_{n}} ) ; \mathbb{E}[(\sigma( \overline{X}_{\Gamma_{n}}) U_{n+1})^{\otimes \bar{i}} ) \vert \overline{X}_{\Gamma_{n}} ] )}{\bar{i} ! } \\
&+ \gamma_{n+1}^{3-i} \Lambda_{D^if,2,1}(\overline{X}_{\Gamma_{n}},\gamma_{n+1}) 
\end{align*}
with $\Lambda_{D^if,2,2}(x, \gamma)=\Tr[\sigma \sigma^{T}(x)]^{3-i}  \tilde{\mathbb{E}}[\tilde{\Lambda}_{D^i f,2,2}(x,z,\gamma)]$ where $\tilde{\Lambda}_{D^if ,2,2} ( x,\gamma)= \tilde{\mathcal{R}}_{D^i f,2,2}(x,z,\gamma,U, \Theta)$ with $U \sim \mathbb{P}_U$, $\Theta \sim \mathcal{U}_{[0,1]}$ under $\tilde{\mathbb{P}}$, and
 \begin{align*}
\begin{array}{crcl}
\tilde{\mathcal{R}}_{D^i f,2,2} & :  \mathbb{R}^d  \times \mathbb{R}_+ \times \mathbb{R}^{N }  \times [0,1] & \to & \mathbb{R}_+ \\
 &( x, \gamma , u, \theta ) & \mapsto &  \tilde{\mathcal{R}}_{D^i f,2,2} ( x, \gamma , u, \theta )  ,
\end{array}
\end{align*}
with
\begin{align*}
\tilde{\mathcal{R}}_{D^i f,2,2} ( x, \gamma , u, \theta )  =\frac{\vert u \vert^{2(3-i)}}{(5-2i)!}  (1- \theta)^{5-2i} \vert D^{3-i}  f( x + \theta \sqrt{\gamma} \sigma(  x) u)- D^{3-i}  f( x ) \vert .
\end{align*}
Following the same approach as for the case $j=1$ we can show that $\tilde{\Lambda}_{D^if,2,2}$ satisfies (\ref{hyp:erreur_temps_cours_fonction_test_reg_Lambda_representation_2_1}) and (\ref{eq:domination_erreur_L1}).
We do not detail the rest of the proof which is similar but simply describe the approach we use. For $j=\{3,4,5\}$ we apply the same method as for $j=2$: We first use the Taylor expansion at point $\overline{X}^{j-1}_{\Gamma_{n}}$ such that the remainder has the form $\gamma_{n+1}^3 \Lambda_{f,2,j}$. Then we develop each term of this expansion at point $\overline{X}^{j-2}_{\Gamma_{n}}$ at a well chosen order such that the global remainder is still of the form $\gamma_{n+1}^3 \Lambda_{f,2,j}$ ($\Lambda_{f,2,j}$ is obviously changed). We iterate the method until we use the Taylor expansion at point $\overline{X}_{\Gamma_{n}}$. Then, the final remainder $\Lambda_{f,2}$ has the expected form and the term which appears in the expansion can be identified with $\gamma_{n+1} Af(\overline{X}_{\Gamma_{n}})+\frac{\gamma_{n+1}^2}{2} A^2f(\overline{X}_{\Gamma_{n}})+ \gamma_{n+1}^{3} \tilde{\mathfrak{M}}_2f(\overline{X}_{\Gamma_{n}}) $. To complete the proof we notice that for every $f \in \mathcal{C}^6(\mathbb{R}^d)$ and every $j \in \{1, \ldots, 5\}$, $\tilde{\mathcal{R}}_{f,2,j}=\tilde{\mathcal{R}}_{-f,2,j}$.
\end{proof}

\end{proof}

\subsection{Growth control}

\begin{lemme}
\label{lemme:incr_lyapunov_X_Talay}
 Let $p > 0,a \in (0,1]$, $s \geqslant 1$, $\rho \in [1,2]$ and, $\psi(y)=y^p$ and $\phi(y)=y^a$. We suppose that the sequence $(U_n)_{n \in \mathbb{N}^{\ast}}$ satisfies $M_{\rho \vee (2 p \rho /s)}(U)$ (see (\ref{hyp:moment_ordre_p_va_schema_Talay})).
Then, for every $n \in \mathbb{N}$ and every $f \in \DomA_0$,

\begin{align}
\label{eq:incr_lyapunov_X_Talay_f_DomA}
 \mathbb{E}[  \vert f(\overline{X}_{\Gamma_{n+1}})-  & f( \overline{X}_{\Gamma_{n}} + \gamma_{n+1} b( \overline{X}_{\Gamma_{n}}) +  \gamma_{n+1}^2 \frac{1}{2}Ab( \overline{X}_{\Gamma_{n}}))\vert^{\rho}\vert \overline{X}_{\Gamma_{n}}  ]  \\
  \leqslant &   C_f \gamma_{n+1}^{\rho/2}  \Tr [\sigma \sigma^{T} (\overline{X}_{\Gamma_n} )  ]^{\rho/2}  + C_f \gamma_{n+1}^{\rho}   \vert D\sigma \vert^{\rho} \Tr[\sigma \sigma^{T}]^{\rho/2}  + C_f \gamma_{n+1}^{\rho 3/2}  \Tr [\tilde{\sigma} \tilde{\sigma}^{T} (\overline{X}_{\Gamma_n} )  ]^{\rho/2} . \nonumber
\end{align}
with $\DomA_0 =\mathcal{C}^2_K (\mathbb{R}^d )$. In other words, we have $\mathcal{GC}_{Q}(\DomA_0,  g_{\sigma},\rho,\epsilon_{\mathcal{I}}) $ (see (\ref{hyp:incr_X_Lyapunov})) with $g_{\sigma}=\Tr[ \sigma \sigma^{T} ]^{ \rho/2}+ \vert D\sigma \vert^{\rho} \Tr[\sigma \sigma^{T}]^{\rho/2} +\Tr [\tilde{\sigma} \tilde{\sigma}^{T} (\overline{X}_{\Gamma_n} )  ]^{\rho/2}$ and $\epsilon_{\mathcal{I}}(\gamma)=\gamma^{\rho/2}$ for every $\gamma \in \mathbb{R}_+$. \\


Moreover, if (\ref{hyp:Lyapunov_control_Talay}) and $\mathfrak{B}(\phi)$ (see (\ref{hyp:controle_coefficients_Talay})) hold and 


\begin{align}
\label{hyp:control_step_weight_pol_Talay}
\mathcal{S} \mathcal{W}_{pol}(p,a,s,\rho) \qquad ap \rho/s \leqslant p+a-1.
\end{align}

Then, for every $n \in \mathbb{N}$, we have
\begin{align}
\label{eq:incr_lyapunov_X_Talay_f_tens}
 \mathbb{E}[\vert  V^{p/s}(\overline{X}_{\Gamma_{n+1}})-V^{p/s}(\overline{X}_{\Gamma_{n}}) \vert^{\rho}\vert \overline{X}_{\Gamma_{n}}]  
\leqslant  C \gamma_{n+1}^{\rho/2} V^{p+a-1}(\overline{X}_{\Gamma_{n}}).
\end{align}
In other words, we have $\mathcal{GC}_{Q}(V^{p/s},V^{p+a-1},\rho,\epsilon_{\mathcal{I}}) $ (see (\ref{hyp:incr_X_Lyapunov})) with and $\epsilon_{\mathcal{I}}(\gamma)=\gamma^{\rho/2}$ for every $\gamma \in \mathbb{R}_+$.

\end{lemme}

\begin{proof}
We begin by noticing that

\begin{align*}
 \vert \overline{X}_{\Gamma_{n+1}} -&( \overline{X}_{\Gamma_{n}} + \gamma_{n+1} b( \overline{X}_{\Gamma_{n}}) +  \gamma_{n+1}^2 \frac{1}{2} Ab( \overline{X}_{\Gamma_{n}})) \vert \\
 \leqslant & C  \gamma_{n+1}^{1/2} \Tr [\sigma \sigma^{T} (\overline{X}_{\Gamma_n}) ]^{1/2} \vert U_{n+1} \vert  +C\gamma_{n+1}\vert D\sigma \vert \Tr[\sigma \sigma^{T}]^{1/2}  \vert \mathcal{W}_{n+1} \vert+\gamma_{n+1}^{ 3/2}  \Tr [\tilde{\sigma} \tilde{\sigma}^{T} (\overline{X}_{\Gamma_n} )  ]^{1/2} \vert U_{n+1} \vert  .
\end{align*}

Let $f \in \DomA$. Then $f$ is Lipschitz and the previous inequality gives (\ref{eq:incr_lyapunov_X_Talay_f_DomA}). Using Remark \ref{rmrk:Accroiss_mes}, we obtain $\mathcal{GC}_{Q}(\DomA_0,  g_{\sigma},\rho,\epsilon_{\mathcal{I}}) $.\\

We focus now on the proof of (\ref{eq:incr_lyapunov_X_Talay_f_tens}).
We first notice that $\mathfrak{B}(\phi)$ (see (\ref{hyp:controle_coefficients_Talay})) implies that for any $n \in \mathbb{N}$,
\begin{align*}
 \vert \overline{X}_{\Gamma_{n+1}} -\overline{X}_{\Gamma_n} \vert\leqslant C  \gamma_{n+1}^{1/2}\sqrt{\phi \circ V (\overline{X}_{\Gamma_n})}(1+\vert U_{n+1}\vert +\vert \mathcal{W}_{n+1} \vert \vert )
\end{align*}  

 \paragraph{Case $2p\leqslant s$.} 
 
 We notice that  $V^{p/s}$ is $\alpha$-Hölder for any $\alpha \in [2p/s,1]$ (see Lemma 3. in \cite{Panloup_2008}) and then $V^{p/s}$ is $2p/s$-Hölder. We deduce that
\begin{align*}
\mathbb{E} [ \vert V^{p/s}( \overline{X}_{\Gamma_{n+1}} )-& V^{p/s}(\overline{X}_{\Gamma_n}) \vert^{\rho} \vert \overline{X}_{\Gamma_n} ] \leqslant  C[V^{p/s}]^{\rho}_{2p/s} \gamma_{n+1}^{ \rho/2}  V^{a \rho /2}(\overline{X}_{\Gamma_{n}}).
\end{align*}
In order to obtain (\ref{eq:incr_lyapunov_X_Talay_f_tens}), it remains to use $ap \rho /s \leqslant a+p-1$.\\

\paragraph{Case $2p\geqslant s$.}  Using  the following inequality
 \begin{align*}
\forall u,v \in \mathbb{R}_+,\forall \alpha \geqslant 1, \qquad \vert u^{\alpha} -v^{\alpha} \vert \leqslant &  \alpha 2^{\alpha-1} ( v^{\alpha-1} \vert u -v \vert + \vert u -v \vert ^{\alpha} ),
 \end{align*}
with $\alpha=2p/s$, and since $\sqrt{V}$ is Lipschitz, we have

\begin{align*}
\big \vert V^{p/s}( \overline{X}_{\Gamma_{n+1}} )-V^{p/s}(\overline{X}_{\Gamma_n}) \big \vert \leqslant &  2^{2p/s}p/s ( V^{p/s-1/2}(\overline{X}_{\Gamma_n}) \vert \sqrt{V}( \overline{X}_{\Gamma_{n+1}} )  - \sqrt{V} (\overline{X}_{\Gamma_n} ) \vert  \\
 & + \vert  \sqrt{V}( \overline{X}_{\Gamma_{n+1}} )  - \sqrt{V}(\overline{X}_{\Gamma_n} ) \vert^{2p/s}  ) \\
\leqslant &  2^{2p/s}p/s ( [ \sqrt{V}]_1 V^{p/s-1/2}(\overline{X}_{\Gamma_n} ) \vert  \overline{X}_{\Gamma_{n+1}}- \overline{X}_{\Gamma_n} \vert \\
& + [ \sqrt{V}]_1^{2p/s}  \vert   \overline{X}_{\Gamma_{n+1}} -\overline{X}_{\Gamma_n}\vert^{2p/s}   ).
\end{align*}
In order to obtain (\ref{eq:incr_lyapunov_X_Talay_f_tens}), it remains to use the assumptions $\mathfrak{B}(\phi)$ (see (\ref{hyp:controle_coefficients_Talay}))  and then $ap\rho/s \leqslant p+a-1$.

\end{proof}

\begin{lemme}
\label{lemme:incr_lyapunov_X_Talay_TCL}
 Let $\rho \in [1,2]$ and, $\psi(y)=y^p$ and $\phi(y)=y^a$. We suppose that the sequence $(U_n)_{n \in \mathbb{N}^{\ast}}$ satisfies $M_{\rho}(U)$ (see (\ref{hyp:moment_ordre_p_va_schema_Talay})).
Then, for every $n \in \mathbb{N}$, we have: for every $f \in F=\{f \in\mathcal{C}^2(\mathbb{R}^d;\mathbb{R}), D^qf \in \mathcal{C}_b(\mathbb{R}^d;\mathbb{R}), \forall q\in \{1,2\} \}.$
\begin{align}
\label{eq:incr_lyapunov_X_Talay_f_DomA_TCL}
 \mathbb{E}[  \vert f(\overline{X}_{\Gamma_{n+1}})- & f( \overline{X}_{\Gamma_{n}}) - \sqrt{\gamma_{n+1}}(Df(\overline{X}_{\Gamma_{n}});\sigma(\overline{X}_{\Gamma_{n}})U_{n+1}) \vert^{\rho}\vert \overline{X}_{\Gamma_{n}}  ]   \\
 \leqslant&   C_f \gamma_{n+1}^{\rho}  \Tr [\sigma \sigma^{T} (\overline{X}_{\Gamma_n} )  ]^{\rho} +C_f \gamma_{n+1}^{\rho}  \vert  b(X_n) \vert  + C_f \gamma_{n+1}^{\rho}  \vert D\sigma(\overline{X}_{\Gamma_n}) \vert^{\rho} \Tr[\sigma \sigma^{T}(\overline{X}_{\Gamma_n})]^{\rho/2}  \nonumber \\
 &+ C_f \gamma_{n+1}^{\rho 3/2}  \Tr [\tilde{\sigma} \tilde{\sigma}^{T} (\overline{X}_{\Gamma_n} )  ]^{\rho/2}   +  C_f \gamma_{n+1}^{2 \rho}  \vert Ab(\overline{X}_{\Gamma_n}) \vert^{\rho} .  \nonumber
\end{align}
In particular for $q \in \{1,2\}$, assume that $\mathbb{P}-a.s.$, $\lim_{n \to + \infty} \nu^{\gamma}_n(\vert \sigma^{T} Df \vert^2)=\nu(\vert \sigma^{T} Df \vert^2)$ for every $f \in F$ satisfying $Af \in \mathcal{C}_b(\mathbb{R}^d; \mathbb{R})$ when $q=2$ and that $\Tr[ \sigma \sigma^{T} ]=o_{\vert x \vert \to +\infty}(W)$ with $\sup_{n \in \mathbb{N}^{\ast}}\nu^{\gamma}_n(W)<+\infty$. \\

Then $\mathcal{GC}_{Q,q}(F,g,\rho,\epsilon_{\mathscr{X}},\epsilon_{\mathcal{G}\mathcal{C}},\mathfrak{V})$ (see (\ref{hyp:incr_X_Lyapunov_vitesse})) is satisfied with $g=\Tr[ \sigma \sigma^{T} ]^{ \rho}+ \vert b \vert^{\rho}+ \vert D\sigma \vert^{\rho} \Tr[\sigma \sigma^{T}]^{\rho/2}  +\Tr [\tilde{\sigma} \tilde{\sigma}^{T}  ]^{\rho/2}+\vert Ab \vert^{\rho}$, $\epsilon_{\mathscr{X}}(\gamma)= \gamma$ and $\epsilon_{\mathcal{G}\mathcal{C}}(\gamma)=\gamma^{\rho}$ for every $\gamma \in \mathbb{R}_+$ and $\mathfrak{V}f=\vert \sigma^{T} Df \vert^2$ for every $f \in\mathcal{C}^1(\mathbb{R}^d;\mathbb{R})$.

\end{lemme}

\begin{proof} 
The first step consists in writing
\begin{align}
\label{preuve:decomp_incr_lyapunov_X_Talay_TCL}
f(\overline{X}_{\Gamma_{n+1}})- f( \overline{X}_{\Gamma_{n}}) =&  f( \overline{X}_{\Gamma_{n}}+ \sqrt{\gamma_{n+1}}\sigma(\overline{X}_{\Gamma_{n}})U_{n+1} )   - f( \overline{X}_{\Gamma_{n}})  \\
&+ f(\overline{X}_{\Gamma_{n+1}}) - f( \overline{X}_{\Gamma_{n}}+ \sqrt{\gamma_{n+1}}\sigma(\overline{X}_{\Gamma_{n}})U_{n+1} ) .  \nonumber 
\end{align}
We study the first term of the $r.h.s.$ of the above equation. Using Taylor expansion at order two and the fact that $Df \in \mathcal{C}_b(\mathbb{R}^d)$ yields
\begin{align*}
\big\vert f( \overline{X}_{\Gamma_{n}}+ \sqrt{\gamma_{n+1}}\sigma(\overline{X}_{\Gamma_{n}})U_{n+1} )   - f( \overline{X}_{\Gamma_{n}}) & - \sqrt{\gamma_{n+1}}(Df(\overline{X}_{\Gamma_{n}});\sigma(\overline{X}_{\Gamma_{n}})U_{n+1}) \big \vert \\
&\leqslant \frac{1}{2} \Vert D^2f \Vert_{\infty}  \vert  \sqrt{\gamma_{n+1}}\sigma(\overline{X}_{\Gamma_{n}})U_{n+1}  \vert^2 .
\end{align*}
Now we study the second term of the $r.h.s.$ of (\ref{preuve:decomp_incr_lyapunov_X_Talay_TCL}). From Taylor expansion at order one
\begin{align*}
 \vert f(\overline{X}_{\Gamma_{n+1}}) - f( \overline{X}_{\Gamma_{n}}+ \sqrt{\gamma_{n+1}}\sigma(\overline{X}_{\Gamma_{n}})U_{n+1} ) \vert
\leqslant & \Vert Df \Vert_{\infty}  \Big\vert \gamma_{n+1}  \Big( b(\overline{X}_{\Gamma_{n}}) + (D \sigma( \overline{X}_{\Gamma_n} ) ; \sigma( \overline{X}_{\Gamma_n} )\mathcal{W}_{n+1}^{T}  )\Big)  \\
& + \gamma_{n+1}^{3/2} \tilde{\sigma}( \overline{X}_{\Gamma_n} ) U_{n+1}  + \gamma_{n+1}^{2} Ab( \overline{X}_{\Gamma_n} )  \Big\vert .
\end{align*}
Gathering both terms of (\ref{preuve:decomp_incr_lyapunov_X_Talay_TCL}), raising to the power $\rho$ and taking conditional expectancy thus yields (\ref{eq:incr_lyapunov_X_Talay_f_DomA_TCL}). To obtain $\mathcal{GC}_{Q,q}(F,g,\rho,\epsilon_{\mathscr{X}},\epsilon_{\mathcal{G}\mathcal{C}},\mathfrak{V})$ (see (\ref{hyp:incr_X_Lyapunov_vitesse})), we observe that $Af$ is bounded when $q=2$ and it remains to show that (\ref{hyp:Lindeberg_CLT_scheme}) holds with $\mathscr{X}_{f,n} =\sqrt{\gamma_{n+1}}(Df(\overline{X}_{\Gamma_{n}});\sigma(\overline{X}_{\Gamma_{n}})U_{n+1})$, $n \in \mathbb{N}$. This is already done in the seminal paper \cite{Lamberton_Pages_2002} (see Proposition 2.) and we invite the reader to refer to this result.
\end{proof}

\subsection{Proof of Theorem \ref{th:cv_was_Talay}}
\paragraph{Proof of Theorem \ref{th:cv_was_Talay} point \ref{th:cv_was_Talay_point_A}\\}

This result follows from Theorem \ref{th:tightness} and Theorem \ref{th:identification_limit}. The proof consists in showing that the assumptions from those theorems are satisfied.

\paragraph{Step 1. Mean reverting recursive control}
First, we show that $\mathcal{RC}_{Q,V}(\psi_p,\phi,p\tilde{\alpha},p\beta)$ and $\mathcal{RC}_{Q,V}(\psi_{1},\phi,\tilde{\alpha},\beta)$ (see (\ref{hyp:incr_sg_Lyapunov})) is satisfied for every $\tilde{\alpha} \in (0,\alpha)$.\\

 Since (\ref{hyp:Lyapunov_control_Talay}), $\mathfrak{B}(\phi)$ (see (\ref{hyp:controle_coefficients_Talay})) and $\mathcal{R}_p(\alpha,\beta,\phi,V)$ (see (\ref{hyp:recursive_control_param_Talay})) hold, it follows from Proposition \ref{prop:recursive_control_Talay} that $\mathcal{RC}_{Q,V}(\psi_p,\phi,p\tilde{\alpha},p\beta)$ (see (\ref{hyp:incr_sg_Lyapunov})) is satisfied for every $\tilde{\alpha} \in (0,\alpha)$ since $\liminf_{y \to + \infty} \phi(y) > \beta / \tilde{\alpha}$. Moreover let us notice that for every $p \leqslant 1$ then $\mathcal{R}_p(\alpha,\beta, \phi, V)$ (see (\ref{hyp:recursive_control_param_Talay})) is similar to $\mathcal{R}_1(\alpha,\beta, \phi, V)$ and then $\mathcal{RC}_{Q,V}(\psi_1,\phi,\tilde{\alpha},\beta)$ (see (\ref{hyp:incr_sg_Lyapunov})) is satisfied for every $\tilde{\alpha} \in (0,\alpha)$

\paragraph{Step 2. Step weight assumption} 
Now, we show that $\mathcal{S}\mathcal{W}_{\mathcal{I}, \gamma,\eta}(V^{p \vee 1+a-1} ,\rho,\epsilon_{\mathcal{I}}) $ (see (\ref{hyp:step_weight_I_gen_chow})) and $\mathcal{S}\mathcal{W}_{\mathcal{II},\gamma,\eta}(V^{p \vee 1 +a-1}) $ (see (\ref{hyp:step_weight_I_gen_tens})) hold. \\

First we noticel that from Step1. the assumption $\mathcal{RC}_{Q,V}(\psi_{p \vee 1},\phi,(p \vee 1)\tilde{\alpha},(p \vee 1)\beta)$ (see (\ref{hyp:incr_sg_Lyapunov})) is satisfied for every $\tilde{\alpha} \in (0,\alpha)$. Then, using $\mathcal{S}\mathcal{W}_{\mathcal{I}, \gamma,\eta}(\rho, \epsilon_{\mathcal{I}})$ (see (\ref{hyp:step_weight_I})) with Lemma \ref{lemme:mom_V} gives $\mathcal{S}\mathcal{W}_{\mathcal{I}, \gamma,\eta}( V^{p \vee 1 +a-1},\rho,\epsilon_{\mathcal{I}}) $ (see (\ref{hyp:step_weight_I_gen_chow})). Similarly, $\mathcal{S}\mathcal{W}_{\mathcal{II},\gamma,\eta}(V^{p \vee 1+a-1}) $ (see (\ref{hyp:step_weight_I_gen_tens}) follows from $\mathcal{S}\mathcal{W}_{\mathcal{II},\gamma,\eta} $ (see (\ref{hyp:step_weight_II})) and  Lemma \ref{lemme:mom_V}. 

\paragraph{Step 3. Growth control assumption}
Now, we prove $\mathcal{GC}_{Q}(F,V^{p \vee 1+a-1},\rho,\epsilon_{\mathcal{I}}) $ (see (\ref{hyp:incr_X_Lyapunov})) for $F= \DomA_0$ and $F=\{V^{p/s}\}$ .\\

This is a consequence of Lemma \ref{lemme:incr_lyapunov_X_Talay}. We recall that $\rho \in [1,2]$. Consequently $M_{\rho \vee( 2p\rho /s) }(U)$ (see (\ref{hyp:moment_ordre_p_va_schema_Talay})) holds.  Now, we notice that from $\mathfrak{B}(\phi)$ (see (\ref{hyp:controle_coefficients_Talay})), we have $\Tr[ \sigma \sigma^{T} ]^{ \rho/2}+ \vert D\sigma \vert^{\rho} \Tr[\sigma \sigma^{T}]^{\rho/2} +\Tr [\tilde{\sigma} \tilde{\sigma}^{T} ]^{\rho/2} \leqslant CV^{ \rho a/2} $ with $a \rho/2 \leqslant p+a-1$ since $\mathcal{S} \mathcal{W}_{pol}(p,a,s,\rho)$ (see (\ref{hyp:control_step_weight_pol_Talay})) holds. Then Lemma \ref{lemme:incr_lyapunov_X_Talay} implies that for $F= \DomA_0$ and $F=\{V^{p/s}\}$, then $\mathcal{GC}_{Q}(F,V^{p \vee 1 + a -1},\rho,\epsilon_{\mathcal{I}}) $ (see (\ref{hyp:incr_X_Lyapunov})) holds

\paragraph{Step 4. Conclusion}
\begin{enumerate}[label=\textbf{\roman*.}]
\item
The first part of Theorem \ref{th:cv_was_Talay} (see (\ref{eq:tightness_Talay})) is a consequence of Theorem \ref{th:tightness}. Let us observe that assumptions from Theorem \ref{th:tightness} indeed hold. \\

On the one hand, we observe that from Step 2. and Step 3. the assumptions $\mathcal{GC}_{Q}(V^{p/s},V^{p \vee 1 +a-1},\rho,\epsilon_{\mathcal{I}}) $ (see (\ref{hyp:incr_X_Lyapunov})), $\mathcal{S}\mathcal{W}_{\mathcal{I}, \gamma,\eta}( V^{p \vee 1 +a-1},\rho,\epsilon_{\mathcal{I}}) $ (see (\ref{hyp:step_weight_I_gen_chow})) and $\mathcal{S}\mathcal{W}_{\mathcal{II},\gamma,\eta}(V^{p \vee 1 +a-1}) $ (see (\ref{hyp:step_weight_I_gen_tens})) hold which are the hypotheses from Theorem \ref{th:tightness} point \ref{th:tightness_point_A} with $g=V^{p \vee 1 +a-1}$.\\

On the other hand, form Step 1. the assumption$\mathcal{RC}_{Q,V}(\psi_p,\phi,p\tilde{\alpha},p\beta)$ (see (\ref{hyp:incr_sg_Lyapunov})) is satisfied for every $\tilde{\alpha} \in (0,\alpha)$. Moreover, since $\mbox{L}_{V}$ (see (\ref{hyp:Lyapunov})) holds and  that $p/s+a-1>0$, then the hypotheses  from Theorem \ref{th:tightness} point \ref{th:tightness_point_B} are satisfied. \\

We thus conclude from Theorem \ref{th:tightness} that $(\nu_n^{\eta})_{n \in \mathbb{N}^{\ast}}$ is $\mathbb{P}-a.s.$ tight and (\ref{eq:tightness_Talay}) holds which concludes the proof of the first part of Theorem \ref{th:cv_was_Talay} point \ref{th:cv_was_Talay_point_A}. \\

\item Let us now prove the second part of Theorem \ref{th:cv_was_Talay} (see (\ref{eq:cv_was_Talay})) which is a consequence of Theorem \ref{th:identification_limit}.\\

On the one hand, we observe that from Step 2. and Step 3. the assumptions $\mathcal{GC}_{Q}(\DomA_0,V^{p \vee 1 +a-1},\rho,\epsilon_{\mathcal{I}}) $ (see (\ref{hyp:incr_X_Lyapunov})) and $\mathcal{S}\mathcal{W}_{\mathcal{I}, \gamma,\eta}( V^{p \vee 1 +a-1},\rho,\epsilon_{\mathcal{I}}) $ (see (\ref{hyp:step_weight_I_gen_chow})) hold which are the hypotheses from Theorem \ref{th:identification_limit} point \ref{th:identification_limit_A} with $g=V^{p \vee 1 +a-1}$.\\

On the other hand, since $b$, $\sigma$, $\vert D\sigma \vert \Tr[\sigma \sigma^{T}]^{1/2}$, $\tilde{\sigma}$ and $Ab$ have sublinear growth and that $g_{\sigma}\leqslant C  V^{p/s+a-1}$, with  $g_{\sigma}=\Tr[ \sigma \sigma^{T} ]+ \vert D\sigma \vert \Tr[\sigma \sigma^{T}]^{1/2}+\Tr[ \tilde{\sigma} \tilde{\sigma}^{T} ]^{1/2} $, so that $\mathbb{P} \mbox{-a.s.} \;\sup_{n \in \mathbb{N}^{\ast}} \nu_n^{\eta}( g_{\sigma}) < + \infty $,  it follows from Proposition \ref{prop:Talay_infinitesimal_approx} that $\mathcal{E}(\widetilde{A},A,\DomA_0) $ (see (\ref{hyp:erreur_tems_cours_fonction_test_reg})) is satisfied. Then, the hypotheses from Theorem \ref{th:identification_limit} point \ref{th:identification_limit_B} hold and (\ref{eq:cv_was_Talay}) follows from (\ref{eq:test_function_gen_cv}).

\end{enumerate}

\paragraph{Proof of Theorem \ref{th:cv_was_Talay} point \ref{th:cv_was_Talay_point_B}\\}

First we notice that using Theorem \ref{th:cv_was_Talay} point \ref{th:cv_was_Talay_point_A}, then for every $f \in F_q$, $\vert \sigma^{T} Df \vert^2 \in \mathcal{C}_{\tilde{V}_{\psi_p,\phi,s}}(\mathbb{R}^d)$, $\mathfrak{M}_q f \in \mathcal{C}_{\tilde{V}_{\psi_p,\phi,s}}(\mathbb{R}^d)$ and 
\begin{align*}
\mathbb{P}-a.s. \quad \lim_{n \to \infty} \nu^{\gamma}_n(\vert \sigma^{T} Df \vert^2 ) = \nu(\vert \sigma^{T} Df \vert^2 ) \qquad \mbox{and} \qquad \lim_{n \to \infty} \nu^{\tilde{\eta}_q}_n(\mathfrak{M}_q f) = \nu(\mathfrak{M}_q f ) .
\end{align*}

Now, we notice that using Proposition \ref{prop:Talay_infinitesimal_approx}, point \ref{prop:Talay_infinitesimal_approx_point_B} and point \ref{prop:Talay_infinitesimal_approx_point_C}, gives $\mathcal{E}_q(F_q,\tilde{A},A, \mathfrak{M}_q,\tilde{\eta}_q)$ (see (\ref{hyp:rate_erreur_tems_cours_fonction_test_reg})).\\

 Moreover, Lemma \ref{lemme:incr_lyapunov_X_Talay_TCL} gives $\mathcal{GC}_{Q,q}(F_q,g,\rho,\epsilon_{\mathscr{X}},\epsilon_{\mathcal{G}\mathcal{C}},\mathfrak{V})$ (see (\ref{hyp:incr_X_Lyapunov_vitesse})) with $g=\Tr[ \sigma \sigma^{T} ]^{ \rho}+ \vert b \vert^{\rho}+ \vert D\sigma \vert^{\rho} \Tr[\sigma \sigma^{T}]^{\rho/2}  +\Tr [\tilde{\sigma} \tilde{\sigma}^{T}  ]^{\rho/2}+\vert Ab \vert^{\rho}$, $\epsilon_{\mathscr{X}}(\gamma)= \gamma$ and $\epsilon_{\mathcal{G}\mathcal{C}}(\gamma)=\gamma^{\rho}$ for every $\gamma \in \mathbb{R}_+$, every $\rho \in [1,2]$, and with $\mathfrak{V}f=\vert \sigma^{T} Df \vert^2$. Since $\mathfrak{B}(\phi)$ (see (\ref{hyp:controle_coefficients_Talay})) holds, then $g \leqslant CV^{ \rho a/2} $ and it follows that $\mathcal{GC}_{Q,q}(F_q,V^{p \vee 1 +a-1},\tilde{\rho}_q,\epsilon_{\mathscr{X}},\epsilon_{\mathcal{G}\mathcal{C}},\mathfrak{V})$ (see (\ref{hyp:incr_X_Lyapunov_vitesse})) is satisfied. \\
 
 Observing that $\mathcal{S}\mathcal{W}_{\mathcal{G}\mathcal{C}, \gamma}( \tilde{\rho}_q , \gamma,\gamma)$  (see (\ref{hyp:step_weight_I_gen_chow_rate_sans_g})) holds, the proof of Theorem \ref{th:cv_was_Talay} point \ref{th:cv_was_Talay_point_B} is thus a direct consequence of Theorem \ref{th:conv_gnl_ordre_q} taking $q=1$ and $q=2$.

\bibliography{Biblio}
\bibliographystyle{plain}

\end{document}